\documentclass[a4paper,10pt,english,leqno]{amsart}

\def\ocirc#1{\ifmmode\setbox0=\hbox{$#1$}\dimen0=\ht0
    \advance\dimen0 by1pt\rlap{\hbox to\wd0{\hss\raise\dimen0
    \hbox{\hskip.2em$\scriptscriptstyle\circ$}\hss}}#1\else
    {\accent"17 #1}\fi}

\usepackage{color}
\usepackage{dsfont}
\usepackage[colorlinks = false]{hyperref}
\usepackage{mathrsfs}
\usepackage{mathtools}

\numberwithin{equation}{section}
\allowdisplaybreaks[1]

\usepackage{environ}
\makeatletter
\NewEnviron{Lalign}{\tagsleft@true\begin{align}\BODY\end{align}}
\makeatother

\usepackage[shortcuts,cyremdash]{extdash}

\newtheorem{theorem}{Theorem}[section]
\newtheorem{lemma}[theorem]{Lemma}

\newtheorem{proposition}[theorem]{Proposition}
\theoremstyle{definition}\newtheorem{definition}{Definition}[section]
\theoremstyle{remark}\newtheorem{remark}{Remark}[section]

\newtheorem{limit}{Limit}

\newcommand{\marginlabel}[1]{\mbox{}\marginpar{\raggedleft\hspace{0pt}\tiny{\textcolor{red}{#1}}}}
\renewcommand{\marginlabel}[1]{} 

\newcommand{\R}{\mathbb{R}}

\newcommand{\N}{\mathbb{N}}
\newcommand{\RKHS}{{L^2(Z)}}

\newcommand{\dt}{\,dt}
\newcommand{\dx}{\,dx}

\newcommand{\ds}{\,ds}
\newcommand{\dz}{\,dz}

\newcommand{\dX}{\,dX}
\newcommand{\dxdt}{\,dxdt}
\newcommand{\ue}{u^\varepsilon}
\newcommand{\F}{\mathscr{F}}
\newcommand{\Sm}{\mathcal{S}}
\newcommand{\Lin}{\mathscr{L}}

\newcommand{\Pred}{\mathscr{P}}
\newcommand{\test}{\varphi}
\newcommand{\Sd}{S_\delta}
\newcommand{\D}{\mathbb{D}}

\newcommand{\Borel}[1]{\mathscr{B}\left(#1\right)}
\newcommand{\Young}[1]{\mathcal{YM}\left(#1\right)}
\newcommand{\Y}[1]{\mathcal{Y}\left(#1\right)}
\newcommand{\Lebesgue}{\mathcal{L}}
\newcommand{\Zcal}{\mathscr{Z}}

\newcommand{\HS}{\textrm{HS}}

\newcommand{\Rad}[1]{\mathcal{M}(#1)}
\newcommand{\Entropy}{\mathbf{Ent}}

\newcommand{\Jac}{\partial}

\newcommand{\seq}[1]{\left\{#1\right\}}
\newcommand{\abs}[1]{\left|#1\right|}
\newcommand{\sign}[1]{\mathrm{sign}\left(#1\right)}
\newcommand{\signd}[1]{\mathrm{sign}^\prime\left(#1\right)}
\newcommand{\signp}[1]{\mathrm{sign}_+\left(#1\right)}
\newcommand{\signpd}[1]{\mathrm{sign}_+^\prime\left(#1\right)}
\newcommand{\norm}[1]{\left\|#1\right\|}
\newcommand{\Eb}[1]{E\Bigl[#1\Bigr]}
\newcommand{\E}[1]{E\left[#1\right]}

\newcommand{\car}[1]{\mathds{1}_{#1}}
\newcommand{\inner}[2]{\left\langle #1,#2 \right\rangle}

\DeclareMathOperator*{\esssup}{ess\,sup}

\title[Stochastic conservation laws and Malliavin calculus]
{On stochastic conservation laws\\ and Malliavin calculus}

\author[K. H. Karlsen]{K. H. Karlsen}
\address[Kenneth Hvistendahl Karlsen]
{\newline Department of mathematics
\newline University of Oslo
\newline P.O. Box 1053,  Blindern
\newline N--0316 Oslo, Norway} 
\email[]{kennethk@math.uio.no}

\author[E. B. Storr\o{}sten]{E. B. Storr\o{}sten}
\address[Erlend Briseid Storr\o{}sten]
{\newline Department of mathematics
\newline University of Oslo
\newline P.O. Box 1053, Blindern
\newline N--0316 Oslo, Norway} 
\email[]{erlenbs@math.uio.no}

\date{\today}



\thanks{This work was partially supported by the Research Council of Norway 
through the project Stochastic Conservation Laws.}

\begin{document}

\begin{abstract}
For stochastic conservation laws 
driven by a semilinear noise term, we propose a 
generalization of the Kru\v{z}kov entropy condition 
by allowing the Kru\v{z}kov constants to be Malliavin 
differentiable random variables. Existence and 
uniqueness results are provided.  
Our approach sheds some new light on 
the stochastic entropy conditions put forth 
by Feng and Nualart \cite{FengNualart2008} 
and Bauzet, Vallet, and Wittbold \cite{Bauzet:2012kx},
and in our view simplifies some of the proofs.
\end{abstract}

\maketitle

\tableofcontents

\section{Introduction}\label{sec:Introduction}
Stochastic partial differential equations (stochastic PDEs) arise in 
many fields, such as biology, physics, engineering, and economics, in which 
random phenomena play a crucial role. Complex systems always 
contain some element of uncertainty.  Uncertainty may arise in the system 
parameters, initial and boundary conditions, and external forcing processes. 
Moreover, in many situations there is incomplete or partial understanding of the 
governing physical laws, and many models are therefore best 
formulated using stochastic PDEs. 

Recently there has been an interest in studying the effect of 
stochastic forcing on nonlinear conservation laws 
\cite{Bauzet:2012kx,Biswas:2014gd,ChenKarlsen2012,
DebusscheVovelle2010,Kim2003,HR1997,
FengNualart2008,Vallet:2009uq,Vallet:2000ys}, with particular 
emphasis on existence and uniqueness questions (well-posedness). 
Deterministic conservation laws possess 
discontinuous (shock) solutions, and a 
weak formulation coupled with an appropriate entropy condition 
is required to ensure the well-posedness \cite{Kruzkov1970,Malek1996}. 
The question of well-posedness gets somewhat more difficult 
by adding a stochastic source term, due to the 
interaction between noise and nonlinearity. 

In a different direction, we also mention the recent 
works \cite{Lions:2013aa,Lions:2014aa} by 
Lions, Perthame, and Souganidis on conservation 
laws with (rough) stochastic fluxes.

To be more precise, we are interested in stochastic 
conservation laws driven by Gaussian noise 
of the following form:
\begin{equation}\label{eq:StochasticBalanceLaw}
 \left\{
  \begin{aligned}
  du(t,x) + \nabla \cdot f(u(t,x))dt & = 
  \int_Z \sigma(x,u(t,x),z)W(dt,dz), &(t,x) \in \Pi_T, \\
  u(0,x) &= u^0(x),  & x \in \R^d,
  \end{aligned}
 \right.
\end{equation}
where $\Pi_T = (0,T) \times \R^d$, $T > 0$ is some fixed 
final time, $u^0 = u^0(x,\omega)$ is a given $\F_0$\=/measurable 
random function, and the unknown $u = u(t,x,\omega)$ is a 
random (scalar) function. The flux function
\begin{equation}\label{assumption:LipOnf}
 f:\R \rightarrow \R^d \mbox{ is assumed to 
 be $C^1$ and globally Lipschitz.} \tag{$\mathcal{A}_{f}$}
\end{equation}

Concerning the source term, we let $(Z,\Zcal,\mu)$ be a $\sigma$-finite 
separable measure space and $W$ be a space-time Gaussian white noise 
martingale random measure with respect to a filtration 
$\seq{\F_t}_{0 \leq t \leq T}$ \cite{Walsh1984}. Its covariance 
measure is given by $dt \otimes d\mu$, where $dt$ denotes Lebesgue 
measure on $[0,T]$, that is, for $A,B \in \Zcal$,
\begin{displaymath}
 \E{W(t,A)W(t,B)} = t\mu(A \cap B).
\end{displaymath}

The noise coefficient $\sigma:\R^d \times \R \times Z \rightarrow \R$ is 
a measurable function satisfying
\begin{equation}\label{assumption:LipOnSigma}
    \mbox{$\exists M \in L^2(Z)$ s.t.} 
    \begin{cases}
      \abs{\sigma(x,u,z)-\sigma(x,v,z)} \leq \abs{u-v}M(z),\\
      \abs{\sigma(x,u,z)} \leq M(z)(1 + \abs{u}),
    \end{cases}
\tag{$\mathcal{A}_{\sigma}$}
\end{equation}
for all $(x,z) \in \R^d \times Z$. Note that $W$ induces a cylindrical 
Wiener process (with identity covariance operator) 
on $\RKHS = L^2(Z,\Zcal,\mu)$ which we also denote 
by $W$ \cite[\S~7.1.2]{PeszatZabczyk2007}. 
For a nonnegative $\phi \in C(\R^d) \cap L^1(\R^d)$, let $L^2(\R^d,\phi)$ 
denote the $\phi$-weighted $L^2$-space, cf.~Section~\ref{sec:Entropy_Formulation}.
Define $G(u):\RKHS \rightarrow L^2(\R^d,\phi)$ by
\begin{equation}\label{eq:GDef}
 G(u)h(x) = \int_Z \sigma(x,u(x),z)h(z)\,d\mu(z).
\end{equation}
The collection of all Hilbert-Schmidt operators from 
$\RKHS$ into $L^2(\R^d,\phi)$ is denoted 
by $\Lin_2(\RKHS;L^2(\R^d,\phi))$. 
Due to \eqref{assumption:LipOnSigma}, $G$ is a Lipschitz map 
from $L^2(\R^d,\phi)$ into $\Lin_2(\RKHS;L^2(\R^d,\phi))$, i.e., 
\begin{displaymath}
 \norm{G(u)-G(v)}_{\Lin_2(\RKHS;L^2(\R^d,\phi))} 
 \leq \norm{M}_{L^2(Z)}\norm{u-v}_{L^2(\R^d,\phi)}.
\end{displaymath}
In the above setting, \eqref{eq:StochasticBalanceLaw} 
may be written as 
\begin{displaymath}
 du + \nabla \cdot f(u)dt = G(u)dW(t),
\end{displaymath}
where the right-hand side is interpreted with 
respect to the cylindrical Wiener process \cite{DebusscheVovelle2010}. 
In what follows, we will in general stick to the $\sigma$ notation. 
We refer to \cite{DalangSardanyons2011} for a comparison of the 
different stochastic integrals. 

The Malliavin calculus used later 
is developed with respect to the 
isonormal Gaussian process $W:H \rightarrow L^2(\Omega)$ defined by
\begin{equation}\label{eq:IsoGausProcessDef}
	W(h) = \int_0^T\int_Z h(s,z)W(ds,dz) = \int_0^T h(s)dW(s),
\end{equation}
where $H$ denotes the space $L^2([0,T] \times Z,\Borel{[0,T]} 
\otimes \Zcal,dt \otimes d\mu)$. Concerning the notation and 
basic theory of Malliavin calculus we 
refer to \cite{NualartMalliavinCalc2006}.

When the noise term in \eqref{eq:StochasticBalanceLaw} 
is additive ($\sigma$ is independent of $u$), 
Kim \cite{Kim2003} used Kru{\v{z}}kov's entropy condition 
and proved the well-posedness of entropy solutions, see 
also Vallet and Wittbold \cite{Vallet:2009uq}. 
When the noise term is additive, a change of variable 
turns \eqref{eq:StochasticBalanceLaw} into a conservation law with random 
flux function and well-known ``deterministic" techniques apply.

When the noise is multiplicative (i.e., $\sigma$ depends on $u$), a 
simple adaptation of Kru{\v{z}}kov's 
techniques fails to capture a specific ``noise-noise" 
interaction term correlating two entropy solutions, and as a consequence 
they do not lead to the $L^1$-contraction principle. 
This issue was resolved by Feng and Nualart \cite{FengNualart2008}, introducing 
an additional condition capturing the missing noise-noise interaction. 
These authors employ the Kru{\v{z}}kov entropy condition (on It\^{o} form)
\begin{equation}\label{eq:kruz-intro}
	\begin{split}
		& \partial_t \abs{u-c} +  \partial_x \left[ \sign{u-c}(f(u)-f(c))\right] 
		\\ & \qquad \le \frac12 \signd{u-c}\sigma(u)^2+\sign{u-c}\sigma(u)\, \partial_t W, 
		\qquad \forall c \in \R,
	\end{split}
\end{equation}
which is understood in the distributional 
sense (and via an approximation of $\sign{\cdot}$). 
Here, for the sake of simplicity, we take $W$ to be an 
ordinary Brownian motion ($Z$ is a point) and $d=1$.  
The above family of inequalities, indexed over the ``Kru{\v{z}}kov" constants $c$, 
is in \cite{FengNualart2008} augmented with an additional condition 
related to certain substitution formulas \cite[\S~3.2.4]{NualartMalliavinCalc2006}, allowing 
the authors to recover the above mentioned interaction 
term and thus provide, for the first time, a general uniqueness result for stochastic 
conservation laws.

The additional condition proposed in \cite{FengNualart2008} is rather technical 
and difficult to comprehend at first glance. Furthermore, the existence proof 
(passing to the limit in a sequence vanishing viscosity approximations) 
becomes increasingly difficult, with several added arguments 
revolving around fractional Sobolev spaces, 
estimates of the moments of increments, modulus 
of continuity of It\^{o} processes, 
and the Garcia-Rodemich-Rumsey lemma. 

Recently, Bauzet, Vallet, and Wittbold \cite{Bauzet:2012kx} provided a 
framework that uses the Kru{\v{z}}kov entropy 
inequalities \eqref{eq:kruz-intro} but bypasses the Feng-Nualart condition. 
Rather than comparing two entropy solutions 
directly, their uniqueness result compares 
the entropy solution against the vanishing viscosity solution, which is 
generated as the weak limit (as captured by the Young measure) of a 
sequence of solutions to stochastic parabolic equations 
with vanishing viscosity parameter. Although with this approach 
the existence proof becomes simple, many technical difficulties 
are added to the uniqueness proof.  At this point, let us mention that 
Debussche and Vovelle \cite{DebusscheVovelle2010} 
have provided an alternative well-posedness theory based 
on a kinetic formulation. The kinetic formulation 
avoids some of the difficulties alluded to above, thanks to the so-called 
entropy defect measure.

The purpose of our work is to propose a slight modification of 
the Kru{\v{z}}kov entropy condition \eqref{eq:kruz-intro} that will shed some new 
light on \cite{FengNualart2008}, and also \cite{Bauzet:2012kx}.
To this end, we recall that the uniqueness proof 
for entropy solutions is based on a technique 
known as ``doubling of variables''. Suppose 
that $v$ is another entropy solution of \eqref{eq:StochasticBalanceLaw} with 
initial condition $v^0$. The key idea is to consider $v$ as 
a function of a different set of variables, say $v = v(s,y)$, and then 
for each fixed $(s,y) \in \Pi_T$, take $c = v(s,y)$ in the 
entropy condition for $u$. In the case that $u$ and $v$ are 
stochastic fields, $v(s,y)$ is no longer a constant, but rather a random variable. 
Hence it seems natural to utilize an entropy condition 
in which the Kru\v{z}kov parameters $c$ in \eqref{eq:kruz-intro} 
are random variables rather than constants.

Let us do an informal derivation of an entropy condition based on this idea. 
As above, we let $W$ be an ordinary Brownian motion and $d = 1$. 
For each fixed $\varepsilon>0$, suppose $\ue$ is a sufficiently regular 
solution of the stochastic parabolic equation
\begin{equation*}
	\partial_t\ue + \partial_xf(\ue) 
	= \sigma(\ue)\partial_tW + \varepsilon\partial_x^2\ue,
\end{equation*}
where the time derivative is understood in the sense of distributions. 
We apply the \textit{anticipating} It\^o 
formula (Theorem~\ref{theorem:AntIto}) to $\abs{\ue-V}$, 
with $V$ being an arbitrary \textit{Malliavin} 
differentiable random variable. 
Taking expectations, we obtain
\begin{align*}
	&\Eb{\partial_t \abs{\ue-V} + \partial_x(\sign{\ue-V}(f(\ue)-f(V)))} \\
	&\hphantom{XXXX}+\Eb{\signd{\ue-V}\sigma(\ue)D_tV}
	-\frac{1}{2}\Eb{\signd{\ue-V}\sigma^2(\ue)} \\
	&\hphantom{XXXXXX}
	= \varepsilon\Eb{\sign{\ue-V}\partial_x^2\ue},
\end{align*}
where $D_tV$ is the Malliavin derivative of $V$ at $t$. As
$$
 \varepsilon\Eb{\sign{\ue-V}\partial_x^2\ue} 
 \le \varepsilon\Eb{\partial_x^2 \abs{\ue-V}},
$$
it follows that
\begin{align*}
	&\Eb{\partial_t \abs{\ue-V} + \partial_x(\sign{\ue-V}(f(\ue)-f(V)))} \\
	&\hphantom{XXXX}+\Eb{\signd{\ue-V}\sigma(\ue)D_tV}
	-\frac{1}{2}\Eb{\signd{\ue-V}\sigma^2(\ue)} \\
	&\hphantom{XXXXXX}\leq \varepsilon\Eb{\partial_x^2 \abs{\ue-V}}.
\end{align*}
Suppose $\ue \rightarrow u$ in a suitable sense 
as $\varepsilon \downarrow 0$. Then the limit 
$u$ ought to satisfy
\begin{equation}\label{eq:kruz-intro2}
	\begin{split}
		&\Eb{\partial_t \abs{u-V} + \partial_x(\sign{u-V}(f(u)-f(V)))} \\
		& \hphantom{XXXX}+\Eb{\signd{u-V}\sigma(u)D_tV}
		-\frac{1}{2}\Eb{\signd{u-V}\sigma^2(u)} \leq 0,
	\end{split}
\end{equation}
which is the entropy condition that we 
propose should  replace \eqref{eq:kruz-intro}. 

At least informally, it is easy to see why this entropy condition implies the 
$L^1$ contraction principle. Let $u = u(t,x)$ and $v = v(s,y)$ be 
two solutions satisfying the entropy condition \eqref{eq:kruz-intro2}. 
Suppose $u,v$ are both Malliavin differentiable 
and spatially regular. The entropy condition for $u$ yields
\begin{multline*}
	\Eb{\partial_t \abs{u-v} + \partial_x(\sign{u-v}(f(u)-f(v)))} \\
	+\Eb{\signd{u-v}\sigma(u)D_tv} - \frac{1}{2}\Eb{\signd{u-v}\sigma^2(u)} \leq 0.
\end{multline*}
Similarly, the entropy condition of $v$ yields 
\begin{multline*}
 \Eb{\partial_s \abs{v-u} + \partial_y(\sign{v-u}(f(v)-f(u)))} \\
  +\Eb{\signd{v-u}\sigma(v)D_su} - \frac{1}{2}\Eb{\signd{v-u}\sigma^2(v)} \leq 0.
\end{multline*}
Suppose that $t > s$. Then $D_tv(s) = 0$ as $v$ is adapted (to the underlying filtration). 
Adding the last two equations we obtain
\begin{multline*}
	\Eb{(\partial_t + \partial_s)\abs{u-v} 
	+ (\partial_x + \partial_y)(\sign{u-v}(f(u)-f(v)))} \\
	+\Eb{\signd{u-v}\sigma(v)D_su} 
	- \frac{1}{2}\Eb{\signd{u-v}(\sigma^2(u) + \sigma^2(v))} \leq 0.
\end{multline*}
Completing the square yields
\begin{multline}\label{eq:OutlineOfDoubeling}
	\Eb{(\partial_t + \partial_s)\abs{u-v} + (\partial_x + \partial_y)(\sign{u-v}(f(u)-f(v)))} \\
	+\Eb{\signd{u-v}\sigma(v)(D_su-\sigma(u))} 
	-\underbrace{\frac{1}{2}\Eb{\signd{u-v}(\sigma(u) - \sigma(v))^2}}_{=0} \leq 0.
\end{multline}
Next we write
$$
D_su(t)-\sigma(u(t)) = (D_su(t)-\sigma(u(s))) + (\sigma(u(s))-\sigma(u(t))),
$$
and attempt to send $t \downarrow s$. 
The second term tends to zero almost everywhere. 
Formally, for fixed $s$, we observe that $D_su(t)$ satisfies the 
initial value problem

\begin{equation}\label{eq:transport}
\left\{
  \begin{aligned}
  dw + \partial_x\bigl(f'(u)w\bigr) \,dt &
  = \bigl(\sigma'(u)w\bigr)\,dW(t), & (t > s),  \\
  w(s) &= \sigma(u(s)).
  \end{aligned}
\right.
\end{equation}
And so, concluding that 
$$
D_su(t) \rightarrow \sigma(u(s)), \quad \text{as $t\downarrow s$},
$$ 
amounts to showing that the solution to \eqref{eq:transport} satisfies the initial 
condition (in some weak sense). 
Given this result, the $L^1$ contraction property follows 
from \eqref{eq:OutlineOfDoubeling} in a standard way: 
\begin{equation}\label{eq:Contraction}
\frac{d}{dt} \E{\norm{u(t)-v(t)}_{L^1(\R)}} \leq 0.
\end{equation}

The above argument is hampered 
by an obstacle; namely, that the Malliavin 
differentiability of an entropy solution 
seems hard to establish. This can be seen related to 
the \textit{discontinuous} coefficient $f'(u)$ in the stochastic continuity 
equation \eqref{eq:transport}, making it difficult to establish 
the existence of a (properly defined) weak solution.  
In the deterministic context, continuity (and related transport) equations 
with low-regularity coefficients have been an active area 
of research, see for example 
\cite{Ambrosio:2014aa,Bouchut:1998rt,Bouchut:2005aa} (and 
also \cite{Flandoli:2010yq} for a particular stochastic setting). 
Continuity equations arise in many applications, such as 
fluid mechanics. They also appear naturally when linearizing  
a nonlinear conservation law $u_t+f(u)_x=0$ into $w_t + (f'(u)w)_x=0$, 
see \cite{Ambrosio:2014aa,Bouchut:1998rt,Bouchut:2005aa}. 
The present work shows that stochastic continuity equations arise naturally 
as well, through linearisation (by the Malliavin derivative) of stochastic 
conservation laws driven by semilinear noise. However, the study 
of such equations is beyond the present paper and is left for future work.

As alluded to above, to make the $L^1$ contraction argument 
rigorous we would need to know that at least one of the 
two entropy solutions being compared, is Malliavin differentiable.   
To avoid this nontrivial issue, we shall employ a more 
indirect approach, motivated by \cite{Bauzet:2012kx}, comparing one 
entropy solution against the solution of the viscous problem 
linked to the other entropy solution, relying on weak compactness in the 
space of Young measures for the viscous approximation. 
The Malliavin differentiability of the viscous solution is then 
established and its Malliavin derivative is shown to satisfy a linear 
stochastic parabolic equation, with an initial condition fulfilled 
in the weak sense. Given these results, the proof of the $L^1$ contraction 
property follows as outlined above. 

Finally, we mention that the approach developed herein appears 
useful in the study of error estimates for numerical approximations 
of stochastic conservation laws,  whenever the 
approximation is Malliavin differentiable. 
It seems to us that this Malliavin differentiability 
is indeed often available. Furthermore, the approach 
may be extended so as to cover strongly 
degenerate parabolic equations with L\'{e}vy noise, cf.~\cite{Biswas:2014gd}, 
\cite{NunnoProskeOksendal2009}. It also constitutes a starting point 
for developing a well-posedness theory for stochastic conservation laws with 
random, possibly anticipating initial data. 
Note however that this seems to depend on the Malliavin differentiability of the entropy 
solution (Lemma~\ref{lemma:EntIneqViscForAdapted} is no longer applicable).

The remaining part of the paper is organized as follows: 
We present the solution framework and gather some 
preliminary results in Section \ref{sec:Entropy_Formulation}. Well-posedness results 
for the viscous approximations are provided in Section \ref{sec:ViscousApprox}. 
Furthermore, we establish the Malliavin differentiability 
of these approximations and show that the Malliavin derivative 
can be cast as the solution of a linear stochastic parabolic equation. 
The question of (weak) satisfaction of the initial condition is addressed.
Sections~\ref{sec:Existence} and~\ref{sec:Uniqueness} supply 
detailed proofs for the existence and uniqueness of Young measure-valued 
entropy solutions. Finally, some basic results are collected in Section~\ref{sec:Appendix}.

\section{Entropy solutions}\label{sec:Entropy_Formulation}
Under the assumption $\sigma(x,0,z) = 0$, the ordinary $L^p$ spaces 
($2\leq p < \infty$) constitute a natural choice for \eqref{eq:StochasticBalanceLaw}. 
Without this assumption, a certain class of weighted 
$L^p$ spaces seem to be better suited. For non-negative $\phi$ we define 
\begin{displaymath}
 	\norm{u}_{p,\phi} := \left(\int_{\R^d} \abs{u(x)}^p \phi(x)\,dx\right)^{1/p}.
\end{displaymath}
The relevant weights, denoted by $\mathfrak{N}$, 
consist of non-zero $\phi \in C^1(\R^d) \cap L^1(\R^d)$ for which 
there is a constant $C_\phi$ such that $\abs{\nabla\phi(x)} \leq C_\phi \phi(x)$. 
The weighted $L^p$-space associated with $\phi$ is denoted by $L^p(\R^d,\phi)$. 

To see that $\mathfrak{N}$ is non-empty, consider 
$\phi_N(x) = (1 + \abs{x}^2)^{-N}$ for $N \in \N$. 
Then we claim that $\phi_N \in \mathfrak{N}$ for all $N \geq d$. 
To this end, observe that  
\begin{displaymath}
 \nabla \phi_N(x) = -2N\frac{x}{1+\abs{x}^2}\phi_N(x),
\end{displaymath}
so $\abs{\nabla \phi_N(x)} \leq 2N \phi_N(x)$. 
Furthermore,
\begin{equation*}
\begin{split}
\int_{\R^d} \phi_N(x) \,dx 
= \int_0^\infty\int_{\partial B(0,r)}\left(\frac{1}{1+r^2}\right)^N \,dS dr 
< \infty.
\end{split}
\end{equation*} 
Another family of functions in $\mathfrak{N}$ is $\phi_\lambda(x) 
= \exp(-\lambda \sqrt{1 + \abs{x}^2})$, for $\lambda > 0$ \cite{VolpertHudjaev1969}. 

The fact that $\phi \in L^1(\R^d)$ 
yields $L^q(\R^d,\phi) \subset L^p(\R^d,\phi)$ for all $1 \leq p < q<\infty$. 
Indeed, $\norm{u}_{p,\phi} \leq \norm{u}_{q,\phi}\norm{\phi}_{L^1(\R^d)}^{1/p-1/q}$.
We shall also make use of the weighted $L^\infty$-norm
\begin{displaymath}
	\norm{h}_{\infty,\phi^{-1}} := \sup_{x \in \R^d}
	\left\{\frac{\abs{h(x)}}{\phi(x)}\right\}, \qquad h \in C(\R^d).
\end{displaymath}
Note that any compactly supported $h \in C(\R^d)$ is bounded 
in this norm, for $\phi \in \mathfrak{N}$. The norm is convenient 
due to the inequality 
$\norm{u}_{p,h} \leq \norm{u}_{p,\phi}\norm{h}_{\infty,\phi^{-1}}$.

Denote by $\mathscr{E}$ the set of non-negative 
convex functions in $C^2(\R)$ with $S(0) = 0$, $S'$ 
bounded, and $S''$ compactly supported. 
Suppose $Q:\R^2 \rightarrow \R^d$ satisfies  
\begin{displaymath}
 \Jac_1 Q(u,c) = S'(u-c)f'(u), \qquad Q(c,c) = 0, \qquad u,c \in \R,
\end{displaymath}
where $S \in \mathscr{E}$. Then we call $(S(\cdot-c),Q(\cdot,c))$ 
an \emph{entropy/entropy-flux pair} (indexed over $c\in \R$). 
For short, we say that $(S,Q)$ is 
in $\mathscr{E}$ if $S$ is in $\mathscr{E}$. 

We denote by $\D^{1,2}$ the space 
of Malliavin differentiable random variables 
in $L^2(\Omega)$ with Malliavin derivative in 
$L^2(\Omega;L^2([0,T] \times Z))$ \cite[p.~27]{NualartMalliavinCalc2006}. 

For $(S,Q) \in \mathscr{E}$, $\test \in C^\infty_c([0,T) 
\times \R^d)$, and $V \in \D^{1,2}$,  we define the functional
\begin{equation*}
  \begin{split}    
      &  \Entropy[(S,Q),\test,V](u) := \E{\int_{\R^d} S(u^0(x)-V)\test(0,x) \dx}\\
      &  \qquad + \E{\iint_{\Pi_T} S(u-V)\partial_t\test + Q(u,V)\cdot \nabla \test \dxdt} \\
      &  \qquad - \E{\iint_{\Pi_T}\int_Z S''(u-V) \sigma(x,u,z)D_{t,z}V \test \,d\mu(z)\dxdt } \\
      &  \qquad +\frac{1}{2}\E{ \iint_{\Pi_T}\int_Z S''(u-V)\sigma(x,u,z)^2\test\,d\mu(z)\dxdt},
  \end{split}
\end{equation*}
where $D_{t,z}V$ is the Malliavin derivative 
of $V$ at $(t,z) \in [0,T] \times Z$.  

We claim that $\Entropy$ is well-defined whenever $V \in \D^{1,2}$, 
$u \in L^2([0,T] \times \Omega;L^2(\R^d,\phi))$, 
$u^0 \in L^2(\Omega;L^2(\R^d,\phi))$, and 
\begin{displaymath}
\norm{\test(t)}_{\infty,\phi^{-1}},
 \norm{\partial_t\test(t)}_{\infty,\phi^{-1}},
 \norm{\nabla\test(t)}_{\infty,\phi^{-1}}
\quad \text{are bounded on $[0,T]$};
\end{displaymath}
note that any $\test \in C^\infty_c([0,T) \times \R^d)$ meets these criteria. 
To this end, observe that the first three terms are bounded 
due to the Lipschitz condition on $S$. Indeed,
\begin{equation}\label{eq:LipEstOnEntropyFlux}
 \abs{Q(u,V)} = \abs{\int_V^u S'(z-V)\Jac f(z)\,dz} 
 \leq \norm{S}_{\mathrm{Lip}}\norm{f}_{\mathrm{Lip}}\abs{u-V},
\end{equation}
and so
\begin{equation}\label{eq:EstOnFluxTermEntIneq}
 \begin{split}
 &\abs{\E{\iint_{\Pi_T}Q(u,V)\cdot \nabla \test \dxdt}} 
  \leq \norm{S}_{\mathrm{Lip}}\norm{f}_{\mathrm{Lip}}
  \E{\iint_{\Pi_T}(\abs{u}+\abs{V})\abs{\nabla \test} \dxdt} \\
  &\hphantom{X}\leq \norm{S}_{\mathrm{Lip}}\norm{f}_{\mathrm{Lip}}
  \int_0^T \left(\E{\norm{u(t)}_{1,\phi}}\norm{\nabla \test(t)}_{\infty,\phi^{-1}}
  +\E{\abs{V}}\norm{\nabla \test(t)}_{L^1(\R^d)}\right)dt,
 \end{split}
\end{equation}
which is finite. The terms involving $\sigma$ are easily 
seen to be well-defined since the Hilbert Schmidt norm 
of $G(u)$ (cf.~\eqref{eq:GDef}) is bounded. 
To simplify the notation, we set $HS = \Lin_2(L^2(Z);L^2(\R^d,\phi))$. 
Due to assumption~\eqref{assumption:LipOnSigma},
\begin{align*}
 \norm{G(u)}_{HS}^2 &= \int_{\R^d}\int_Z \sigma^2(x,u(x),z)\phi(x)\,d\mu(z)\,dx \\
&\leq 2\norm{M}_{L^2(Z)}^2\int_{\R^d}(1 + \abs{u(x)}^2)\phi(x)\,dx.
\end{align*}
Boundedness of the last term follows as
\begin{multline}\label{eq:EstOnSquareTermEntIneq}
 \abs{\E{ \iint_{\Pi_T}\int_Z S''(u-V)\sigma(x,u,z)^2\test\,d\mu(z)\dxdt}} \\
  \leq \norm{S''}_\infty \E{\int_0^T
  \norm{\test(t)}_{\infty,\phi^{-1}}\norm{G(u(t))}_{\HS}^2 \,dt}.
\end{multline}
By the sub-multiplicativity of the Hilbert Schmidt norm and 
H\"older's inequality, it follows that
\begin{equation}\label{eq:EstOnMallTermEntIneq}
  \begin{split}
   &\abs{\E{\iint_{\Pi_T}\int_Z S''(u-V) \sigma(x,u,z)D_{t,z}V \varphi \,d\mu(z)\dxdt }} \\
   & \quad \leq \norm{S''}_\infty \E{\iint_{\Pi_T}
   \norm{\test(t)}_{\infty,\phi^{-1}} \abs{\int_Z\sigma(x,u,z)D_{t,z}V \,d\mu(z)}\phi\dxdt} \\
   & \quad \leq \norm{S''}_\infty\norm{\phi}_{L^1(\R^d)}^{1/2}
   \E{\int_0^T\norm{\test(t)}_{\infty,\phi^{-1}} \norm{G(u(t))D_tV}_{2,\phi}dt}  \\
   & \quad \leq \norm{S''}_\infty\norm{\phi}_{L^1(\R^d)}^{1/2}
   \E{\int_0^T \norm{\test(t)}_{\infty,\phi^{-1}}^2
   \norm{G(u(t))}_{HS}^2\,dt}^{1/2}
   \\ & \qquad \qquad\qquad 
   \times \norm{DV}_{L^2(\Omega;L^2([0,T] \times Z))}<\infty. 
  \end{split}
\end{equation} 

Let $\Pred$ denote the predictable $\sigma$-algebra 
on $[0,T] \times \Omega$ with respect to $\seq{\F_t}$ \cite[\S~2.2]{ChungWilliams2014}. 
In general we are working with equivalence classes of functions with respect 
to the measure $dt \otimes dP$. The equivalence class $u$ is said to be \emph{predictable} 
if it has a version $\tilde{u}$ that is $\Pred$-measurable. In some of the arguments, to 
avoid picking versions, we consider the completion of $\Pred$ with 
respect to $dt \otimes dP$, denoted by $\Pred^*$. 
We recall that any jointly measurable and adapted process 
is $\Pred^*$-measurable, see \cite[Theorem~3.7]{ChungWilliams2014}.

\begin{definition}[Entropy solution]\label{Def:EntropySolution}
 An entropy solution $u = u(t,x;\omega)$ 
 of \eqref{eq:StochasticBalanceLaw}, with initial 
 condition $u^0 \in L^2(\Omega,\F_0,P;L^2(\R^d,\phi))$, 
 is a function satisfying:
 \begin{itemize}
   \item[(i)] $u$ is a predictable process in 
   $L^2([0,T] \times \Omega;L^2(\R^d,\phi))$.
   \item[(ii)] For any random variable $V\in \D^{1,2}$, 
   any entropy/entropy-flux pair $(S,Q)$ in $\mathscr{E}$, 
   and all nonnegative test functions 
   $\test \in C^\infty_c([0,T) \times \R^d)$,
   \begin{displaymath}
    	\Entropy[(S,Q),\test,V](u) \geq 0.
   \end{displaymath}
  \end{itemize}
\end{definition}
Here $L^2([0,T] \times \Omega;L^2(\R^d,\phi))$ is the 
Lebesgue-Bochner space, see Section~\ref{sec:LebBoch}.

\begin{remark}
One consequence of the upcoming results is that 
the viscous approximations \eqref{eq:ViscousApprox}
converge (strongly) to the entropy solution in the 
sense of Definition~\ref{Def:EntropySolution}. By passing to the limit 
in the weak formulation of \eqref{eq:ViscousApprox}, 
it follows that the entropy solution is also a weak solution. 
At an \textit{informal} level, this is linked to the Malliavin integration by 
parts formula. To see this, let $d = 1$, $W$ be an ordinary Brownian 
motion, and suppose that $u$ is a Malliavin differentiable and 
spatially regular entropy solution. We outline a nonrigorous argument 
showing that $u$ is a (strong) solution 
of \eqref{eq:StochasticBalanceLaw}. Let
\begin{displaymath}
 (u)_+ = \begin{cases}
          u \mbox{ for $u > 0$,} \\
	  0 \mbox{ else,}
         \end{cases}
 \mbox{ and } \quad 
 \signp{u} = \begin{cases}
	      1 \mbox{ for $u > 0$,} \\
	      0 \mbox{ else,}
	     \end{cases}
\end{displaymath}
so that $(u)_+' = \signp{u}$. Suppose for any $A \in \F_T$ 
there is a Malliavin differentiable random variable $V$ satisfying  
\begin{equation*}
 \begin{cases}
	u-V > 0 & \mbox{for $\omega \in A$,} \\
  	u-V < 0 & \mbox{else},
 \end{cases}
\end{equation*} 
so that $\signpd{u-V} = \car{A}$. Let us point out that since random 
variables of the form $\car{A}$ are not Malliavin 
differentiable \cite[Proposition~1.2.6]{NualartMalliavinCalc2006}, this 
argument is in need of an additional approximation step. 
As the argument is already informal, we skip this step.
Let $S(\cdot) = (\cdot)_+$ and 
\begin{displaymath}
Q(u,c) = \signp{u-c}(f(u)-f(c)), \qquad u,c \in \R.  
\end{displaymath}
The entropy inequality yields
\begin{equation*}
	\begin{split}
		&\Eb{\partial_t (u-V)_+ + \partial_x(\signp{u-V}(f(u)-f(V)))} \\
		& \hphantom{XXXX}+\Eb{\signpd{u-V}\sigma(u)D_tV}
		-\frac{1}{2}\Eb{\signpd{u-V}\sigma^2(u)} \leq 0.
	\end{split}
\end{equation*}
Apriori, the trace $t \mapsto D_tu(t)$ is not well-defined. 
However, due to \eqref{eq:transport}, $D_tu(\tau) \rightarrow \sigma(u(t))$ 
as $\tau \downarrow t$ (essentially), while $D_tu(\tau) = 0$ for $\tau < t$, and 
so it is natural to assign the value
\begin{displaymath}
 D_tu(t) = \lim_{\delta \downarrow 0}\frac{1}{2\delta}
 \int_{t-\delta}^{t + \delta} D_tu(\tau) \,d\tau = \frac{1}{2}\sigma(u(t)),
\end{displaymath}
cf.~\cite[p.~173]{NualartMalliavinCalc2006}. By the chain 
rule for Malliavin derivatives, 
$$
D_t\signp{u-V} = \signpd{u-V}\left(\frac{1}{2}\sigma(u)-D_tV\right),
$$
and so
\begin{equation*}
	\Eb{\partial_t (u-V)_+ + \partial_x(\signp{u-V}(f(u)-f(V)))} 
	\leq \Eb{D_t\signp{u-V}\sigma(u)}.
\end{equation*}
The integration by parts formula of Malliavin calculus yields
\begin{displaymath}
 \Eb{D_t\signp{u-V}\sigma(u)} 
 = \Eb{\signp{u-V}\sigma(u)\partial_tW}.
\end{displaymath}
As $\signp{u-V} = \car{A}$, $(u-V)_+ = (u-V)\car{A}$, and 
$A$ is arbitrary, it follows that 
\begin{equation*}
 \partial_tu + \partial_xf(u) \leq \sigma(u)\partial_tW.
\end{equation*}
The reverse inequality follows by considering $S(\cdot) = (\cdot)_-$.
\end{remark}

Let us fix some notation. For $n = 1,2,\ldots$, we 
will denote by $J^n$ a non-negative, smooth function satisfying 
\begin{displaymath}
 \mathrm{supp}(J^n) 
 \subset B(0,1), \, \int_{\R^n} J^n(x) \dx = 1, 
 \mbox{ and } J^n(x) = J^n(-x),
\end{displaymath}
for all $x \in \R^n$. For any $r > 0$ we let 
$J_r^n(x) = \frac{1}{r^n}J^n(\frac{x}{r})$. 
For $n = 1$ we let $J_r^+(x) = J_r(x-r)$ and 
note that $\mathrm{supp}(J_r^+) \subset (0,2r)$. 
As the value of $n$ is understood 
from the context, we will write $J = J^n$. 

According to Theorem~\ref{theorem:ExistenceOfSolution} 
and Theorem~\ref{theorem:UniquenessOfEntSol}, if 
$u^0 \in L^p(\Omega;L^p(\R^d,\phi))$, then the entropy solution 
belongs to $L^p(\Omega \times [0,T];L^p(\R^d,\phi))$ 
for any $2 \leq p < \infty$. As a consequence of the entropy 
inequality we obtain the following:
\begin{proposition}
 Let $2 \leq p < \infty$ and suppose $u^0 \in L^p(\Omega,\F_0,P;L^p(\R^d,\phi))$. 
 If $u \in L^p([0,T] \times \Omega;L^p(\R^d,\phi))$ is an 
 entropy solution of \eqref{eq:StochasticBalanceLaw}, then
 \begin{displaymath}
  \esssup_{0 \leq t \leq T}\seq{\E{\norm{u(t)}_{p,\phi}^p}} < \infty.
 \end{displaymath}
\end{proposition}
\begin{proof}
Set
 \begin{displaymath}
  \varphi_\delta(t,x) = \left(1 
  - \int_0^t J_\delta(\sigma-\tau)\,d\sigma\right)\phi(x).
 \end{displaymath}
Introduce the entropy function
 \begin{equation*}
  S_R(u) =
   \begin{cases}
      R^p + pR^{p-1}(u-R) & \mbox{ for $u \geq R$}, \\
      \abs{u}^p & \mbox{ for $-R < u < R$}, \\
      R^p -pR^{p-1}(u+R) & \mbox{ for $u \leq -R$},
   \end{cases}
 \end{equation*}
 and denote by $Q_R$ the corresponding entropy-flux. 
 Strictly speaking, $S_R$ is not in $\mathscr{E}$, but this can be 
 amended by a simple mollification step (which we ignore). Note that 
 $S_R \rightarrow \abs{\cdot}^p$ pointwise. Furthermore,
 \begin{equation}\label{eq:pNormEstOnEnt}
  \left\{ \begin{split}
	    &\abs{S_R'(u)}  \leq p\abs{u}^{p-1}, \\
	    &\abs{S_R''(u)} \leq p(p-1)\abs{u}^{p-2}, \\
	    &\abs{Q_R(u,c)} \leq \norm{f}_{\mathrm{Lip}}\abs{u-c}^p.
          \end{split}
  \right.
 \end{equation}
 We apply the Lebesgue differentiation and 
 dominated convergence theorems to make appear 
 $\lim_{\delta \downarrow 0}\Entropy[(S_R,Q_R),\test_\delta,0](u) \geq 0$. 
 This yields 
 \begin{equation*}
  \begin{split}    
      \E{\int_{\R^d} S_R(u(\tau))\phi(x)\dx} 
      & \leq \E{\int_{\R^d} S_R(u^0(x))\phi(x) \dx} \\
      &\qquad +\E{\int_0^\tau \int_{\R^d} Q_R(u,0) \cdot \nabla \phi \dxdt} \\
      &\qquad +\frac{1}{2}\E{\int_0^\tau\int_{\R^d}
      \int_Z S_R''(u)\sigma(x,u,z)^2\phi\,d\mu(z)\dxdt}, 
  \end{split}
\end{equation*}
for almost all $\tau \in [0,T]$. Due to \eqref{eq:pNormEstOnEnt} it is 
straightforward to supply estimates, uniform in $R$, of 
the type \eqref{eq:EstOnFluxTermEntIneq} 
and \eqref{eq:EstOnSquareTermEntIneq}. By the dominated 
convergence theorem, we may send $R \rightarrow \infty$. The result follows.
\end{proof}

It is enough to consider smooth random variables in 
Definition~\ref{Def:EntropySolution}, i.e., random variables of the form
\begin{displaymath}
 V = f(W(h_1),\dots,W(h_n))
\end{displaymath}
where $f \in C^\infty_c(\R^n)$, $W$ is the isonormal Gaussian process 
defined by $\eqref{eq:IsoGausProcessDef}$, and $h_1,\dots, h_n$ are in 
$H = L^2([0,T] \times Z)$, see \cite[p.~25]{NualartMalliavinCalc2006}. 
We denote the space of smooth random variables by $\Sm$.

\begin{lemma}\label{lemma:ContinuityOfEntWRTV}
 Suppose \eqref{assumption:LipOnf} and \eqref{assumption:LipOnSigma} are satisfied. 
 Fix $u \in L^2([0,T] \times \Omega;L^2(\R^d,\phi))$, an 
 entropy/entropy-flux pair $(S,Q) \in \mathscr{E}$, 
 and $\test \in C^\infty_c([0,T) \times \R^d)$. Then 
 \begin{displaymath}
  V \mapsto \Entropy[(S,Q),\test,V](u)
 \end{displaymath}
 is continuous on $\D^{1,2}$ (in the strong topology).
\end{lemma}

\begin{remark}
It is not necessary that $S''$ is compactly supported in the upcoming 
proof (it is sufficient with boundedness/continuity).
\end{remark}

\begin{proof}
 Suppose that $V_n \rightarrow V$ in $\D^{1,2}$ as $n \rightarrow \infty$, 
 and write
\begin{align*}
    &  \Entropy[(S,Q),\test,V](u)-\Entropy[(S,Q),\test,V_n](u) \\
    &  \quad = \E{\int_{\R^d} (S(u^0(x)-V)-S(u^0(x)-V_n))\test(0,x) \dx}\\
    &  \qquad + \E{\iint_{\Pi_T} (S(u-V)-S(u-V_n))\partial_t\test \dxdt} \\
    &  \qquad + \E{\iint_{\Pi_T}(Q(u,V)-Q(u,V_n))\cdot \nabla \test \dxdt} \\
    &  \qquad + \E{\iint_{\Pi_T}\int_Z (S''(u-V_n)D_{t,z}V_n-S''(u-V)
    D_{t,z}V)\sigma(x,u,z)\test \,d\mu(z)\dxdt} \\
    &  \qquad +\frac{1}{2}\E{ \iint_{\Pi_T}\int_Z (S''(u-V)-S''(u-V_n))
    \sigma(x,u,z)^2\test\,d\mu(z)\dxdt} \\
    &  \quad =: \mathcal{T}_1 + \mathcal{T}_2 + \mathcal{T}_3 + \mathcal{T}_4 + \mathcal{T}_5.
\end{align*}
We need to show that $\lim_{n \rightarrow \infty}\mathcal{T}_i(n) = 0$ 
for $1 \leq i \leq 5$. 

First, note that $V_n \rightarrow V$ in $L^2(\Omega)$. Next,
\begin{displaymath}
  \abs{\mathcal{T}_1} \leq \norm{S}_{\mathrm{Lip}}\E{\abs{V-V_n}}\norm{\test(0)}_{L^1(\R)}.
\end{displaymath}
 Similarly,
 \begin{displaymath}
  \abs{\mathcal{T}_2} \leq  \norm{S}_{\mathrm{Lip}}\E{\abs{V-V_n}}\norm{\partial_t\test}_{L^1(\Pi_T)}.  
 \end{displaymath}
 It follows as $V_n \rightarrow V$ in $L^2(\Omega)$ that $\mathcal{T}_1,\mathcal{T}_2 \rightarrow 0$ 
 as $n \rightarrow \infty$. 
 
 Concerning $\mathcal{T}_3$, we first observe that for any $\zeta,\xi,\theta \in \R$,
 \begin{align*}
  &\abs{Q(\zeta,\xi)-Q(\zeta,\theta)} = \abs{\int_\xi^\zeta S'(z-\xi)\Jac f(z)\,dz 
  - \int_\theta^\zeta S'(z-\theta)\Jac f(z)\,dz} \\
  & \qquad \leq \abs{\int_\xi^\zeta (S'(z-\xi)-S'(z-\theta))\Jac f(z)\,dz} 
  + \abs{\int_\xi^\theta S'(z-\theta)\Jac f(z)\,dz}.
 \end{align*} 
 Hence,
 \begin{align*}
  &\abs{\mathcal{T}_3} \leq \E{\iint_{\Pi_T}\abs{\int_V^u (S'(z-V)-S'(z-V_n))
  \Jac f(z)\,dz}\abs{\nabla \test} \dxdt} \\
  & \qquad\qquad 
  + \E{\iint_{\Pi_T}\abs{\int_V^{V_n} S'(z-V_n)\Jac f(z)\,dz}
  \abs{\nabla \test} \dxdt} \\
  & \qquad =:\mathcal{T}_3^1 + \mathcal{T}_3^2.
 \end{align*} 
 Consider $\mathcal{T}_3^1$. Note that 
 \begin{displaymath}
  \abs{\int_V^u (S'(z-V)-S'(z-V_n))\Jac f(z)\,dz} \leq 
  \norm{f}_{\mathrm{Lip}}\norm{S''}_\infty\abs{V-V_n}(\abs{u}+ \abs{V}).
 \end{displaymath}
 Due to H\"{o}lder's inequality it follows that 
 \begin{multline*}
  \mathcal{T}_3^1 \leq \norm{V-V_n}_{L^2(\Omega)}
  \norm{f}_{\mathrm{Lip}}\norm{S''}_\infty\norm{\nabla \test}_{L^1(\Pi_T)}^{1/2} \\
  \times \E{2\iint_{\Pi_T}(\abs{u}^2 + \abs{V}^2)\abs{\nabla \test}\,dxdt}^{1/2}.
 \end{multline*}
 Since 
\begin{displaymath}
  \mathcal{T}_3^2 \leq \norm{S}_{\mathrm{Lip}}\norm{f}_{\mathrm{Lip}}
  \E{\abs{V_n-V}}\norm{\nabla \test}_{L^1(\Pi_T)},
\end{displaymath}
it follows that $\lim_{n \rightarrow \infty}\mathcal{T}_3 = 0$. 

Concerning the $\mathcal{T}_4$-term, we first split it as follows:
\begin{align*}
   \mathcal{T}_4 &= \E{\iint_{\Pi_T}\int_Z S''(u-V_n)(D_{t,z}V_n-D_{t,z}V)\sigma(x,u,z)
   \test \,d\mu(z)\dxdt} \\
   &+ \E{\iint_{\Pi_T}\int_Z (S''(u-V_n)-S''(u-V))D_{t,z}
   V\sigma(x,u,z)\test \,d\mu(z)\dxdt} \\
   &= \mathcal{T}_4^1 + \mathcal{T}_4^2.
 \end{align*}
By \eqref{eq:EstOnMallTermEntIneq}, $\lim_{n \rightarrow \infty}\mathcal{T}_4^1 = 0$. 
Owing to \eqref{eq:EstOnMallTermEntIneq}, the dominated 
convergence theorem implies 
$\lim_{n \rightarrow \infty}\mathcal{T}_4^2 = 0$. 
Finally, by \eqref{eq:EstOnSquareTermEntIneq} and 
the dominated convergence theorem, 
also $\lim_{n \rightarrow \infty}\mathcal{T}_5 = 0$.
\end{proof}

For the existence proof, it will be convenient to introduce 
a weaker notion of entropy solution based on Young 
measures (see, e.g., \cite{Bauzet:2012kx,DiPerna1985,EymardGallouetHerbin1995,Panov1996}). 
The reason beeing the application of Young measures as generalized 
limits in the sense of Theorem~\ref{theorem:YoungMeasureLimitOfComposedFunc}. 
Denote by $\Young{\Pi_T \times \Omega;\R}$ the set of all Young 
measures from $\Pi_T \times \Omega$ into $\R$, 
cf.~Section~\ref{sec:YoungMeasures}. Instead of representing the 
solution/limit as an element in $\Young{\Pi_T \times \Omega;\R}$ we 
use the notion of entropy process proposed in \cite{EymardGallouetHerbin1995} 
or equivalently the strong measure-valued solution proposed in \cite{Panov1996}.  
Any probability measure $\nu$ on the real line may be represented by a 
measurable function $u:[0,1] \rightarrow \R \cup \seq{\infty}$ 
such that $\nu$ is the image of the Lebesgue 
measure $\Lebesgue$ on $[0,1]$ by $u$. 
In fact, we may take (see \cite[\S~2.2.2]{Villani2003})
\begin{equation}\label{eq:ReprOfProcessByYoung}
 u(\alpha) = \inf \left\{ \xi \in \R \,:\, \nu((-\infty,\xi]) > \alpha \right\}.
\end{equation}
A Young measure $\nu \in \Young{\Pi_T \times \Omega;\R}$ is thus 
represented by a (higher dimensional) function $u:\Pi_T  \times [0,1] 
\times \Omega \rightarrow \R$, such that 
$\nu_{t,x,\omega}(B) = \Lebesgue(u(t,x,\cdot,\omega)^{-1}(B))$ 
for any measurable $B \subset \R$. The extension to Young 
measure-valued solutions is obtained through the embedding defined by 
\begin{equation}\label{eq:EmbeddingL2L2times01}
 \Phi(u)(t,x,\alpha,\omega) = 
 u(t,x,\omega).
\end{equation}
Given a functional $F$ we define the extension 
\begin{displaymath}
\Y{F}(u) = \int_0^1 F(u(\alpha))\,d\alpha,
\end{displaymath}
so that $\Y{F} \circ \Phi = F$. For $1 \leq p < \infty$ we let
\begin{displaymath}
 \norm{u}_{p,\phi \otimes 1} = \left(\int_0^1
 \int_{\R^d} \abs{u(x,\alpha)}^p\phi(x)\,dx\,d\alpha\right)^{1/p}.
\end{displaymath}
The associated space is denoted by $L^p(\R^d \times [0,1],\phi)$.

\begin{definition}[Young measure-valued entropy solution]\label{Def:YoungEntropySolution}
  A Young measure-valued entropy solution $u = u(t,x,\alpha;\omega)$ 
  of \eqref{eq:StochasticBalanceLaw}, with initial condition 
  $u^0$ belonging to $L^2(\Omega,\F_0,P;L^2(\R^d,\phi))$, is 
  a function satisfying:
  \begin{itemize}
   \item[(i)] $u$ is a predictable process in 
   $L^2([0,T] \times \Omega;L^2(\R^d \times [0,1],\phi))$.
   \item[(ii)] For any random variable $V\in \D^{1,2}$, 
   any entropy/entropy-flux pair $(S,Q)$ in $\mathscr{E}$, 
   and all nonnegative test 
   functions $\test \in C^\infty_c([0,T) \times \R^d)$, 
   \begin{equation}\label{eq:YoungEntropyCondition}
     \Y{\Entropy[(S,Q),\test,V]}(u) \geq 0.
   \end{equation}
  \end{itemize}
\end{definition}

The next result is concerned with the essential continuity of the 
solutions at $t = 0$. A similar argument can be 
found in \cite{CancesClementGallouet2011}.

\begin{lemma}[Initial condition]\label{lemma:InitialCondition}
 Suppose \eqref{assumption:LipOnf} and \eqref{assumption:LipOnSigma} 
 are satisfied, and that $u^0$ belongs to $L^2(\Omega,\F_0,P;L^2(\R^d,\phi))$. 
 Let $u$ be a Young measure-valued entropy 
 solution of \eqref{eq:StochasticBalanceLaw} 
 in the sense of Definition~\ref{Def:YoungEntropySolution}. 
 Let $S:\R \rightarrow [0,\infty)$ be Lipschitz continuous and 
 satisfy $S(0) = 0$. For any $\psi \in C^\infty_c(\R^d)$,
 \begin{displaymath}
  \mathcal{T}_{r_0} := \E{\iint_{\Pi_T}\int_{[0,1]} 
  S(u(t,x,\alpha)-u^0(x))\psi(x)J_{r_0}^+(t)\,d\alpha dxdt} 
  \rightarrow 0 \mbox{ as } r_0 \downarrow 0.
 \end{displaymath}
\end{lemma}

\begin{remark}
  The proof does not depend on the differentiability of $J_{r_0}^+$. 
  Hence the above limit may be replaced by
 \begin{displaymath}
  \lim_{\tau \downarrow 0} \E{\frac{1}{\tau}\int_0^\tau\int_{\R^d} \int_{[0,1]} 
  S(u(t,x,\alpha)-u^0(x))\psi(x)\,d\alpha dx dt} = 0.
 \end{displaymath}
\end{remark}

\begin{proof}
  Let $S \in C^\infty(\R)$ with bounded derivatives. Take
  \begin{displaymath}
   \test(t,x,y) = \xi_{r_0}(t)\psi(x)J_r(x-y)\, \mbox{ where }\,\xi_{r_0}(t) = 1 - \int_0^t J_{r_0}^+(s)\ds.
  \end{displaymath}
  Then let $V = u^0(y)$ in \eqref{eq:YoungEntropyCondition} and 
  integrate in $y$. This implies
  \begin{equation}\label{eq:IneqTimeContAtZero}
   \begin{split}    
      I &:= \E{\int_{\R^d}\iint_{\Pi_T}\int_{[0,1]} S(u(t,x,\alpha)-u^0(y))\psi(x)
      J_r(x-y)J_{r_0}^+(t) d\alpha\dxdt dy} \\
      & \leq \E{\int_{\R^d}\iint_{\Pi_T}\int_{[0,1]}Q(u(t,x,\alpha),u^0(y))\cdot 
      \nabla_x \test d\alpha\dxdt dy} \\
      & \qquad +\E{\int_{\R^d}\int_{\R^d} S(u^0(x)-u^0(y))\test(0,x,y)\dx dy} \\
      & \qquad +\frac{1}{2} \E{\int_{\R^d}\iint_{\Pi_T}\int_{[0,1]}\int_Z 
      S''(u(t,x,\alpha)-u^0(y))\sigma(x,u,z)^2\test \,d\mu(z) d\alpha\dxdt dy} \\
      & =: \mathcal{T}^1 + \mathcal{T}^2 + \mathcal{T}^3.
   \end{split}
  \end{equation}
Let us first observe that 
\begin{multline*}
 I = \mathcal{T}_{r_0} + E\bigg[\int_{\R^d}\iint_{\Pi_T}
 \int_{[0,1]} (S(u(t,x,\alpha)-u^0(y))-S(u(t,x,\alpha)-u^0(x))) \\
\times \psi(x)J_r(x-y)J_{r_0}^+(t) d\alpha\dxdt dy\bigg] =: \mathcal{T}_{r_0} + I^1.
\end{multline*}
We want to take the limit $r_0 \downarrow 0$. 
  First observe that we have the bound 
  \begin{displaymath}
   \abs{I^1} \leq \norm{S}_{\mathrm{Lip}}\E{\int_{\R^d}\int_{\R^d} \abs{u^0(x)-u^0(y)}
   \psi(x)J_r(x-y)dxdy} =: R,
  \end{displaymath}
  which is independent of $r_0$. Similarly, $\abs{\mathcal{T}^2} \leq R$. 
  Note that $\xi_{r_0} \rightarrow 0$ a.e.~as $r_0 \downarrow 0$, so due to 
  assumptions~\eqref{assumption:LipOnf} 
  and \eqref{assumption:LipOnSigma}, one may conclude by the dominated 
  convergence theorem and estimates similar 
  to those in \eqref{eq:EstOnFluxTermEntIneq} 
  and \eqref{eq:EstOnSquareTermEntIneq} that 
  \begin{displaymath}
   \lim_{r_0 \downarrow 0} \mathcal{T}^1 = \lim_{r_0 \downarrow 0} \mathcal{T}^3 = 0.
  \end{displaymath}
  Thus, it follows by \eqref{eq:IneqTimeContAtZero} that 
  \begin{displaymath}
   \lim_{r_0 \downarrow 0} \mathcal{T}_{r_0} \leq 2R.
  \end{displaymath}
  Since $r > 0$ was arbitrary, and $\lim_{r \downarrow 0}R = 0$, we 
  have arrived at $\lim_{r_0 \downarrow 0} \mathcal{T}_{r_0} \leq 0$.
  The desired result follows, since we can approximate 
  any Lipschitz function uniformly by 
  smooth functions with bounded derivatives.
\end{proof}

\section{The viscous approximation}\label{sec:ViscousApprox}
For each fixed $\varepsilon>0$, we denote by $\ue$ the 
solution of the regularized problem
\begin{equation}\label{eq:ViscousApprox}
 \left\{
  \begin{aligned}
  d\ue + \nabla \cdot f(\ue)dt &= \int_Z \sigma(x,\ue,z)W(dt,dz) 
  + \varepsilon \Delta \ue dt, &(t,x) \in \Pi_T, \\
  \ue(0,x) &= u^0(x),  			  & x \in \R^d.
  \end{aligned}
 \right.
\end{equation}
As in the deterministic case, the idea is to let $\varepsilon\to 0$ and 
obtain a solution to the stochastic 
conservation law \eqref{eq:StochasticBalanceLaw}. 
The entropy condition is meant to single out this limit as the 
only proper (weak) solution; the entropy solution. 
To show that this limit exists, a type of compactness 
argument is needed \cite{Bauzet:2012kx, FengNualart2008, ChenKarlsen2012}. 

The existence of a unique solution to \eqref{eq:ViscousApprox} may be 
found several places \cite{Bauzet:2012kx, FengNualart2008}. 
In particular, the semi-group approach presented in \cite[ch.~9]{PeszatZabczyk2007} 
may be applied. The functional setting of \cite{PeszatZabczyk2007} is 
that of a Hilbert space, and so the natural choice here 
is $L^2(\R^d,\phi)$ where $\phi \in \mathfrak{N}$. 
Due to the new functional setting, we have chosen to include proofs 
for some of the results relating to \eqref{eq:ViscousApprox}.

\subsection{A priori estimates and well-posedness}
Let $S_\varepsilon$ be the semi-group generated by the heat kernel. 
That is $S_\varepsilon(t) u = \Phi_\varepsilon(t) \star u$ where 
\begin{displaymath}
 \Phi_\varepsilon(t,x) := \frac{1}{(4\varepsilon\pi t)^{d/2}} 
 \exp\left(-\frac{\abs{x}^2}{4 \varepsilon t}\right).
\end{displaymath}
Let $F(u) = \nabla \cdot f(u)$ and $G$ be defined by \eqref{eq:GDef}. 
In this setting the key conditions \cite[p.~142]{PeszatZabczyk2007} for 
well-posedness of \eqref{eq:ViscousApprox} are:
\begin{itemize}
 \item[(F)] $D(F)$ is dense in $L^2(\R^d,\phi)$ and there is a 
 function $a:(0,\infty) \rightarrow (0,\infty)$ satisfying $\int_0^T a(t)\,dt < \infty$ for all $T < \infty$ 
 such that, for all $t > 0$ and $u,v \in D(F)$,
 \begin{align*}
  \norm{S_\varepsilon(t)F(u)}_{2,\phi} &\leq a(t)\left(1 + \norm{u}_{2,\phi}\right), \\
  \norm{S_\varepsilon(t)(F(u)-F(v))}_{2,\phi} &\leq a(t)\norm{u-v}_{2,\phi}.
 \end{align*}
 \item[(G)] $D(G)$ is dense in $L^2(\R^d,\phi)$ and there is a 
 function $b : (0,\infty) \rightarrow (0, \infty)$ satisfying $\int_0^T b^2(t) dt < \infty$ 
 for all $T < \infty$ such that, for all $t > 0$ and $u,v \in D(G)$,
 \begin{align*}
  \norm{S_\varepsilon(t)G(u)}_{\Lin_2(L^2(Z);L^2(\R^d,\phi))} 
  & \leq b(t)\left(1 + \norm{u}_{2,\phi}\right), \\
  \norm{S_\varepsilon(t)(G(u)-G(v))}_{\Lin_2(L^2(Z);L^2(\R^d,\phi))} 
  & \leq b(t)\norm{u-v}_{2,\phi}.
 \end{align*}
\end{itemize}
Suppose $u^0 \in L^2(\Omega,\F_0,P;L^2(\R^d,\phi))$. 
Under assumptions (F) and (G) we may conclude 
by \cite[Theorem~9.15, Theorem~9.29]{PeszatZabczyk2007} that 
there exists a unique predictable 
process $\ue:[0,T] \times \Omega \rightarrow L^2(\R^d,\phi)$ such that
\begin{itemize}
 \item[(i)] 
  \begin{equation}\label{eq:BoundednessOfViscApprox}
  	\sup_{0 \leq t \leq T}\E{\norm{\ue(t)}_{2,\phi}^2} < \infty.
  \end{equation}
 \item[(ii)] For all $0 \leq t \leq T$
  \begin{equation}\label{eq:MildViscSol}
   \begin{split}
    	\qquad \quad \ue(t,x) &= \int_{\R^d}\Phi_\varepsilon(t,x-y)u^0(y)\,dy \\
	 & \qquad - \int_{0}^{t}\int_{\R^d} \nabla_x
	 \Phi_{\varepsilon}(t-s,x-y) \cdot f(\ue(s,y)) \,dyds \\
	 & \qquad + \int_{0}^{t}\int_Z\int_{\R^d} 
	 \Phi_\varepsilon(t-s,x-y)\sigma(y,\ue(s,y),z)\,dy \,W(ds,dz).
   \end{split}
  \end{equation}
 \item[(iii)] $\ue$ is a weak solution of $\eqref{eq:ViscousApprox}$, i.e., for any 
 test function $\test \in C^\infty_c(\R^d)$ and any pair of 
 times $t_0,t$ with $0 \leq t_0 \leq t \leq T$,
  \begin{multline}\label{eq:WeakSolutionVisc}
   \qquad \int_{\R^d} \ue(t)\test \dx = \int_{\R^d} \ue(t_0)\test\dx 
   - \int_{t_0}^t\int_{\R^d} f(\ue(s)) \cdot \nabla\test\dx ds \\
   + \int_{t_0}^t\int_{\R^d} \int_Z \sigma(x,\ue(s,x),z)\test W(ds,dz) \dx 
   + \varepsilon \int_{t_0}^t\int_{\R^d} \ue \Delta \test\dx ds, 
  \end{multline}
\qquad \qquad \qquad  $dP$-almost surely.
\end{itemize}
To see that conditions (F) and (G) are satisfied 
we prove the following estimate:
\begin{lemma}\label{lemma:HeatKernelYoungIneq}
 Fix $\phi \in \mathfrak{N}$ and $1 \leq p < \infty$. 
 Let $v \in W^{1,p}(\R^d,\phi;\R^d)$, $u \in L^p(\R^d,\phi)$. Whenever 
 $C_\phi\sqrt{4\varepsilon t} \leq 1$,
 \begin{align}
    &\norm{\Phi_\varepsilon(t) \star u}_{p,\phi} 
    \leq \kappa_{1,d} \norm{u}_{p,\phi}, \tag{i} \label{eq:estOnHeat}\\
   &\norm{\Phi_\varepsilon(t) \star \nabla \cdot v}_{p,\phi} 
   \leq \frac{\kappa_{2,d}}{\sqrt{\varepsilon t}}\norm{v}_{p,\phi}, 
   \tag{ii}\label{eq:estOnHeatDiv}
 \end{align}
 where $\kappa_{1,d} =c_{d-1} \frac{d\alpha(d)}{\pi^{d/2}}, 
 \kappa_{2,d} =c_d\frac{d\alpha(d)}{\pi^{d/2}}$, and 
 \begin{displaymath}
  c_d = \int_0^\infty \zeta^d(1 + \zeta)^2\exp(\zeta-\zeta^2)\,d\zeta.
 \end{displaymath}
The volume of the unit ball in $\R^d$ is denoted by $\alpha(d)$.
\end{lemma}
Before we give a proof let us see why (F) and (G) follow. 
Recall that we may assume $f(0) = 0$ without any loss of generality. 
By Lemma~\ref{lemma:HeatKernelYoungIneq} 
and \eqref{assumption:LipOnf},
\begin{displaymath}
 \norm{S_\varepsilon(t)F(u)}_{2,\phi} =  \norm{\Phi_\varepsilon(t) 
 \star \nabla \cdot f(u)}_{2,\phi} 
 \leq\underbrace{\frac{\kappa_{2,d}}{\sqrt{\varepsilon t}}
 \norm{f}_{\mathrm{Lip}} }_{a(t)}\norm{u}_{2,\phi}.
\end{displaymath}
It remains to observe that $\int_0^T \frac{1}{\sqrt{t}}\,dt = 2\sqrt{T} < \infty$. 
The second part of $(F)$ follows similarly. 
Let us consider (G). First observe that 
\begin{displaymath}
 S_\varepsilon(t)G(u)h(x) = \int_Z \left(\int_{\R^d}
 \Phi_\varepsilon(t,x-y)\sigma(y,u(y),z)\,dy\right)h(z)\,d \mu(z).
\end{displaymath}
Recall that $HS = \Lin_2(L^2(Z);L^2(\R^d,\phi))$. By 
Lemma~\ref{lemma:HeatKernelYoungIneq} and \eqref{assumption:LipOnSigma}
\begin{align*}
 \norm{S_\varepsilon(t)G(u)}_{HS}^2 
      &= \int_Z\int_{\R^d} \left(\int_{\R^d}\Phi_\varepsilon(t,x-y)\sigma(y,u(y),z)\,dy\right)^2 
      \phi(x) \,dx d\mu(z) \\
      &= \int_Z\norm{\Phi_\varepsilon(t) \star \sigma(\cdot,u,z)}_{2,\phi}^2 d\mu(z) \\
      &\leq   \kappa_{1,d}^2\norm{M}_{L^2(Z)}^2 (\norm{\phi}_{L^1(\R^d)} 
      + \norm{u}_{2,\phi})^2.
\end{align*}
This yields the first part of condition (G). The second part follows similarly, in view of 
the Lipschitz assumption on $\sigma$.

\begin{proof}[Proof of Lemma~\ref{lemma:HeatKernelYoungIneq}]
 Consider \eqref{eq:estOnHeat}. By Proposition~\ref{prop:YoungsForLocalized},
 \begin{displaymath}
  \norm{\Phi_\varepsilon(t) \star u}_{p,\phi} \leq \underbrace{\left(\int_{\R^d} 
  \abs{\Phi_\varepsilon(t,x)}(1 + w_{p,\phi}(\abs{x}))\,dx\right)}_{\norm{\Phi(t)}} 
  \norm{u}_{p,\phi}.
 \end{displaymath}
 where 
 \begin{displaymath}
  w_{p,\phi}(r) = \frac{C_\phi}{p}r\left(1 
  + \frac{C_\phi}{p}r\right)\exp\left(\frac{C_\phi}{p} r\right).
 \end{displaymath}
 We apply polar coordinates to compute $\norm{\Phi(t)}$. This yields
 \begin{align*}
  \norm{\Phi(t)} &= \int_0^\infty \int_{\partial B(0,r)} 
  \abs{\Phi_\varepsilon(t)}(r)(1 + w_{p,\phi}(r))\,dS(r)\,dr \\
  &= \frac{d\alpha(d)}{(4\varepsilon\pi t)^{d/2}}\int_0^\infty r^{d-1}
  \exp\left(\frac{C_\phi}{p} r-\frac{r^2}{4\varepsilon t}\right)
  \\ & \qquad \qquad\qquad\quad
  \times \left(\exp\left(-\frac{C_\phi}{p}r\right) 
  + \frac{C_\phi}{p}r\left(1 + \frac{C_\phi}{p}r\right)\right)\,dr.
 \end{align*}
 To simplify, we note that  
 \begin{displaymath}
  \exp\left(-\frac{C_\phi}{p} r\right) + \frac{C_\phi}{p}r\left(1 + \frac{C_\phi}{p}r\right) 
  \leq \left(1 + \left(\frac{C_\phi}{p}r\right)\right)^2.
 \end{displaymath}
 Let $\zeta = r/\sqrt{4\varepsilon t}$. Provided 
 $C_\phi\sqrt{4\varepsilon t} \leq 1$, it follows that 
 \begin{displaymath}
  \frac{C_\phi}{p}r = \frac{C_\phi}{p}\sqrt{4\varepsilon t}\zeta \leq \zeta.
 \end{displaymath}
 Inserting this we obtain
 \begin{displaymath}
  \norm{\Phi(t)} \leq \frac{d\alpha(d)}{\pi^{d/2}}\int_0^\infty \zeta^{d-1}\left(1 
  + \zeta\right)^2\exp\left(\zeta-\zeta^2\right)\,d\zeta.
 \end{displaymath}
 
 Estimate \eqref{eq:estOnHeatDiv} follows 
 along the same lines. Integration by parts yields
 \begin{displaymath}
  \int_{\R^d}\Phi_\varepsilon(t,x-y)\nabla \cdot v(y)\,dy 
  = \int_{\R^d}\nabla_x\Phi_\varepsilon(t,x-y)\cdot v(y)\,dy.
 \end{displaymath}
 Hence,
 \begin{displaymath}
  \abs{\Phi_\varepsilon(t) \star \nabla \cdot v} \leq \abs{\nabla \Phi(t)} \star \abs{v}.
 \end{displaymath}
 By Proposition~\ref{prop:YoungsForLocalized},
 \begin{displaymath}
   \norm{\nabla_x \Phi_\varepsilon(t) \star v}_{L^p(\R^d)}  
        \leq \underbrace{\left(\int_{\R^d} \abs{\nabla \Phi_\varepsilon(t,x)}(1 
        + w_{p,\phi}(\abs{x}))\,dx\right)}_{\norm{\nabla \Phi(t)}}
        \norm{v}_{p,\phi}.
 \end{displaymath}
 Let $r = \abs{x}$. Then
 \begin{displaymath}
  \norm{\nabla \Phi(t)} = \int_0^\infty \underbrace{\int_{\partial B(0,r)} 
  \abs{\nabla \Phi_\varepsilon(t)}(r)(1 + w_{p,\phi}(r))\,dS(r)}_{\Psi(r)}\,dr.
 \end{displaymath}
 Now,
 \begin{displaymath}
  \nabla\Phi_\varepsilon(t,x) = -\frac{2\pi x}{(4\pi\varepsilon t)^{d/2 + 1}}
  \exp\left(-\frac{\abs{x}^2}{4\varepsilon t}\right),
 \end{displaymath}
and so
 \begin{displaymath}
  \Psi(r) = \frac{d\alpha(d)}{2\varepsilon t\pi^{d/2}}
  \left(\frac{r}{\sqrt{4\varepsilon t}}\right)^d
  \exp\left(\frac{C_\phi}{p} r-\frac{r^2}{4\varepsilon t}\right)
  \left(\exp\left(-\frac{C_\phi}{p} r\right) + \frac{C_\phi}{p}r
  \left(1 + \frac{C_\phi}{p}r\right)\right).
 \end{displaymath}
 
 Let $\zeta(r) = r/\sqrt{4\varepsilon t}$ and 
 suppose $C_\phi \sqrt{4\varepsilon t} \leq 1$. Then 
 \begin{multline*}
  \int_0^\infty \left(\frac{r}{\sqrt{4\varepsilon t}}\right)^d\exp
  \left(\frac{C_\phi}{p} r-\frac{r^2}{4\varepsilon t}\right)
  \left(1 + \left(\frac{C_\phi}{p}r\right)\right)^2\,dr \\
  \leq \sqrt{4\varepsilon t} \int_0^\infty 
  \zeta^d(1 + \zeta)^2\exp(\zeta-\zeta^2)\,d\zeta.
 \end{multline*}
This concludes the proof of the lemma.
\end{proof}

The following two lemmas constitute the reason why 
Lemma~\ref{lemma:HeatKernelYoungIneq} is the key to 
the well-posedness of \eqref{eq:ViscousApprox}. 
As we will see, the relevant properties of $\ue$ follow rather 
easily with these estimates at hand. 

\begin{lemma}\label{lemma:FGenEst}
Let $1 \leq p \leq \infty$ and $\phi \in \mathfrak{N}$. Suppose 
$v \in C([0,T];W^{1,p}(\R^d,\phi;\R^d))$. 
Set
\begin{displaymath}
 \mathcal{T}[v](t,x) = \int_0^t\int_{\R^d} 
 \Phi_\varepsilon(t-s,x-y)(\nabla \cdot v(s,y)) \,dyds.
\end{displaymath}
Then, for any $1 \leq q < \infty$,
\begin{displaymath}
 \norm{\mathcal{T}[v](t)}_{p,\phi}^q \leq 
 \kappa_{2,d}^q\left(2 \sqrt{\frac{t}{\varepsilon}}\right)^{q-1} 
 \int_0^t \frac{1}{\sqrt{\varepsilon(t-s)}}\norm{v(s)}_{p,\phi}^q \,ds,
\end{displaymath}
where $\kappa_{2,d} = c_d\frac{d\alpha(d)}{\pi^{d/2}}$ and $c_d$ is 
defined in Lemma~\ref{lemma:HeatKernelYoungIneq}.
\end{lemma}

\begin{proof}
By Minkowski's integral inequality \cite[p.271]{Stein1970} 
and Lemma~\ref{lemma:HeatKernelYoungIneq},
\begin{align*}
 \norm{\mathcal{T}[v](t)}_{p,\phi}^q 
  &\leq \left(\int_0^t\norm{\Phi_\varepsilon(t-s) 
  \star (\nabla \cdot v(s,y))}_{p,\phi} \,ds\right)^q \\
  &\leq \left(\int_0^t\frac{\kappa_{2,d}}{\sqrt{\varepsilon(t-s)}}
  \norm{v(s)}_{p,\phi} \,ds\right)^q.
\end{align*}
If $q = 1$ we are done, so we may assume $1 < q < \infty$. Let $r$ 
satisfy $1 = r^{-1} + q^{-1}$ and take
\begin{displaymath}
 h(s) := \left(\frac{1}{\sqrt{\varepsilon(t-s)}}\right)^{1/r} \quad \mbox{and} \quad 
 g(s) := \left(\frac{1}{\sqrt{\varepsilon(t-s)}}\right)^{1-1/r}\norm{v(s)}_{p,\phi}.
\end{displaymath}
By H\"older's inequality, $\norm{hg}_{L^1([0,t])}^q \leq 
\norm{g}_{L^q([0,t])}^q\norm{h}_{L^r([0,t])}^q$, 
and so
\begin{displaymath}
 \norm{\mathcal{T}[v](t)}_{p,\phi}^q \leq 
 (\kappa_{2,d}\norm{h}_{L^r([0,t])})^q \int_0^t 
 \frac{1}{\sqrt{\varepsilon(t-s)}}\norm{v(s)}_{p,\phi}^q \,ds.
\end{displaymath}
A simple computation yields
\begin{displaymath}
\norm{h}_{L^r([0,t])}^q = \left(\int_0^t 
 \frac{1}{\sqrt{\varepsilon(t-s)}} \,ds\right)^{q/r} 
= \left(2 \sqrt{\frac{t}{\varepsilon}}\right)^{q-1}.
\end{displaymath}
The result follows.
\end{proof}
\begin{lemma}\label{lemma:GGenEst}
Let $2 \leq p < \infty$ and $\phi \in \mathfrak{N}$. Suppose 
$v:\Omega \times [0,T] \times Z \times \R^d \rightarrow \R$ 
is a predictable process satisfying 
\begin{displaymath}
 \abs{v(s,x,z)} \leq K(s,x)M(z),
\end{displaymath}
for $M \in L^2(Z)$ and a process 
$K \in L^2([0,T];L^p(\Omega;L^p(\R^d,\phi)))$. Define 
\begin{displaymath}
 \mathcal{T}[v](t,x) = \int_0^t \int_Z \int_{\R^d} 
 \Phi_\varepsilon(t-s,x-y)v(s,y,z)\,dy W(dz,ds).
\end{displaymath}
Then
\begin{displaymath}
 \E{\norm{\mathcal{T}[v](t)}_{p,\phi}^p}^{1/p} \leq 
 c_p^{1/p}\kappa_{1,d}\norm{M}_{L^2(Z)}
 \left(\int_0^t \E{\norm{K(s)}_{p,\phi}^p}^{2/p}\,ds\right)^{1/2},
\end{displaymath}
where $c_p$ is the constant appearing in 
the Burkholder-Davis-Gundy inequality 
and $\kappa_{1,d} = c_{d-1}\frac{d \alpha(d)}{\pi^{d/2}}$, with 
$c_d$ defined in Lemma~\ref{lemma:HeatKernelYoungIneq}.
\end{lemma}

\begin{remark}
To prove this result we use the Burkholder-Davis-Gundy 
inequality for real-valued processes. Using Banach space valued 
versions \cite{CoxVeraar2012, NeervenVeraarWeis2007}, one can derive 
more general estimates.
\end{remark}

\begin{proof}
First note that 
\begin{displaymath}
 M(t,x) = \int_0^t \int_Z \int_{\R^d} \Phi_\varepsilon(\tau-s,x-y)v(s,y,z)\,dy W(dz,ds)
\end{displaymath}
is a martingale on $[0,\tau]$, and so by the 
Burkholder-Davis-Gundy inequality \cite{Khoshnevisan2009},
\begin{equation*}
 \E{\abs{\mathcal{T}[v](t,x)}^p} \leq c_p\E{\left(\int_0^t\int_Z
 \abs{\Phi_\varepsilon(t-s) \star v(s,\cdot,z)(x)}^2\,d\mu(z)\,ds \right)^{p/2}}.
\end{equation*}
Upon integrating in space and applying Minkowski's inequality, it follows that
\begin{align*}
&\E{\norm{\mathcal{T}[v](t)}_{p,\phi}^p}^{2/p} \\
& \hphantom{XXX}\leq c_p^{2/p}\E{\int_{\R^d}\left(\int_0^t
\int_Z\abs{\Phi_\varepsilon(t-s) \star v(s,\cdot,z)(x)}^2 
\,d\mu(z) ds \right)^{p/2}\phi(x)\,dx}^{2/p} \\
&\hphantom{XXX}\leq c_p^{2/p}\int_0^t\int_Z\E{\int_{\R^d}
\abs{\Phi_\varepsilon(t-s) \star v(s,\cdot,z)(x)}^p
\phi(x)\,dx}^{2/p}\,d\mu(z) ds.
\end{align*}
By Lemma~\ref{lemma:HeatKernelYoungIneq},
\begin{displaymath}
 \E{\norm{\mathcal{T}[v](t)}_{p,\phi}^p}^{2/p} \leq c_p^{2/p}\kappa_{1,d}^2\int_0^t 
 \int_Z\E{\norm{v(s,\cdot,z)}_{p,\phi}^p}^{2/p}d\mu(z)\,ds.
\end{displaymath}
By assumption,
\begin{displaymath}
 \int_0^t \int_Z\E{\norm{v(s,\cdot,z)}_{p,\phi}^p}^{2/p}d\mu(z)\,ds 
 \leq \norm{M}_{L^2(Z)}^2 \int_0^t \E{\norm{K(s)}_{p,\phi}^p}^{2/p}\,ds.
\end{displaymath}
\end{proof}

For a Banach space $E$ we denote by $\mathcal{X}_{\beta,q,E}$ the space 
of pathwise continuous predictable processes 
$u:[0,T] \times \Omega \rightarrow E$ normed by
\begin{equation}\label{eq:betanorm}
 \norm{u}_{\beta,q,E} := 
 \left(\sup_{t \in [0,T]} e^{-\beta t}\E{\norm{u(t)}_E^q}\right)^{1/q}.
\end{equation}
 The existence of a solution to \eqref{eq:ViscousApprox} is obtained 
 by the Banach fixed-point theorem, applied to the operator
\begin{align*}
 \mathcal{S}(u)(t,x) &:= \int_{\R^d}\Phi_\varepsilon(t,x-y)u^0(y)\,dy \\
	&\quad - \int_{0}^{t}\int_{\R^d} 
	\nabla_x\Phi_\varepsilon(t-s,x-y) \cdot f(u(s,y)) \,dyds \\
	&\quad + \int_{0}^{t}\int_Z \int_{\R^d} 
	\Phi_\varepsilon(t-s,x-y)\sigma(y,u(s,y),z) \,dy\,W(ds,dz),
\end{align*}
in the space $\mathcal{X}_{\beta,2,L^2(\R^d,\phi)}$, with $\beta \in \R$ 
sufficiently large. It follows that the sequence $\seq{u^n}_{n \geq 1}$ 
defined inductively by $u^0 = 0$ and $u^{n+1} = \mathcal{S}(u^n)$ 
converges to $\ue$ in $\mathcal{X}_{\beta,2,L^2(\R^d,\phi)}$ as $n \rightarrow \infty$. 
By Lemmas~\ref{lemma:FGenEst} and \ref{lemma:GGenEst} we are free to use 
the space $\mathcal{X}_{\beta,p,L^p(\R^d,\phi)}$ for any $2 \leq p < \infty$ in the 
fixed-point argument \cite{FengNualart2008}.

We can use Lemmas~\ref{lemma:FGenEst} and \ref{lemma:GGenEst} to 
deduce a continuous dependence result. To do this, we need a measure 
of the distance between the coefficients. For the flux function 
$f$, the Lipschitz norm is a reasonable choice.  Concerning the 
noise function $\sigma$, we introduce 
the norm $\norm{\sigma}_{\mathrm{Lip}} 
= \norm{M_\sigma}_{L^2(Z)}$, where 
\begin{displaymath}
M_\sigma(z) = \sup_{x \in \R^d}\left\{\sup_{u \in \R} 
\frac{\abs{\sigma(x,u,z)}}{1 + \abs{u}}\right\} + \sup_{x \in \R^d}\left\{\sup_{u \neq v}
\frac{\abs{\sigma(x,u,z)-\sigma(x,v,z)}}{\abs{u-v}}\right\}.
\end{displaymath}
Note that for any $\sigma$ satisfying \eqref{assumption:LipOnSigma}, we 
have $\norm{\sigma}_{\text{Lip}} < \infty$. 
\begin{proposition}[Continuous dependence]\label{proposition:ContDependVisc}
Let $2 \leq p < \infty$ and $\phi \in \mathfrak{N}$. Let $f_1,f_2$ 
satisfy \eqref{assumption:LipOnf} and $\sigma_1,\sigma_2$ 
satisfy \eqref{assumption:LipOnSigma}. 
Suppose $u_1^0,u_2^0 \in L^p(\Omega,\F_0,P;L^p(\R^d,\phi))$. 
Let $\ue_1$ and $\ue_2$ denote the weak solutions of the 
corresponding problems \eqref{eq:ViscousApprox} 
with $f = f_i, \sigma = \sigma_i$, and $u^0 = u^0_i$, for $i = 1,2$. 
Then, for $\beta > 0$ sufficiently large there exists a 
constant $C = C(\beta,\varepsilon,T,f_1,\sigma_1)$ such that 
\begin{align*}
 	\norm{\ue_1 - \ue_2}_{\beta,p,L^p(\R^d,\phi)} 
	\leq C\bigg(&\E{\norm{u^0_1-u^0_2}_{p,\phi}^p} 
	+\norm{f_1-f_2}_{\mathrm{Lip}}\norm{\ue_1}_{\beta,p,L^p(\R^d,\phi)} \\
	&\,+\norm{\sigma_1-\sigma_2}_{\mathrm{Lip}}
	\left(\norm{\phi}_{L^1(\R^d)} 
	+ \norm{\ue_1}_{\beta,p,L^p(\R^d,\phi)}\right)\bigg),
\end{align*}
where the norm $\norm{\cdot}_{\beta,p,L^p(\R^d,\phi)}$ 
is defined in \eqref{eq:betanorm}.
\end{proposition}

\begin{proof}
By \eqref{eq:MildViscSol},
\begin{align*}
\ue_1&(t,x)-\ue_2(t,x) = \int_{\R^d}
\Phi_\varepsilon(t,x-y)(u^0_1(y)-u^0_2(y))\,dy \\
&- \int_{0}^{t}\int_{\R^d} \nabla_x\Phi_{\varepsilon}(t-s,x-y) 
\cdot (f_1(\ue_1(s,y))-f_2(\ue_2(s,y))) \,dyds \\
&+ \int_{0}^{t}\int_Z\int_{\R^d} 
\Phi_\varepsilon(t-s,x-y)(\sigma_1(y,\ue_1(s,y),z)
-\sigma_2(y,\ue_2(s,y),z))\,dy \,W(ds,dz) \\
&= \mathcal{T}_1 + \mathcal{T}_2 + \mathcal{T}_3.
\end{align*}
By Lemma~\ref{lemma:HeatKernelYoungIneq},
\begin{displaymath}
 \E{\norm{\mathcal{T}_1(t)}_{p,\phi}^p} 
 \leq \kappa_{1,d}^p\E{\norm{u^0_1-u^0_2}_{p,\phi}^p}.
\end{displaymath}
where $\kappa_{1,d}$ is defined 
in Lemma~\ref{lemma:GGenEst}. Hence 
\begin{equation}\label{eq:ContDepEstT1}
 \norm{\mathcal{T}_1}_{\beta,p,L^p(\R^d,\phi)} 
 \leq \kappa_{1,d}\E{\norm{u^0_1-u^0_2}_{p,\phi}^p}^{1/p}.
\end{equation}
Consider $\mathcal{T}_2$. Note that 
\begin{multline*}
 \E{\norm{f_1(\ue_1(s))-f_2(\ue_2(s))}_{p,\phi}^p}^{1/p} \leq 
 \norm{f_1 - f_2}_{\mathrm{Lip}}\E{\norm{\ue_1(s)}_{p,\phi}^p}^{1/p} \\
 + \norm{f_2}_{\mathrm{Lip}}\E{\norm{\ue_1(s)-\ue_2(s)}_{p,\phi}^p}^{1/p}.
\end{multline*}
By Lemma~\ref{lemma:FGenEst},
\begin{align*}
 &\E{\norm{\mathcal{T}_2(t)}_{p,\phi}^p}  \\
 & \quad \leq \kappa_{2,d}^p\left(2\sqrt{\frac{t}{\varepsilon}}\right)^{p-1}
 \norm{f_1 - f_2}_{\mathrm{Lip}}^p\int_0^t \frac{1}{\sqrt{\varepsilon(t-s)}}
 \E{\norm{\ue_1(s)}_{p,\phi}^p} \,ds \\
 &\quad\quad + \kappa_{2,d}^p\left(2\sqrt{\frac{t}{\varepsilon}}\right)^{p-1}
 \norm{f_2}_{\mathrm{Lip}}^p\int_0^t \frac{1}{\sqrt{\varepsilon(t-s)}}
 \E{\norm{\ue_1(s)-\ue_2(s)}_{p,\phi}^p} \,ds.
\end{align*}
Multiplying by $e^{-\beta t}$ and taking the supremum yields
\begin{multline}\label{eq:ContDepEstT2}
 \norm{\mathcal{T}_2}_{\beta,p,L^p(\R^d,\phi)}  
\leq \delta_{\beta,1}\norm{f_1 - f_2}_{\mathrm{Lip}} \norm{\ue_1}_{\beta,p,L^p(\R^d,\phi)} \\
 + \delta_{\beta,1}\norm{f_2}_{\mathrm{Lip}}\norm{\ue_1-\ue_2}_{\beta,p,L^p(\R^d,\phi)},
\end{multline}
where
\begin{displaymath}
 \delta_{\beta,1} = \kappa_{2,d}\sup_{t \in [0,T]}\left(2\sqrt{\frac{t}{\varepsilon}}\right)^{1-1/p}
 \left(\int_0^t \frac{e^{-\beta(t-s)}}{\sqrt{\varepsilon(t-s)}} \,ds\right)^{1/p}.
\end{displaymath}
Consider $\mathcal{T}_3$. First, observe that 
\begin{displaymath}
 \abs{\sigma_1(y,\ue_1,z)-\sigma_2(y,\ue_2,z)} \leq M_{\sigma_1}(z) 
 \abs{\ue_1 - \ue_2} + M_{\sigma_1-\sigma_2}(z)(1 + \abs{\ue_1}).
\end{displaymath}
Due to a simple extension of Lemma~\ref{lemma:GGenEst},
\begin{align*}
 \E{\norm{\mathcal{T}_3(t)}_{p,\phi}^p}^{1/p} 
 &\leq c_p^{1/p}\kappa_{1,d}\norm{\sigma_1}_{\mathrm{Lip}}
 \left(\int_0^t \E{\norm{\ue_1(s) - \ue_2(s)}_{p,\phi}^p}^{2/p}\,ds\right)^{1/2} \\
 &+c_p^{1/p}\kappa_{1,d}\norm{\sigma_1-\sigma_2}_{\mathrm{Lip}}
 \left(\int_0^t \E{\norm{1 + \abs{\ue_1(s)}}_{p,\phi}^p}^{2/p}\,ds\right)^{1/2}.
\end{align*}
Multiplication by $e^{-\beta t/p}$ and then taking the supremum yields
\begin{multline}\label{eq:ContDepEstT3}
 \norm{\mathcal{T}_3}_{\beta,p,L^p(\R^d,\phi)} 
 \leq \delta_{\beta,2}\norm{\sigma_1}_{\mathrm{Lip}}
 \norm{\ue_1 - \ue_2}_{\beta,p,L^p(\R^d,\phi)} \\
 +\delta_{\beta,2}\norm{\sigma_1-\sigma_2}_{\mathrm{Lip}}
 (\norm{\phi}_{L^1(\R^d)} + \norm{\ue_1}_{\beta,p,L^p(\R^d,\phi)}),
\end{multline}
where 
\begin{displaymath}
 \delta_{\beta,2} = c_p^{1/p}\kappa_{1,d}\sup_{t \in [0,T]}\left(\int_0^t 
 e^{-\beta 2(t-s)/p}\,ds\right)^{1/2} \leq c_p^{1/p}\kappa_{1,d}\sqrt{\frac{p}{2\beta}}.
\end{displaymath}
Here we used that $\norm{1}_{\beta,p,L^p(\R^d,\phi)} 
= \norm{\phi}_{L^1(\R^d)}$. Combine \eqref{eq:ContDepEstT1}, \eqref{eq:ContDepEstT2}, 
and \eqref{eq:ContDepEstT3}, and note that 
$\delta_{\beta,i} \rightarrow 0$ as $\beta \rightarrow \infty$ for $i = 1,2$. 
This concludes the proof.
\end{proof}

In order to apply It\^o's formula to the process $t \mapsto \ue(t,x)$ we 
need to know that the weak (mild) solution 
$\ue$ of \eqref{eq:ViscousApprox} is in fact a strong solution. 
The following result provides the existence of weak derivatives.
\begin{proposition}\label{proposition:SobolevBoundsOnViscApprox}
Fix $\phi \in \mathfrak{N}$ and a multiindex $\tilde{\alpha}$. 
Make the following assumptions:
\begin{itemize}
\item[(i)]  The flux-function $f$ belongs to $C^{\abs{\tilde{\alpha}}}(\R;\R^d)$ 
with all derivatives bounded.
\item[(ii)] For each fixed $z \in Z$, $(x,u) \mapsto \sigma(x,u,z)$ 
belongs to $C^{\abs{\tilde{\alpha}}}(\R^d \times \R)$ and 
for each $0 < \alpha \leq \tilde{\alpha}$ and $0 \leq n \leq \abs{\tilde{\alpha}}$ there 
exists $M_{\alpha,n} \in L^2(Z)$ such that  
\begin{displaymath}
\begin{cases}
\partial_1^\alpha \partial_2^n \sigma(x,u,z) 
\leq M_{\alpha,n}(z), &  1 \leq n \leq \abs{\tilde{\alpha}}, \\
\partial_1^\alpha \sigma(x,u,z) \leq M_{\alpha,0}(z)(1 + \abs{u}).
\end{cases}
\end{displaymath}
\item[(iii)] The initial function $u^0$ satisfies for all $\alpha \leq \tilde{\alpha}$,
\begin{displaymath}
\E{\norm{\partial^\alpha u^0}_{p,\phi}^p} < \infty \qquad (2 \leq p < \infty).
\end{displaymath}
\end{itemize}
Let $\ue$ be the weak solution of \eqref{eq:ViscousApprox}. 
For any $\alpha \leq \tilde{\alpha}$, there exists a predictable process 
\begin{displaymath}
(t,x,\omega) \mapsto \partial^\alpha_x\ue(t,x,\omega) 
\mbox{ in }L^p([0,T] \times \Omega;L^p(\R^d,\phi))
\end{displaymath}
such that for all $\test \in C^\infty_c(\Pi_T)$, 
\begin{displaymath}
\iint_{\Pi_T} \partial^\alpha_x \ue \test \,dxdt 
= (-1)^{\abs{\alpha}}
\iint_{\Pi_T} \ue \partial^\alpha_x \test \,dxdt, 
\qquad \text{$dP$-almost surely.}
\end{displaymath}
\end{proposition}

To prove Proposition~\ref{proposition:SobolevBoundsOnViscApprox} we apply
\begin{lemma}\label{lemma:IteratedChainRule}
Let $\sigma \in C^\infty(\R^d \times \R)$ and suppose $u \in C^\infty(\R^d)$. 
For any multiindex $\alpha$, let 
$\partial^\alpha_x := \prod_k \partial_{x_k}^{\alpha_k}$. Then
\begin{displaymath}
 \partial^\alpha_x \sigma(x,u(x)) = \sum_{\zeta \leq \alpha}\sum_{\gamma \in \pi(\zeta)} 
 C_{\gamma,\alpha}\partial_1^{\alpha-\zeta}\partial_2^{\abs{\gamma}}\sigma(x,u(x))
 \prod_{i = 1}^{\abs{\gamma}} \partial_x^{\gamma^i}u(x).
\end{displaymath}
Here $\pi(\zeta)$ denotes all partitions of $\zeta$, i.e., all 
multiindices $\gamma = \seq{\gamma^i}_{i \geq 1}$ 
such that $\sum \gamma^i = \zeta$. 
Furthermore, $\abs{\gamma}$ denotes the number of 
terms in the partition $\gamma$.
\end{lemma}

\begin{remark}
Whenever $\zeta \neq 0$ we assume that the terms $\gamma^i$ in 
the partition $\gamma$ satisfies $\gamma^i \neq 0$. 
If $\zeta = 0$ we let $\gamma = \gamma^1 = 0$ 
and by convention let $\abs{\gamma} = 0$.
\end{remark}

\begin{proof}
One may prove by induction and the chain rule that
\begin{displaymath}
 \partial^\alpha_x \sigma(x,u(x)) = 
 \left[(\partial_z + \partial_y)^\alpha \sigma(y,u(z))\right]_{y = x,z = x}.
\end{displaymath}
By the binomial theorem, 
\begin{displaymath}
 (\partial_z + \partial_y)^\alpha \sigma(y,u(z)) = 
 \sum_{\zeta \leq \alpha} \binom{\alpha}{\zeta}
 \partial^{\alpha-\zeta}_y\partial^\zeta_z\sigma(y,u(z)).
\end{displaymath}
Thanks to \cite[Propositions 1 and 2]{Hardy2006}, it follows that 
\begin{displaymath}
 \partial^\zeta_z\sigma(y,u(z)) = \sum_{\gamma \in \pi(\zeta)}
 M_{\gamma}\partial_u^{\abs{\gamma}}\sigma(y,u(z))
 \prod_i\partial^{\gamma_i}_zu(z),
\end{displaymath}
where $M_{\gamma}$ is a constant. 
The result follows by combining the above identities.
\end{proof}

\begin{proof}[Proof of Proposition \ref{proposition:SobolevBoundsOnViscApprox}]
We divide the proof into two steps.
 
\emph{Step~1 (uniform estimates on $\seq{u^n}_{n \geq 1}$)}. 
For all $\zeta < \alpha$, suppose
\begin{equation}\label{eq:InductionHypOnMultider}
 \sup_{0 \leq s \leq T} 
 \E{\norm{\partial^\zeta u^n(s)}_{p,\phi}^p} 
 \leq C_{\zeta,p} \qquad (2 \leq p < \infty).
\end{equation}
We claim that there exists 
a constant $C \geq 0$, independent of $\beta$ 
and $n$, and a number $\delta_\beta \geq 0$ such that 
\begin{equation}\label{eq:MainIndIneqDerEst}
 \norm{\partial^\alpha u^{n+1}}_{\beta,p,L^p(\R^d,\phi)} 
 \leq C + \delta_\beta\norm{\partial^\alpha u^{n}}_{\beta,p,L^p(\R^d,\phi)},
\end{equation}
where $\delta_\beta < 1$ for some $\beta > 0$, and 
$\norm{\cdot}_{\beta,p,L^p(\R^d,\phi)}$ is defined in \eqref{eq:betanorm}. 
Given \eqref{eq:MainIndIneqDerEst}, it follows that  
\begin{displaymath}
 	\norm{\partial^\alpha u^n}_{\beta,p,L^p(\R^d,\phi)} 
 	\leq C \sum_{k = 0}^{n-1}
  	\delta_\beta^k \leq \frac{C}{1 - \delta_\beta},   
\end{displaymath}
and we are done. 

To establish \eqref{eq:MainIndIneqDerEst}, observe that 
the weak derivative satisfies 
\begin{align*}
\partial^\alpha_xu^{n+1}(t,x) 
	&= \int_{\R^d}\Phi_\varepsilon(t,x-y) \partial_y^\alpha u^0(y)\,dy \\
	&\quad - \int_{0}^{t}\int_{\R^d} \nabla_x\Phi_\varepsilon(t-s,x-y) 
	\cdot \partial^\alpha_y f(u^n(s,y)) \,dyds \\
	&\quad + \int_{0}^{t}\int_Z \int_{\R^d} \Phi_\varepsilon(t-s,x-y)
	\partial^\alpha_y\sigma(y,u^n(s,y),z) \,dy\,W(ds,dz) \\
	&=: \mathcal{T}_1(t,x) + \mathcal{T}_2(t,x) + \mathcal{T}_3(t,x).
\end{align*}
To justify this, multiply by a test function and 
apply the Fubini theorem \cite[p.297]{Walsh1984}.  
By the triangle inequality we may estimate each term separately. 

Consider $\mathcal{T}_1$. By Lemma~\ref{lemma:HeatKernelYoungIneq},
\begin{displaymath}
\E{\norm{\mathcal{T}_1(t)}_{p,\phi}^p} 
  \leq \kappa_{1,d}^p\E{\norm{\partial^\alpha u^0}_{p,\phi}^p},
\end{displaymath}
where $\kappa_{1,d}$ is defined in Lemma~\ref{lemma:GGenEst}. 
By assumption (iii) it follows that there exists a constant $C$ such that 
\begin{equation}\label{eq:DerEstT1}
\norm{\mathcal{T}_1}_{\beta,p,L^p(\R^d,\phi)} \leq C.
\end{equation}

Consider $\mathcal{T}_2$. By Lemma~\ref{lemma:FGenEst},
\begin{align*}
\norm{\mathcal{T}_2}&_{\beta,p,L^p(\R^d,\phi)}^p 
  = \sup_{t \in [0,T]} e^{-\beta t} \E{\norm{\mathcal{T}_2(t)}_{p,\phi}^p} \\
  &\qquad \leq \kappa_{2,d}^p\left(2\sqrt{\frac{T}{\varepsilon}}\right)^{p-1}\sup_{t \in [0,T]}
  \int_0^t \frac{e^{-\beta(t-s)}}{\sqrt{\varepsilon(t-s)}}e^{-\beta s} 
  \E{\norm{\partial^\alpha f(u^n(s))}_{p,\phi}^p}\,ds \\
  &\qquad \leq \kappa_{2,d}^p\left(2\sqrt{\frac{T}{\varepsilon}}\right)^{p-1}
  \left(\int_0^T \frac{e^{-\beta(t-s)}}{\sqrt{\varepsilon(t-s)}}\,ds\right) 
  \norm{\partial^\alpha f(u^n(s))}_{\beta,p,L^p(\R^d,\phi)}^p.
\end{align*}
By Lemma~\ref{lemma:IteratedChainRule}, the triangle inequality, and 
the generalized H\"older inequality,
\begin{align*}
\E{\norm{\partial^\alpha f(u^n(s))}_{p,\phi}^p}^{1/p} 
&\leq \sum_{\gamma \in \pi(\alpha)} 
C_{\gamma,\alpha}\norm{\partial^{\abs{\gamma}}f}_{\infty}
\E{\norm{\prod_{i=1}^{\abs{\gamma}} \partial^{\gamma^i}u^n(s)}_{p,\phi}^p}^{1/p} \\
&\leq \sum_{\gamma \in \pi(\alpha)} 
C_{\gamma,\alpha}\norm{\partial^{\abs{\gamma}}f}_{\infty}
\prod_{i=1}^{\abs{\gamma}}\E{\norm{ \partial^{\gamma^i}u^n(s)}_{q_i,\phi}^{q_i}}^{1/q_i},
\end{align*}
whenever $\sum_{i = 1}^{\abs{\gamma}}\frac{1}{q_i} = \frac{1}{p}$. 
Since $\gamma$ is a partition of $\alpha$, $\sum_{i = 1}^{\abs{\gamma}}\abs{\gamma^i} 
= \abs{\alpha}$, and we may take $q_i = \abs{\alpha}p/\abs{\gamma^i}$. 
By assumption there exists a 
constant $C$, independent of $n$, such that 
\begin{displaymath}
\E{\norm{ \partial^{\gamma^i}u^n(s)}_{q_i,\phi}^{q_i}} \leq C,
\end{displaymath}
for all terms where $\gamma^i < \alpha$. Since $C_{\alpha,\alpha} = 1$, 
there is another constant $C$ such that
\begin{equation*}
\E{\norm{\partial^\alpha f(u^n(s))}_{p,\phi}^p}^{1/p} \leq C 
+ \norm{f'}_\infty \E{\norm{\partial^\alpha u^n(s)}_{p,\phi}^p}^{1/p},
\end{equation*}
for all $n \geq 1$. Multiply by $e^{-\beta t/p}$ and 
take the supremum to obtain
\begin{displaymath}
\norm{\partial^\alpha f(u^n)}_{\beta,p,L^p(\R^d,\phi)} \leq C 
+ \norm{f'}_\infty \norm{\partial^\alpha u^n}_{\beta,p,L^p(\R^d,\phi)}.
\end{displaymath}
It follows that 
\begin{equation}\label{eq:DerEstT2}
\begin{split}
\norm{\mathcal{T}_2}_{\beta,p,L^p(\R^d,\phi)} &
 \leq c_d\left(2\sqrt{\frac{T}{\varepsilon}}\right)^{1-1/p}\left(\int_0^T 
 \frac{e^{-\beta(t-s)}}{\sqrt{\varepsilon(t-s)}}\,ds\right)^{1/p} \\
 & \hphantom{XXXXXXXXX} \times \left(C + \norm{f'}_\infty 
 \norm{\partial^\alpha u^n}_{\beta,p,L^p(\R^d,\phi)}\right).
\end{split}
\end{equation}

Consider $\mathcal{T}_3$. By Lemma~\ref{lemma:IteratedChainRule}
\begin{displaymath}
\begin{split}
\mathcal{T}_3(t,x) &= \sum_{\zeta \leq \alpha}\sum_{\gamma \in \pi(\zeta)} 
C_{\gamma,\alpha}\int_{0}^{t}\int_Z \int_{\R^d} \Phi_\varepsilon(t-s,x-y) \\
&\hphantom{XXXXX} \times 
\partial_1^{\alpha-\zeta}\partial_2^{\abs{\gamma}}\sigma(y,u^n(s,y),z)
\left(\prod_{i=1}^{\abs{\gamma}} \partial_y^{\gamma^i}u^n(s,y)\right) \,dy\,W(ds,dz) \\
&= \mathcal{T}_3^0(t,x) + \mathcal{T}_3^1(t,x),
\end{split}
\end{displaymath}
where $\mathcal{T}_3^0$ contains the term with $\zeta = 0$. 
By Lemma~\ref{lemma:GGenEst}, assumption~(ii), and the 
generalised H\"older inequality,
\begin{align*}
\E{\norm{\mathcal{T}_3^1(t)}_{p,\phi}^p}^{1/p} 
&\leq \sum_{0 < \zeta \leq \alpha}\sum_{\gamma \in \pi(\zeta)} 
C_{\gamma,\alpha}c_p^{1/p}\kappa_{1,d} 
\norm{M_{\alpha-\zeta,\abs{\gamma}}}_{L^2(Z)} \\
& \hphantom{XXXXXX} \times \left(\int_0^t\prod_{i=1}^{\abs{\gamma}}
\E{\norm{ \partial^{\gamma^i}u^n(s)}_{q_i,\phi}^{q_i}}^{2/q_i}\,ds\right)^{1/2},
\end{align*}
where $q_i = \abs{\zeta}p/\abs{\gamma^i}$. The term $\mathcal{T}_3^0$ is 
estimated similarly by applying the second case of assumption (ii). 
It follows from \eqref{eq:InductionHypOnMultider} 
that there exists a constant $C$ such that 
\begin{displaymath}
\E{\norm{\mathcal{T}_3(t)}_{p,\phi}^p}^{1/p} 
\leq C + c_p^{1/p}\kappa_{1,d}\norm{M_{0,1}}_{L^2(Z)}
\left(\int_0^t \E{\norm{\partial^\alpha 
u^n(s)}_{p,\phi}^p}^{2/p}\,ds\right)^{1/2}.
\end{displaymath}
Multiplying by $(e^{-\beta t})^{1/p}$ and 
taking the supremum yields
\begin{equation}\label{eq:DerEstT3}
\begin{split}
\norm{\mathcal{T}_3}_{\beta,p,L^p(\R^d,\phi)} 
 &\leq C + c_p^{1/p}\kappa_{1,d}\norm{M_{0,1}}_{L^2(Z)} \\
 & \qquad \times \sup_{t \in [0,T]}\left(\int_0^t e^{-2\beta(t-s)/p}
 \left(e^{-\beta s}\E{\norm{\partial^\alpha u^n(s)}_{p,\phi}^p}\right)^{2/p}\,ds\right)^{1/2} \\
 &\leq C + c_p^{1/p}\kappa_{1,d}\norm{M_{0,1}}_{L^2(Z)}\sqrt{\frac{p}{2\beta}}
 \norm{\partial^\alpha u^n}_{\beta,p,L^p(\R^d,\phi)}.
\end{split}
\end{equation}

Combining \eqref{eq:DerEstT1}, \eqref{eq:DerEstT2}, and \eqref{eq:DerEstT3} 
we obtain inequality \eqref{eq:MainIndIneqDerEst}, where
\begin{align*}
\delta_\beta &= \kappa_{2,d}\left(2\sqrt{\frac{T}{\varepsilon}}\right)^{1-1/p}
\left(\int_0^T \frac{e^{-\beta(t-s)}}{\sqrt{\varepsilon(t-s)}}\,ds\right)^{1/p}\norm{f'}_\infty \\
	      & \hphantom{XXXXXXXXXXXXXXXXXX}
	      + c_p^{1/p}\kappa_{1,d}\norm{M_{0,1}}_{L^2(Z)}\sqrt{\frac{p}{2\beta}}.
\end{align*}
It is clear that $\delta_\beta \rightarrow 0$ as 
$\beta \rightarrow \infty$ and so \eqref{eq:MainIndIneqDerEst} follows. 
By induction, estimate \eqref{eq:InductionHypOnMultider} holds for all $\zeta \leq \alpha$.

\emph{Step~2 (convergence of $u^n$)}. 
Fix $ \alpha \leq \tilde{\alpha}$. We apply 
Theorem~\ref{theorem:YoungMeasureLimitOfComposedFunc} 
to the familiy $\seq{\partial^\alpha u^n}_{n \geq 1}$ on the space
\begin{displaymath}
(X,\mathscr{A},\mu) = 
(\Omega \times \Pi_T, \Pred \otimes \Borel{\R^d}, dP \otimes dt \otimes \phi(x)dx).
\end{displaymath}
By means of \eqref{eq:InductionHypOnMultider},
\begin{displaymath}
\sup_{n \geq 1}\left\{\E{\iint_{\Pi_T} 
\abs{\partial^\alpha u^n(t,x)}^2 \phi(x)\,dxdt}\right\} < \infty.
\end{displaymath}
Hence, $\seq{\partial^\alpha u^n}_{n \geq 1}$ has a 
Young measure limit $\nu^\alpha \in \Young{\Omega \times \Pi_T}$. 
Next, define $\partial^\alpha \ue(t,x,\omega) 
:= \int_\R \,d\nu^\alpha_{t,x,\omega}$. 
By definition, the limit has a $\Pred \otimes \Borel{\R^d}$ measurable version. 
Furthermore, $\partial^\alpha \ue \in L^p(\Omega \times [0,T];L^p(\R^d,\phi))$, 
cf.~proof of Theorem~\ref{theorem:ExistenceOfSolution} 
and Lemma~\ref{lemma:LebBochRepr}. Let us show 
that $\partial^\alpha \ue$ is the 
weak derivative of $\ue$. To this end, observe that 
\begin{displaymath}
	\iint_{\Pi_T}u^n(t,x)\partial^\alpha \test  \,dxdt = 
	(-1)^{\abs{\alpha}} \iint_{\Pi_T} \partial^\alpha u^n(t,x) \test  \,dxdt,
\end{displaymath}
for any $\test \in C^\infty_c(\R^d)$. By Lemma~\ref{lemma:UniformIntCriteria}(ii) 
and Theorem~\ref{theorem:DunfordPettis}, there is 
a subsequence $\seq{n(j)}_{j \geq 1}$ such that for any $A \in \F$,
\begin{align*}
 \lim_{j \rightarrow \infty}\E{\car{A}\iint_{\Pi_T} \partial^\alpha u^{n(j)} \test \,dxdt} 
&= \E{\iint_{\Pi_T}\int_\R \test(t,x)\car{A}(\omega)\,d\nu^\alpha_{t,x,\omega}(\xi)\,dxdt} \\
&= \E{\car{A}\iint_{\Pi_T} \partial^\alpha \ue \test\,dxdt}.
\end{align*}
As $u^n \rightarrow \ue$ in 
$\mathcal{X}_{\beta,2,L^2(\R^d,\phi)}$, it follows that 
\begin{displaymath}
	\lim_{n \rightarrow \infty} \E{\car{A}\iint_{\Pi_T}u^n 
	\partial^\alpha \test \,dxdt} 
	= \E{\car{A}\iint_{\Pi_T}\ue \partial^\alpha \test \,dxdt}.
\end{displaymath}
This concludes the proof.
\end{proof}

\subsection{Malliavin differentiability}
We will establish the Malliavin differentiability 
of the viscous approximations. Furthermore, we will observe that the Malliavin 
derivative satisfies a linear parabolic equation. This equation is then applied 
to show that $D_{r,z}\ue(t,x) \rightarrow \sigma(x,\ue(r,x),z)$ 
as $t \downarrow r$ in a weak 
sense (Lemma~\ref{lemma:MalliavinDerivativeWeakTimeCont}); a 
property that is crucial in the proof of uniqueness. 

\begin{proposition}[Malliavin derivative of viscous 
approximation]\label{proposition:MalliavinDiffOfViscApprox}
Suppose \eqref{assumption:LipOnf} and \eqref{assumption:LipOnSigma} 
are satisfied. Fix $\phi \in \mathfrak{N}$ and $u^0 \in L^2(\Omega,\F_0,P;L^2(\R^d,\phi))$. 
Let $\ue$ be the solution of \eqref{eq:ViscousApprox}. 
Then $\ue$ belongs to $\D^{1,2}(L^2([0,T];L^2(\R^d,\phi)))$ and 
\begin{equation}\label{eq:PointwiseMalliavinBoundOnVisc}
	\esssup_{0 \leq t \leq T} \norm{\ue(t)}_{\D^{1,2}(L^2(\R^d,\phi))} < \infty.
\end{equation}
Furthermore, for $dr \otimes d\mu$-a.a.~$(r,z)$, the $L^2(\R^d,\phi)$-valued 
process $\seq{D_{r,z}\ue(t)}_{t>r}$ is a predictable weak solution of
\begin{equation}\label{eq:MallEqSat}
 \left\{\begin{split}
   dw + \nabla \cdot (f'(\ue)w)\dt 
   &= \int_Z  \partial_2\sigma(x,\ue,z')w\, W(dt,dz') 
   + \varepsilon\Delta w\dt, \quad t \in [r,T], \\
   w(r,x,z) &= \sigma(x,\ue(r,x),z),
 \end{split}\right.
\end{equation}
while $D_{r,z}\ue(t) = 0$ if $r > t$. Furthermore
 \begin{equation}\label{eq:supInrMallEst}
	\esssup_{r \in [0,T]} \left\{\sup_{t \in [0,T]}
	E{\norm{D_r\ue(t)}_{L^2(Z;L^2(\R^d,\phi))}^2}\right\} < \infty.
 \end{equation}
\end{proposition}

\begin{remark}
Let $w_{r,z}(t,x) = D_{r,z}\ue(t,x)$, $t > r$. 
Estimate \eqref{eq:supInrMallEst} may be seen 
as a consequence of the Gr\"onwall-type estimate
\begin{displaymath}
	\E{\norm{w_{r,z}(t)}_{2,\phi}^2} \leq 
 	\left(1 + Ce^{C(t-r)}\right)
	\E{\norm{w_{r,z}(r)}_{2,\phi}^2},
\end{displaymath}
for $t \geq r$. From the perspective of a uniqueness result (see 
Lemma~\ref{lemma:MalliavinDerivativeWeakTimeCont}), it is of 
interest to know whether one can derive such estimates 
independent of $\varepsilon$.
\end{remark}

\begin{proof} 
We divide the proof into two steps.

\emph{Step~1 (uniform bounds)}. Consider the Picard 
approximation $\seq{u^n}_{n \geq 1}$ of $\ue$. We want to prove that 
\begin{equation}\label{eq:UniformBoundOnMallDer}
	\sup_{0 \leq t \leq T} \norm{u^n(t)}_{\D^{1,2}(L^2(\R^d,\phi))} 
	\leq C, \quad \mbox{for all $n \geq 1$}.
\end{equation}
Recall that 
\begin{displaymath}
 \norm{u^n(t)}_{\D^{1,2}(L^2(\R^d,\phi))}^2 = \E{\norm{u^n(t)}_{2,\phi}^2} 
 + \E{\norm{Du^n(t)}_{H \otimes L^2(\R^d,\phi)}^2},
\end{displaymath}
where $H$ is the space $L^2([0,T] \times Z)$. 
Note that there is a constant $C$ such that 
\begin{equation}\label{eq:UniformBoundOnun}
 \sup_{0 \leq t \leq T}\E{\norm{u^n(t)}_{2,\phi}^2} < C,
\end{equation}
by the proof of Proposition~\ref{proposition:SobolevBoundsOnViscApprox}, 
Step~1 with $\alpha = 0$. Next, we claim that 
there exists a constant $C$ such that 
\begin{equation}\label{eq:StepwiseBoundForMallDiff}
\norm{Du^{n+1}}_{\beta,2,H \otimes L^2(\R^d,\phi)} \leq C 
+\delta_\beta \norm{Du^n}_{\beta,2,H \otimes L^2(\R^d,\phi)}, 
\quad \mbox{where $\delta_\beta < 1$},
\end{equation}
for some $\beta > 0$. We then conclude that 
\begin{displaymath}
 \norm{Du^{n}}_{\beta,2,H \otimes L^2(\R^d,\phi)} \leq C \sum_{k = 0}^{n-1} 
 \delta_\beta^k \leq \frac{C}{1-\delta_\beta},
\end{displaymath}
and \eqref{eq:UniformBoundOnMallDer} follows. 

Let us establish \eqref{eq:StepwiseBoundForMallDiff}. 
By \cite[Propositions 1.2.4, 1.3.8, and 1.2.8]{NualartMalliavinCalc2006},
\begin{equation}\label{eq:MallDiffOfUn}
 \begin{split}
	&D_{r,z}u^{n+1}(t,x) \\ 
	& \quad = \int_{\R^d} \Phi_\varepsilon(t-r,x-y) \sigma(y,u^n(r,y),z) \,dy\\
	&\qquad - \int_{r}^{t}\int_{\R^d} \nabla_x\Phi_\varepsilon(t-s,x-y) 
	\cdot f'(u^n(s,y))D_{r,z}u^n(s,y) \,dyds \\
	&\qquad + \int_{r}^{t}\int_Z \int_{\R^d} \Phi_\varepsilon(t-s,x-y) 
	\partial_2\sigma(y,u^n(s,y),z')D_{r,z}u^n(s,y) \,dy\,W(ds,dz') \\
	&\quad =: \mathcal{T}_1^n + \mathcal{T}_2^n + \mathcal{T}_3^n,
 \end{split}
\end{equation}
for all $r \in (0,t]$. Whenever $r > t$, $D_ru^{n+1}(t) = 0$ 
since $u^{n+1}$ is adapted, see \cite[Corollary~1.2.1]{NualartMalliavinCalc2006}. 
We proceed by estimating each term of \eqref{eq:MallDiffOfUn} separately.

Consider $\mathcal{T}_1^n$. By Lemma~\ref{lemma:HeatKernelYoungIneq}
and assumption \eqref{assumption:LipOnSigma},
\begin{align*}
 \norm{\mathcal{T}_1^n(r,t)}_{L^2(Z;L^2(\R^d,\phi))}^2 
    &= \int_Z \norm{\Phi_\varepsilon(t-r) \star  \sigma(\cdot,u^n(r),z)}_{2,\phi}^2\,d\mu(z) \\
    &\leq \kappa_{1,d}^2\norm{M}_{L^2(Z)}^2
    \left(\norm{\phi}_{L^1(\R^d)} + \norm{u^n(r)}_{2,\phi}\right)^2,
\end{align*}
for each $0 \leq r < t$. 
It follows from \eqref{eq:UniformBoundOnun} that 
\begin{equation}\label{eq:MallDiffEstT1}
  \begin{split}
   &\norm{\mathcal{T}_1^n}_{\beta,2,H \otimes L^2(\R^d,\phi)}   \\
	& \hphantom{XX}\leq \kappa_{1,d}\norm{M}_{L^2(Z)} 
	\left(\sup_{0 \leq t \leq T}e^{-\beta t}\E{\int_0^t 
	(\norm{\phi}_{L^1(\R^d)} + \norm{u^n(r)}_{2,\phi})^2\,dr}\right)^{1/2} \leq C.
  \end{split}
\end{equation}

Consider $\mathcal{T}_2^n$. By Lemma~\ref{lemma:FGenEst},
\begin{align*}
 &\norm{\mathcal{T}_2^n(r,t)}_{L^2(Z;L^2(\R^d,\phi))}^2 
  = \int_Z \norm{\int_{r}^{t} \nabla\Phi_\varepsilon(t-s) 
  \star f'(u^n(s))D_{r,z}u^n(s) ds}_{2,\phi}^2\, d\mu(z) \\
  & \hphantom{XXXXXX}\leq 2\kappa_{2,d}^2\norm{f'}_{\infty}^2
  \sqrt{\frac{t-r}{\varepsilon}}\int_{r}^{t} \frac{1}{\sqrt{\varepsilon(t-s)}}
  \int_Z\norm{D_{r,z}u^n(s)}_{2,\phi}^2d\mu(z)\,ds.
\end{align*}
Multiplication by $e^{-\beta t}$ and integration in $r$ yields
\begin{align*}
 &e^{-\beta t}\E{\norm{\mathcal{T}_2^n(t)}_{H \otimes L^2(\R^d,\phi)}^2} \\
 & \hphantom{XXX}\leq 2\kappa_{2,d}^2\norm{f'}_{\infty}^2
 \sqrt{\frac{t}{\varepsilon}}\int_{0}^{t} \frac{e^{-\beta(t-s)}}{\sqrt{\varepsilon(t-s)}}
    e^{-\beta s}\E{\norm{Du^n(s)}_{H \otimes L^2(\R^d,\phi)}^2}\,ds.
\end{align*}
It follows that
\begin{multline}\label{eq:MallDiffEstT2}
 \norm{\mathcal{T}_2^n}_{\beta,2,H \otimes L^2(\R^d,\phi)} 
    \leq \kappa_{2,d}\norm{f'}_{\infty}\left(2\sqrt{\frac{T}{\varepsilon}}\int_{0}^{T} 
    \frac{e^{-\beta(T-s)}}{\sqrt{\varepsilon(T-s)}}\,ds\right)^{1/2} \\
    \times \norm{Du^n(s)}_{\beta,2,H \otimes L^2(\R^d,\phi)}.
\end{multline}

Consider $\mathcal{T}_3^n$. Due to \eqref{assumption:LipOnSigma},
\begin{displaymath}
	\abs{\partial_2 \sigma(x,u,z')D_{r,z}u^n(s,y)} 
	\leq M(z')\abs{D_{r,z}u^n(s,y)}.
\end{displaymath}
By Lemma~\ref{lemma:GGenEst},
\begin{multline*}
	\E{\norm{\mathcal{T}_3^n(r,t)}_{L^2(Z;L^2(\R^d,\phi))}^2} \\
	\leq c_2 \kappa_{1,d}^2\norm{M}_{L^2(Z)}^2 
	\int_r^t \E{\norm{D_{r}u^n(s)}_{L^2(Z;L^2(\R^d,\phi))}^2}\,ds.
\end{multline*}
Integrate in $r$ and multiply by $e^{-\beta t}$ to obtain
\begin{multline*}
 e^{-\beta t}\E{\norm{\mathcal{T}_3^n(t)}_{H \otimes L^2(\R^d,\phi)}^2} \\
  \leq c_2 \kappa_{1,d}^2\norm{M}_{L^2(Z)}^2 \int_0^te^{-\beta(t-s)}e^{-\beta s}
  \E{\norm{Du^n(s)}_{H \otimes L^2(\R^d,\phi)}^2}\,ds.
\end{multline*}
Hence,
\begin{equation}\label{eq:MallDiffEstT3}
	\norm{\mathcal{T}_3^n}_{\beta,2,H \otimes L^2(\R^d,\phi)}
  	\leq c_2^{1/2}\kappa_{1,d}\norm{M}_{L^2(Z)} 
  	\frac{1}{\sqrt{\beta}}\norm{Du^n}_{\beta,2,H \otimes L^2(\R^d,\phi)}.
\end{equation}
Combining \eqref{eq:MallDiffEstT1}, \eqref{eq:MallDiffEstT2}, 
and \eqref{eq:MallDiffEstT3} 
yields \eqref{eq:StepwiseBoundForMallDiff} with
\begin{displaymath}
	\delta_\beta = \kappa_{2,d}\norm{f'}_{\infty}
	\left(2\sqrt{\frac{T}{\varepsilon}}\int_{0}^{T} 
	\frac{e^{-\beta(T-s)}}{\sqrt{\varepsilon(T-s)}}\,ds\right)^{1/2} 
	+ c_2^{1/2}\kappa_{1,d}\norm{M}_{L^2(Z)} \frac{1}{\sqrt{\beta}},
\end{displaymath}
where $\kappa_{2,d}$ is the constant from Lemma~\ref{lemma:FGenEst},
while $c_2$ and $\kappa_{1,d}$ are the 
constants from Lemma~\ref{lemma:GGenEst}. 
Note that $\delta_\beta \downarrow 0$ as 
$\beta \rightarrow \infty$. Leaving out the integration in $r$ 
throughout Step~1, we deduce the estimate
\begin{equation}\label{eq:rBoundOnMallApprox}
	\norm{D_ru^{n}}_{\beta,2,L^2(Z) \otimes 
	L^2(\R^d,\phi)} \leq \frac{C}{1-\delta_\beta}.
\end{equation}

\emph{Step~2 (convergence)}. 
Let $E$  denote the space $L^2([0,T];L^2(\R^d,\phi))$ 
and recall that $H = L^2(Z \times [0,T])$. 
Consider $\seq{u^n}_{n \geq 1}$ as a sequence in $\D^{1,2}(E)$.  
By \eqref{eq:UniformBoundOnMallDer} and the Hilbert 
space valued version of \cite[Lemma~1.2.3]{NualartMalliavinCalc2006} 
(see \cite[Lemma~5.2]{CarmonaTehranchi2006}), it follows that $\ue$ 
belongs to $\D^{1,2}(E)$ and that $Du^n \rightharpoonup D\ue$ (weakly)
in $L^2(\Omega;H \otimes E)$, i.e., for any $h \in H, \test \in E$, 
and $V \in L^2(\Omega)$,
\begin{displaymath}
	\E{\inner{Du^n}{h \otimes \test}_{H \otimes E}V} 
	\rightarrow  \E{\inner{D\ue}{h \otimes \test}_{H \otimes E}V}. 
\end{displaymath}
It follows that the map 
\begin{displaymath}
 (t,\omega) \mapsto D\ue(t,\omega) \in L^2(H \otimes L^2(\R^d,\phi))
\end{displaymath}
is $\Pred^*$-measurable. Note that Lemma~\ref{lemma:LebBochRepr} 
extends to this case, so that $(t,\omega) \mapsto D_{r,z}\ue(t,\omega)$ 
is $\Pred^*$-measurable for $dr \otimes d\mu$ almost all $(r,z) \in [0,T] \times Z$.

For each fixed $t \in [0,T]$, we conclude by \eqref{eq:UniformBoundOnMallDer} 
that $\ue(t) \in \D^{1,2}(L^2(\R^d,\phi))$, where $Du^n(t) \rightharpoonup D\ue(t)$ 
(weakly) along some subsequence. Besides, this limit agrees $dt$-almost everywhere 
with the evaluation of the limit taken in $\D^{1,2}(E)$. This follows by definition for 
smooth Hilbert space valued random variables 
and may be extended to the general case by approximation. 
The weak lower semicontinuity of the norm yields 
\eqref{eq:PointwiseMalliavinBoundOnVisc}. 
Similarly, we may apply \eqref{eq:rBoundOnMallApprox} and 
the Banach-Alaoglu theorem to extract a weakly convergent subsequence 
in the space $L^\infty([0,T];\mathcal{X}_{\beta,2,L^2(Z) \otimes L^2(\R^d,\phi)})$. 
This yields the bound \eqref{eq:supInrMallEst}. 

As above, $\ue(t,x) \in \D^{1,2}$ for $dt \otimes dx$-almost all $(t,x)$, and for 
such $(t,x)$ we have $D_{r,z}\ue(t,x) = D\ue(t,x,r,z)$ for $d\mu \otimes dr$ almost 
all $(r,z)$ where $D\ue(t,x,r,z)$ denotes the evaluation of the limit taken in $\D^{1,2}(E)$. 
Taking the Malliavin derivative of \eqref{eq:MildViscSol} (as above on $u^{n+1}$) 
it follows that $t \mapsto D_{r,z}\ue(t)$ is a mild solution 
of \eqref{eq:MallEqSat} for $dr \otimes d\mu$ almost all $(r,z)$. 
To conclude by \cite[Theorem~9.15]{PeszatZabczyk2007} that it 
is a weak solution, we verify conditions (F) 
and (G), with $F(w,t) = \nabla \cdot (f'(\ue(t))w)$ and
\begin{displaymath}
	G(w,t)h(x) = \int_Z \partial_2\sigma(x,\ue(t,x),z)w(x)h(z)\,d\mu(z).
\end{displaymath}
\end{proof}

The next result concerns the limit of $D_{r,z}\ue(t,x)$ as $t \downarrow r$. 
In view of Lemma~\ref{proposition:MalliavinDiffOfViscApprox}, this is a 
question about the satisfaction of the initial condition for \eqref{eq:MallEqSat}.

\begin{lemma}\label{lemma:MalliavinDerivativeWeakTimeCont}
Let $\phi \in C^\infty_c(\R^d)$ be non-negative. In the setting of 
Proposition~\ref{proposition:MalliavinDiffOfViscApprox}, for 
$\Psi \in L^2(\Omega \times Z;L^2(\R^d,\phi))$, set
\begin{displaymath}
\mathcal{T}_{r_0}(\Psi) := 
\E{\iiint\limits_{\quad Z \times \Pi_T} \left(D_{r,z}\ue(t,x) 
-\sigma(x,\ue(r,x),z)\right)J^+_{r_0}(t-r) \Psi\phi\,dtdxd\mu(z)}.
\end{displaymath} 
Then there exists a constant $C$ independent of $r_0$ such that 
\begin{equation}\label{eq:UniformBoundOnTr0}
\abs{\mathcal{T}_{r_0}(\Psi)} \leq C \E{\norm{\Psi}_{L^2(Z;L^2(\R^d,\phi))}^2}^{1/2}.
\end{equation}
and $\lim_{r_0 \downarrow 0}\mathcal{T}_{r_0}(\Psi) = 0$ 
for $dr$-almost all $r \in [0,T]$.
\end{lemma}

\begin{proof} 
Note that $\mathcal{T}_{r_0} = \mathcal{T}_{r_0}^1-\mathcal{T}_{r_0}^2$, and 
consider each term separately. By H\"older's inequality,
\begin{align*}
	\abs{\mathcal{T}_{r_0}^1} &= \abs{\int_0^T\E{\int_Z\int_{\R^d}
	\Psi(x,z)D_{r,z}\ue(t,x)\phi(x)\,dxd\mu(z)} J^+_{r_0}(t-r) \,dt} \\
	&\leq \esssup_{t \in [0,T]}\E{\int_Z\int_{\R^d}\Psi(x,z)D_{r,z}\ue(t,x)\phi(x)\,dxd\mu(z)} \\
	&\leq\E{\norm{\Psi}_{L^2(Z;L^2(\R^d,\phi))}^2}^{1/2}
	\esssup_{t \in [0,T]}\seq{\E{\norm{D_{r}\ue(t)}_{L^2(Z;L^2(\R^d,\phi)}^2}^{1/2}}.
\end{align*}
Furthermore, due to \eqref{assumption:LipOnSigma},
\begin{align*}
	\mathcal{T}_{r_0}^2 &= E\bigg[\int_Z\int_{\R^d}
	\Psi(x,z)\sigma(x,\ue(r,x),z)\phi(x)\,dxd\mu(z)\,\bigg] \\
	&\leq \norm{M}_{L^2(Z)}\E{\norm{\Psi}_{L^2(Z;L^2(\R^d,\phi))}^2}^{1/2}
	\E{\norm{1 + \abs{\ue(r)}}_{L^2(\R^d,\phi)}^2}^{1/2}.
\end{align*}
The uniform bound \eqref{eq:UniformBoundOnTr0} follows 
by \eqref{eq:supInrMallEst} and \eqref{eq:BoundednessOfViscApprox}. 
Note that $\Psi\mapsto \mathcal{T}_{r_0}(\Psi)$ is a linear functional 
on $L^2(\Omega \times Z;L^2(\R^d,\phi))$, for each $r_0 > 0$. 
By \eqref{eq:UniformBoundOnTr0} the family $\seq{\mathcal{T}_{r_0}}_{r_0 > 0}$ is 
uniformly continuous. Hence, by approximation, it suffices to prove 
the lemma for $\Psi$ smooth in $x$ with bounded derivatives. Let
\begin{displaymath}
	\test(t,x,z) = \Psi(x,z)\phi(x)\xi_{r_0,r}(t), \quad 
	\xi_{r_0,r}(t) =1-\int_0^t J_{r_0}^+(\sigma-r)\,d\sigma.
\end{displaymath}
By Proposition \ref{proposition:MalliavinDiffOfViscApprox},
\begin{align*}
	0 &= \int_{\R^d} \sigma(x,\ue(r,x),z)\test(r,x,z)\,dx \\
	&\quad + \int_r^T\int_{\R^d} D_{r,z}\ue(t,x)\partial_t\test(t,x,z)\,dxdt \\
	&\quad + \int_r^T\int_{\R^d} f'(\ue(t,x))D_{r,z}\ue(t,x) \cdot \nabla \test(t,x,z)\,dxdt \\
	&\quad +\varepsilon\int_r^T\int_{\R^d}D_{r,z}\ue(t,x)\Delta \test(t,x,z)\,dxdt \\
	&\quad + \int_r^T\int_Z \int_{\R^d} 
	\partial_2\sigma(x,\ue(t,x),z')D_{r,z}\ue(t,x)\test(t,x,z)\,dxW(dz',dt),
\end{align*}
$dr \otimes d\mu \otimes dP$-almost all $(r,z,\omega)$. Note that
\begin{displaymath}
	\partial_t\test(t,x,z)= -\Psi(x,z)\phi(x)J_{r_0}^+(t-r).
\end{displaymath}
Taking expectations and integrating in $z$ we obtain
\begin{align*}
	\mathcal{T}_{r_0}(\Psi) &=\E{\int_Z\int_r^T\int_{\R^d} 
	f'(\ue(t,x))D_{r,z}\ue(t,x) \cdot \nabla(\Psi \phi)\xi_{r_0,r}(t)\,dxdtd\mu(z)} \\
	& \qquad +\varepsilon \E{\int_Z\int_r^T
	\int_{\R^d}D_{r,z}\ue(t,x)\Delta (\Psi\phi)\xi_{r_0,r}(t)\,dxdtd\mu(z)}.
\end{align*}
As $\lim_{r_0 \downarrow 0}\xi_{r_0,r}(t) = 0$ for all $t > r$, it follows by the 
dominated convergence theorem 
that $\lim_{r_0 \downarrow 0}\mathcal{T}_{r_0}(\Psi) = 0$.
\end{proof}

\section{Existence of entropy solutions}\label{sec:Existence}
We will now prove the existence entropy 
solutions, as defined in Section~\ref{sec:Entropy_Formulation}. 

\begin{theorem}\label{theorem:ExistenceOfSolution}
Fix $\phi \in \mathfrak{N}$ and $2 \leq p < \infty$. Suppose 
$u^0 \in L^p(\Omega,\F_0,P;L^p(\R^d,\phi))$, and \eqref{assumption:LipOnf} 
and \eqref{assumption:LipOnSigma} hold. Then the generalized limit 
$u = \lim_{\varepsilon \downarrow 0} \ue$ of 
the viscous approximations \eqref{eq:ViscousApprox} 
is a Young measure-valued entropy solution 
of \eqref{eq:StochasticBalanceLaw} in the sense of 
Definition~\ref{Def:YoungEntropySolution}. Moreover, 
$u\in L^p(\Omega \times [0,T];L^p(\R^d \times [0,1],\phi))$. 
\end{theorem}

The bounds in Proposition~\ref{proposition:SobolevBoundsOnViscApprox} 
blow up as $\varepsilon \downarrow 0$. 
Below we establish bounds that are independent of 
the regularization parameter $\varepsilon > 0$. 

\begin{lemma}[Uniform bounds]\label{lemma:UniformBoundsOnVisc}
Suppose \eqref{assumption:LipOnf} and \eqref{assumption:LipOnSigma} hold, 
and $u^0$ belongs to $L^p(\Omega,\F_0,P;L^p(\R^d,\phi))$ for some 
even number $p\ge 2$ and $\phi \in \mathfrak{N}$. 
Then there exists a constant $C$, depending on 
$u^0,f,\sigma,p,T, \phi$ but not 
on $\varepsilon$, such that
\begin{equation}\label{eq:UniformBoundLpPhi}
	\E{\norm{\ue(t)}_{p,\phi}^p} \leq C, \qquad t \in [0,T].
\end{equation}
\end{lemma}

\begin{proof}
Suppose $u^0,f,\sigma$ satisfy the assumptions of 
Proposition~\ref{proposition:SobolevBoundsOnViscApprox} 
for $\abs{\alpha} \leq 2$. In view of 
Proposition~\ref{proposition:ContDependVisc}, the 
general result follows by approximation. 
Set $\phi_\delta = \phi \star J_\delta$. 
By Lemma~\ref{lemma:ContMollWeightedNorm} there 
is a constant $C_{\delta}$ satisfying 
$\abs{\Delta \phi_\delta} \leq C_{\delta}\phi_\delta$. 
By Proposition \ref{proposition:SobolevBoundsOnViscApprox}, $\ue$ is 
a strong solution of \eqref{eq:ViscousApprox}. 
Hence we may apply It\^o's formula to the 
function $S(u) = \abs{u}^p$, cf.~Step 2 in the upcoming proof 
of Theorem~\ref{theorem:ExistenceOfSolution}. 
After multplying by $\phi_\delta$ and integrating the result in $x$, 
\begin{equation*}\label{eq:ItoFormFroSquareAppliedToVisc}
\begin{split}
	&\norm{\ue(t)}_{p,\phi_\delta}^p = \norm{u^0}_{p,\phi_\delta}^p \\
	& \quad-p\int_0^t\int_{\R^d} \abs{\ue(s,x)}^{p-1}
	\sign{\ue(s,x)}\left(\nabla \cdot f(\ue(s,x)) 
	- \varepsilon \Delta \ue(s,x) \right)\phi_\delta(x) \,dxds \\
	& \quad +p\int_0^t\int_Z\int_{\R^d} \abs{\ue(s,x)}^{p-1}\sign{\ue(s,x)}
	\sigma(x,\ue(s,x),z)\phi_\delta(x)\,dx W(ds,dz) \\
	& \quad +\frac{1}{2}p(p-1)\int_0^t\int_{\R^d} \abs{\ue(s,x)}^{p-2} 
	\int_Z \sigma^2(x,\ue(s,x),z)\phi_\delta(x) \, d\mu(z) dxds.
\end{split}
\end{equation*}

Let $q(u) = p\int_0^u \abs{z}^{p-1}\sign{z}\partial f(z)\dz$ and 
note that by \eqref{assumption:LipOnf},
\begin{equation}\label{eq:EntFluxEst}
	\abs{q(u)} = \abs{\int_0^{u} p\abs{z}^{p-1}\sign{z} 
	\partial f(z)\,dz} \leq \norm{f}_{\mathrm{Lip}} \abs{u}^p.
\end{equation}
It follows that $q(\ue(t))\phi_\delta \in L^1(\Omega;L^1(\R^d;\R^d))$ for $0 \leq t \leq T$. 

By the chain rule and integration by parts,
\begin{displaymath}
\begin{split}
	\mathcal{T}_1 :=& \int_0^t\int_{\R^d} p\abs{\ue(s,x)}^{p-1}
	\sign{\ue(s,x)} \nabla \cdot f(\ue(s,x))\phi_\delta(x)\, dx ds\\ 
	 =& \int_0^t\int_{\R^d} \partial q(\ue(s,x)) \cdot \nabla \ue(s,x)\phi_\delta(x)\, dx ds \\
	 =& -\int_0^t\int_{\R^d} q(\ue(s,x))\cdot \nabla \phi_\delta(x)\, dx ds.
 \end{split}
\end{displaymath}
By \eqref{eq:EntFluxEst} and the fact that $\phi_\delta \in \mathfrak{N}$,
\begin{displaymath}
	\abs{\mathcal{T}_1} \leq C_\phi 
	\norm{f}_{\mathrm{Lip}}\int_0^t \norm{\ue(s)}_{p,\phi_\delta}^p\,ds.
\end{displaymath}

Again by the chain rule and integration by parts,
\begin{align*}
	\mathcal{T}_2 :=& \varepsilon\int_0^t\int_{\R^d} p\abs{\ue(s,x)}^{p-1}\sign{\ue(s,x)}  
	\Delta \ue(s,x)\phi_\delta(x)\,dxds \\
	= & -\varepsilon p(p-1)\int_0^t\int_{\R^d} \abs{\ue(s,x)}^{p-2} 
	\abs{\nabla \ue(s,x)}^2\phi_\delta(x)\,dxds \\
	& \quad -\varepsilon\int_0^t\int_{\R^d} p\abs{\ue(s,x)}^{p-1}
	\sign{\ue(s,x)}  \nabla \ue(s,x) \cdot \nabla \phi_\delta(x)\,dxds \\
	=& -\varepsilon p(p-1)\int_0^t\int_{\R^d} \abs{\ue(s,x)}^{p-2} 
	\abs{\nabla \ue(s,x)}^2\phi_\delta(x)\,dxds \\
	& \quad +\varepsilon\int_0^t\int_{\R^d} \abs{\ue(s,x)}^p 
	\Delta \phi_\delta(x)\,dxds.
\end{align*}
Hence, 
$$
\abs{\mathcal{T}_2} \leq C_{\delta} 
\varepsilon \int_0^t \norm{\ue(s)}_{p,\phi_\delta}^p \,ds. 
$$

Finally, by assumption~\eqref{assumption:LipOnSigma},
\begin{align*}
	\mathcal{T}_3 & := \frac{1}{2}p(p-1)\int_0^t\int_{\R^d} 
	\abs{\ue(s,x)}^{p-2} \int_Z \sigma^2(x,\ue(s,x),z)\phi_\delta(x) \, d\mu(z) dxds \\
	& \leq \frac{1}{2}p(p-1)\norm{M}_{L^2(Z)}^2\int_0^t\int_{\R^d} 
	\abs{\ue(s,x)}^{p-2}(1 + \abs{\ue(s,x)})^2\phi_\delta(x) \,dxds \\
	& \leq p(p-1)\norm{M}_{L^2(Z)}^2
	\left(\, \int_0^t \norm{\ue(s)}^{p-2}_{p-2,\phi_\delta}ds 
	+ \int_0^t\norm{\ue(s)}^p_{p,\phi_\delta}ds\right).
\end{align*}

After taking expectations and summarizing our findings, we arrive at
\begin{equation*}
\begin{split}
	\E{\norm{\ue(t)}_{p,\phi_\delta}^p} \leq \underbrace{\E{\norm{u^0}_{p,\phi_\delta}^p}
	+p(p-1)\norm{M}_{L^2(Z)}^2\int_0^t\E{\norm{\ue(s)}^{p-2}_{p-2,\phi_\delta}}ds}_{C_2} \\
	+ \underbrace{\left(C_\phi\norm{f}_{\mathrm{Lip}} 
	+ \varepsilon C_{\delta} + p(p-1)
	\norm{M}_{L^2(Z)}^2\right)}_{C_1}
	\int_0^t \E{\norm{\ue(s)}_{p,\phi_\delta}^p} \ds,
 \end{split}
\end{equation*}
and hence, appealing to Gr\"onwall's inequality,
\begin{displaymath}
	\E{\norm{\ue(t)}_{p,\phi_\delta}^p} 
	\leq C_2\left(1 + C_1t e^{C_1t}\right).
\end{displaymath}

Observe that the result holds for $p$ if it also holds 
for $p-2$, as long as $u^0$ belongs to $L^p(\Omega;L^p(\R^d,\phi_\delta))$. 
For this reason, \eqref{eq:UniformBoundLpPhi} follows by induction for $\phi = \phi_\delta$. 
The bound \eqref{eq:UniformBoundLpPhi} follows (for $\phi = \phi$) by Lemma~\ref{lemma:ContMollWeightedNorm}~(ii).
\end{proof}

\begin{remark}
In the forgoing proof, it is certainly possible to apply the 
Burkholder-Davis-Gundy inequality, resulting 
in the improvement
$$
\E{\, \sup_{0 \leq t \leq T} 
\norm{\ue(t)}_{p,\phi}^p} \leq C,
$$
for some constant $C$ independent of $\varepsilon$.
\end{remark}

\begin{proof}[Proof of Theorem~\ref{theorem:ExistenceOfSolution}]
We divide the proof into two main steps.

\emph{Step~1 (convergence)}.  We apply 
Theorem~\ref{theorem:YoungMeasureLimitOfComposedFunc} 
to the viscous approximation $\seq{\ue}_{\varepsilon > 0}$ 
on the measure space
\begin{displaymath}
	(X,\mathscr{A},\mu) 
	= (\Omega \times \Pi_T, \Pred \otimes \Borel{\R^d}, 
	dP \otimes dt \otimes \phi(x)dx). 
\end{displaymath}
By Lemma~\ref{lemma:UniformBoundsOnVisc},
\begin{displaymath}
	\sup_{\varepsilon > 0}\left\{\E{\iint_{\Pi_T}
	\abs{\ue(t,x)}^2 \phi(x)\,dxdt}\right\} < \infty,
\end{displaymath}
so we may take $\zeta(\xi) = \xi^2$. It follows that there 
exists a subsequence $\varepsilon_k \downarrow 0$ and a 
Young measure $\nu = \nu_{t,x,\omega}$ such that for 
any Carath\'eodory function $\psi = \psi(u,t,x,\omega)$ satisfying 
$$
\psi(u^{\varepsilon_k}(\cdot),\cdot) \rightharpoonup 
\overline{\psi}(\cdot) \,\mbox{ weakly in } L^1(\Omega \times \Pi_T),
$$
we have
\begin{equation}\label{eq:RepOfLimit}
\overline{\psi}(t,x,\omega) = \int_\R \psi(\xi,t,x,\omega)\,d\nu_{t,x,\omega}(\xi) 
= \int_0^1 \psi(u(t,x,\alpha,\omega),t,x,\omega)\,d\alpha.
\end{equation}
Here $u(t,x,\cdot,\omega)$ is defined 
through \eqref{eq:ReprOfProcessByYoung}, i.e.,
$$
u(t,x,\alpha,\omega) = 
\inf \left\{ \xi \in \R \,:\, \nu_{t,x,\omega}((-\infty,\xi]) > \alpha \right\}.
$$

We want to show that the limit $u$ is measurable, i.e., that it has 
a version $\tilde{u}$ such that for any $\beta \in \R$,
$$
B_\beta = \seq{(t,x,\alpha,\omega): \tilde{u}(t,x,\alpha,\omega) 
\geq \beta} \in \Pred \otimes \Borel{\R^d} \otimes \Borel{[0,1]} .
$$
Note that 
$$
\tilde{u}(t,x,\alpha,\omega) \geq \beta 
\Leftrightarrow \inf_{\xi \in \R} \seq{\nu_{t,x,\omega}((-\infty,\xi]) > \alpha} \geq \beta 
\Leftrightarrow \nu_{t,x,\omega}((-\infty,\beta]) \leq \alpha.
$$
By definition of the Young measure we pick a version (not relabeled) 
such that, the mapping $(t,x,\omega) \mapsto \nu_{t,x,\omega}((-\infty,\beta])$ 
is $\Pred \otimes \Borel{\R^d}$-measurable. 
Furthermore, if it were finitely valued it would be clear that $B_\beta$ is 
in the product topology, i.e., $B_\beta \in \F \otimes \Borel{\Pi_T} \otimes \Borel{[0,1]}$. 
Hence, the result follows upon approximation by simple 
functions \cite[Example~5.3.1]{Cohn2013}. 

Let us show that $u \in L^p([0,T] \times \Omega;L^p(\R^d \times [0,1],\phi))$. That is,
\begin{equation}\label{eq:BundOnYoungLimit}
	\E{\iint_{\Pi_T}\int_0^1 \abs{u(t,x,\alpha)}^p\phi(x)\,d\alpha dxdt} < \infty.
\end{equation}
Let $\psi \in C_c^\infty(\R)$ be supported in $(-1,1)$ and 
satisfy $0 \leq \psi \leq 1, \psi(0) = 1$. Take $\psi_R(u) = \psi(u/R)$. 
It follows that $\lim_{R \rightarrow \pm \infty}\psi_R(u) = 1$ for each $u \in \R$. 
Moreover, with $S_R(u) = \abs{u}^p \psi_R(u)$, note that 
$S_R(u) \uparrow \abs{u}^p$ for all $u \in \R$. 
Since $S_R$ is compactly supported it follows that 
$\seq{S_R(\ue)}_{\varepsilon > 0}$ is uniformly integrable on $X$. 
By Theorem \ref{theorem:DunfordPettis}, there is 
a subsequence $\varepsilon_{k(j)} \downarrow 0$ (denoted by $\varepsilon_j$) 
and a limit $\overline{\psi}$ such that 
$S_R(u^{\varepsilon_j}) \rightharpoonup \overline{\psi}$ 
(weakly) in $L^1(\Omega \times \Pi_T)$, where the weak 
limit $\overline{\psi}$ can be expressed in terms of the Young 
measure, cf.~\eqref{eq:RepOfLimit}. For this reason,
 \begin{align*}
 	&\E{\iint_{\Pi_T}\int_0^1 S_R(u(t,x,\alpha)) \phi(x)\,d\alpha dxdt} 
	\\ & \qquad = \lim_{j \rightarrow \infty} 
	\E{\iint_{\Pi_T} S_R(u^{\varepsilon_j}(t,x)) \phi(x)\,dxdt} \\
	& \qquad \leq \limsup_{\varepsilon \downarrow 0}
	\E{\iint_{\Pi_T} \abs{\ue(t,x)}^p \phi(x)\,dxdt} 
	\le C \quad \text{(by Lemma~\ref{lemma:UniformBoundsOnVisc})},
\end{align*}
for some constant $C$ independent of $R$.
The claim \eqref{eq:BundOnYoungLimit} follows 
upon sending $R \rightarrow \infty$, applying the 
monotone convergence theorem.

\emph{Step~2 (entropy condition)}. 
Let us for the moment assume that $f, \sigma, u^0$ satisfy 
the assumptions of Proposition \ref{proposition:SobolevBoundsOnViscApprox} for all 
multiindices $\abs{\alpha} \leq 2$. Fix an entropy/entropy-flux pair $(S,Q)$ in 
$\mathscr{E}$, a nonnegative test function $\test \in C^\infty_c([0,T) \times \R)$, and a 
random variable $V \in \Sm$. The goal is to show that the 
limit $u$ from Step~1 satisfies $\Y{\Entropy[(S,Q),\test,V]}(u) \geq 0$. 

By Proposition~\ref{proposition:SobolevBoundsOnViscApprox}, $\ue$ 
is a strong solution of \eqref{eq:ViscousApprox}. Indeed, consider 
the weak form~\eqref{eq:WeakSolutionVisc}, integrate by parts
(cf.~Proposition~\ref{proposition:SobolevBoundsOnViscApprox}), and 
use a (separating) countable subset $\seq{\test_n}_{n \geq 1} 
\subset C_c^\infty(\R^d)$ of test functions, to arrive at
\begin{align*}
	\ue(t,x) = u^0(x) & + \int_0^t\varepsilon 
	\Delta \ue(s,x)-\nabla \cdot f(\ue(s,x))\,ds \\
	&+\int_0^t \int_Z\sigma(x,\ue(s,x),z) W(ds,dz),
	\quad \text{$dx \otimes dP$-almost surely.}
\end{align*}
Next, we apply the anticipating It\^{o} formula (Theorem ~\ref{theorem:AntIto}), for 
fixed $x \in \R^d$, to $X_t = \ue(t,x)$ and $F(X,V,t) = S(X-V)\test(t,x)$. 
This yields, after taking expectations and integrating in $x$,
\begin{equation}\label{eq:AppliedItoFormToViscApprox}
\begin{split}
	0 &= \E{\int_{\R^d} S(u^0-V)\test(0)\dx} \\
	& \qquad + \E{\iint_{\Pi_T} S(\ue(t)-V)\partial_t\test(t)\dx dt} \\
	&\qquad - \E{\iint_{\Pi_T} \nabla \cdot f(\ue(t))S'(\ue(t)-V)\test(t) \dx dt} \\
	&\qquad + \E{\varepsilon\iint_{\Pi_T} \Delta \ue(t)S'(\ue(t)-V)\test(t)\dx dt} \\
	&\qquad - \E{\iint_{\Pi_T}\int_Z S''(\ue(t)-V)\test(t)\sigma(x,\ue(t),z)D_{t,z}V\,d\mu(z)\dx dt} \\
	&\qquad + \frac{1}{2} \E{\iint_{\Pi_T}\int_Z S''(\ue(t)-V)\test(t)\sigma(x,\ue(t),z)^2\,d\mu(z)\dx dt},
\end{split}
\end{equation}
where $D_{t,z}V$ is the Malliavin derivative of $V$ at $(t,z)$. 
By the chain rule and integration by parts,
\begin{align*}
\varepsilon\iint_{\Pi_T} \Delta \ue(t)S'(\ue(t)-V)\test(t)\,dxdt 
&= \varepsilon\iint_{\Pi_T}S(\ue(t)-V)\Delta\test(t)\,dxdt \\
&\quad \underbrace{- \varepsilon\iint_{\Pi_T} 
S''(\ue(t)-V)\abs{\nabla\ue(t)}^2\test(t)\,dxdt}_{\leq 0}.
\end{align*}
It follows from \eqref{eq:AppliedItoFormToViscApprox} that
\begin{equation}\label{eq:EntIneqApprox}
\begin{split}
	&E\Bigg[\int_{\R^d} S(u^0(x)-V)\test(0,x)\dx\Bigg] \\
	&\quad +E\Bigg[\iint_{\Pi_T} \underbrace{S(\ue(t,x)-V)
	\partial_t\test(t,x)}_{\psi_1(\ue,\cdot)} 
	+ \underbrace{Q(\ue(t,x),V)\cdot 
	\nabla \test(t,x)}_{\psi_2(\ue,\cdot)}\dxdt\Bigg] \\
	&\quad - E\Bigg[\iint_{\Pi_T} \underbrace{S''(\ue(t,x)-V)
	\int_Z\sigma(x,\ue(t,x),z)D_{t,z}V\,d\mu(z)\test(t,x)}_{\psi_3(\ue,\cdot)}\dxdt\Bigg]\\
	&\quad + \frac{1}{2}E\Bigg[\iint_{\Pi_T}
	\underbrace{S''(\ue(t,x)-V)\int_Z\sigma^2(x,\ue(t,x),z)\,d\mu(z)
	\test(t,x)}_{\psi_4(\ue,\cdot)}\dxdt\Bigg] \\
	&\quad + \varepsilon E\Bigg[\iint_{\Pi_T} S(\ue(t,x)-V)
	\Delta \test(t,x)\dxdt \Bigg] \geq 0.
 \end{split}
\end{equation} 
At this point we may apply Proposition~\ref{proposition:ContDependVisc} to 
relax the assumptions on $f,\sigma, u^0$ to the ones listed 
in Theorem~\ref{theorem:ExistenceOfSolution}, leaving 
the details to the reader.

Next, we wish to send $\varepsilon \downarrow 0$ 
in \eqref{eq:EntIneqApprox}; expressing the limits in 
terms of the function $u$ obtained in Step~1. Obviously,
\begin{displaymath}
 \lim_{\varepsilon \downarrow 0}
 \E{\varepsilon \iint_{\Pi_T} S(\ue(t,x)-V)\Delta \test(t,x)\dxdt} = 0.
\end{displaymath}

For the remaining terms, it suffices by Step 1 
and the upcoming Theorem~\ref{theorem:DunfordPettis} to show 
that $\seq{\psi_i(\ue,\cdot)\phi^{-1}}_{\varepsilon > 0}$ 
is uniformly integrable ($i=1,2,3,4$). In view of 
Lemma~\ref{lemma:UniformIntCriteria}(ii), we must show that
\begin{equation}\label{eq:UniformIntpsii}
\sup_{\varepsilon > 0}\E{\iint_{\Pi_T} 
\abs{\psi_i(\ue(t,x),t,x)\phi^{-1}(x)}^2 \phi(x)\,dxdt} < \infty,
\quad i=1,2,3,4.
\end{equation}
As $S$ is in $\mathscr{E}$ and $\test \in C^\infty_c(\Pi_T)$,
\begin{align*}
	\abs{\psi_1(\ue(t,x),t,x)\phi^{-1}(x)} & 
	= \abs{S(\ue(t,x)-V)\partial_t\test(t,x)\phi^{-1}(x)}^2 \\
	& \leq 2\norm{S}_{\mathrm{Lip}}^2 
	\norm{\partial_t \test(t)}_{\infty,\phi^{-1}}(\abs{\ue(t,x)}^2 + \abs{V}^2).
\end{align*}
So \eqref{eq:UniformIntpsii}, with $i=1$, 
follows from Lemma~\ref{lemma:UniformBoundsOnVisc}. 
The term in \eqref{eq:EntIneqApprox} involving 
$\psi_2$ is treated in the same way. 

Consider the term involving the Malliavin derivative, namely $\psi_3$. 
By \eqref{assumption:LipOnSigma},
$$
\abs{\int_Z\sigma(x,\ue(t,x),z)D_{t,z}V\,d\mu(z)}^2 \leq 
\norm{M}_{L^2(Z)}^2 \norm{D_tV}_{L^2(Z)}^2(1 + \abs{\ue(t,x)})^2.
$$
Recall that $V$ is uniformly bounded and 
also that $\mathrm{supp}\,(S'') \subset (-R,R)$ 
for some $R < \infty$. Hence,
$$
S''(\ue-V)(1 + \abs{\ue}) \leq \norm{S''}_\infty(1 + R + \norm{V}_\infty).
$$
Consequently,
\begin{align*}
	&\E{\iint_{\Pi_T} \abs{\psi_3(\ue(t,x),t,x)\phi^{-1}(x)}^2 \phi(x)\,dxdt} 
	\\ &\, 
	\leq \norm{S''}_\infty^2\norm{M}_{L^2(Z)}^2(1 + R 
	+ \norm{V}_\infty)^2\norm{\test}_{\infty,\phi^{-1}}^2 
	\E{\int_0^T\norm{D_tV}_{L^2(Z)}^2\,dt}\norm{\phi}_{L^1(\R^d)},
\end{align*}
and \eqref{eq:UniformIntpsii} holds with $i=3$. 

Consider the $\psi_4$-term. By \eqref{assumption:LipOnSigma},
\begin{displaymath}
 S''(\ue-V)\int_Z\sigma^2(x,\ue,z)\,d\mu(z) 
 \leq \norm{S''}_\infty\norm{M}_{L^2(Z)}^2(1 + R + \norm{V}_\infty)^2.
\end{displaymath}
Hence
\begin{multline*}
	\E{\iint_{\Pi_T} \abs{\psi_4(\ue(t,x),t,x)\phi^{-1}(x)}^2 \phi(x)\,dxdt} \\
	\leq \norm{S''}_\infty^2\norm{M}_{L^2(Z)}^4(1 + R 
	+ \norm{V}_\infty)^4\norm{\test}_{\infty,\phi^{-1}}^2 \iint_{\Pi_T} \phi(x)\,dxdt. 
\end{multline*}

Summarizing, upon sending $\varepsilon \downarrow 0$ 
along a subsequence, it follows that 
$$
\Y{\Entropy[(S,Q),\test,V]}(u) \geq 0,
$$
where $u$ is the process defined in Step 1. 
Finally, the result follows for general $V \in \D^{1,2}$ by 
the density of $\Sm \subset \D^{1,2}$ 
and Lemma~\ref{lemma:ContinuityOfEntWRTV}.
\end{proof}

\section{Uniqueness of entropy solutions}\label{sec:Uniqueness}
To prove the uniqueness of Young measure-valued entropy solutions, we 
need an additional assumption on $\sigma$: there exists 
$M \in L^2(Z)$ and $0 < \kappa \leq 1/2$ such that 
\begin{equation}\label{assumption:SigmaRegularity}
	\abs{\sigma(x,u,z)-\sigma(y,u,z)} \leq 
	M(z)\abs{x-y}^{\kappa + 1/2}(1 + \abs{u}) \tag{$\mathcal{A}_{\sigma,1}$},
\end{equation} 
for $x,y \in \R^d$ and $u \in \R$. Actually, it suffices 
that the criterion is satisfied locally, i.e., for each compact 
$K \subset \R^d \times \R^d$ there exists $M = M_K$ such 
that \eqref{assumption:SigmaRegularity} is satisfied for all $(x,y) \in K$. 

\begin{theorem}\label{theorem:UniquenessOfEntSol}
Fix $\phi \in \mathfrak{N}$, and suppose 
$u^0 \in L^2(\Omega,\F_0,P;L^2(\R^d,\phi))$. 
Assume that assumptions~\eqref{assumption:LipOnf}, 
\eqref{assumption:LipOnSigma}, \eqref{assumption:SigmaRegularity} 
are satisfied. Let $u$ be the Young measure-valued entropy solution
to \eqref{eq:StochasticBalanceLaw} with initial condition $u^0$ 
obtained in Theorem~\ref{theorem:ExistenceOfSolution}, and 
let $v$ be any Young measure-valued entropy 
solution with initial condition $u^0$ in the sense 
of Definition~\ref{Def:YoungEntropySolution}. Then 
$$
u(t,x,\alpha) = v(t,x,\beta), 
\qquad (t,x,\alpha,\beta,\omega) \mbox{-almost everywhere.} 
$$ 
Consequently, $\hat{u} := \int_0^1 u \,d\alpha$ is the unique 
entropy solution to \eqref{eq:StochasticBalanceLaw} in 
the sense of Definition~\ref{Def:EntropySolution}.
\end{theorem}

The proof is  found at the end of this section. 
As discussed in the introduction, due 
to the lack of Malliavin differentiability at the hyperbolic level, the 
uniqueness argument will invoke the viscous approximations and 
their limit taken in the weak sense of Young measures. 

Retracing the proof of Theorem~\ref{theorem:UniquenessOfEntSol}, making 
some small modifications, we obtain the following spatial regularity result:

\begin{proposition}[Spatial regularity]\label{proposition:FracBounds}
Fix $\phi \in \mathfrak{N}$, and suppose $u^0$ belongs to 
$L^2(\Omega,\F_0,P;L^2(\R^d,\phi))$. 
Under assumptions \eqref{assumption:LipOnf}, \eqref{assumption:LipOnSigma}, 
and \eqref{assumption:SigmaRegularity} the 
entropy solution $u$ to \eqref{eq:StochasticBalanceLaw} satisfies
\begin{multline*}
	E\Bigg[\,\,\iint\limits_{\,\, \R^d \times \R^d}
	\abs{u(t,x + z)-u(t,x-z)}\phi(x)J_r(z)\,dxdz\Bigg] \\
	\leq CE\Bigg[\,\, \iint\limits_{\,\,\R^d \times \R^d}
	\abs{u^0(x + z)-u^0(x-z)}\phi(x)J_r(z)\,dxdz\Bigg] + \mathcal{O}(r^\kappa),
\end{multline*}
where the constant $C$ depends only on $C_\phi, \norm{f}_{\mathrm{Lip}}, T$, 
and $\kappa$ is the exponent from assumption \eqref{assumption:SigmaRegularity}. 
If $\sigma$ is independent of $x$, i.e., $\sigma(x,u,z) = \sigma(u,z)$, then the 
last term on the right vanishes, i.e., $\mathcal{O}(\cdot) \equiv 0$.
\end{proposition}

See \cite{DebusscheVovelle2010,ChenKarlsen2012} for 
similar results, and how to turn this result into a fractional 
$BV$ estimate. The proof of Proposition \ref{proposition:FracBounds} is 
found at the very end of this section.

The next lemma contains the ``entropy condition'' at the 
parabolic level, which is utilized later in the uniqueness proof.

\begin{lemma}\label{lemma:EntIneqViscForAdapted}
For each fixed $\varepsilon>0$, let $\ue$ be the solution of \eqref{eq:ViscousApprox}. 
Suppose $V \in L^2(\Omega)$ is $\F_s$-measurable for some $s \in (0,T)$, 
and $0\le \test \in C^\infty_c([0,T) \times \R^d)$ with 
$\mathrm{supp}\,(\test)\subset (s,T) \times \R^d$. Then
\begin{equation*}
\begin{split}    
  & \qquad \E{\iint_{\Pi_T} S(\ue-V)\partial_t\test + Q(\ue,V)\cdot \nabla \test \dxdt}\\
  & \qquad \qquad \geq-\frac{1}{2}\E{\iint_{\Pi_T}\int_Z S''(\ue-V)
  \sigma(x,\ue,z)^2\test(t,x)\,d\mu(z)\dxdt} \\
  & \qquad \qquad \qquad -\varepsilon \E{\iint_{\Pi_T} S(\ue(t)-V)\Delta\test(t)\dx dt},
  \end{split}
\end{equation*}
for any entropy/entropy-flux pair $(S,Q)$ in $\mathscr{E}$.
\end{lemma}

\begin{proof}
Consider \eqref{eq:EntIneqApprox}. 
Note that for any $V \in \D^{1,2}$ that is $\F_s$-measurable,
\begin{displaymath}
	\E{\iint_{\Pi_T}\int_Z S''(\ue(t)-V)\sigma(x,\ue(t),z)D_{t,z}V\test(t)
	\,d\mu(z)\dx dt} = 0,
\end{displaymath}
thanks to \cite[Proposition~1.2.8]{NualartMalliavinCalc2006}. 
The general result follows by approximation 
as in Lemma~\ref{lemma:ContinuityOfEntWRTV}.
\end{proof}

The following ``doubling of variables" lemma is at the heart of the matter. 
To some extent it may be instructive to compare its proof with the rather involved 
computations in \cite[Lemma~3.2]{FengNualart2008} and \cite[Section~4.1]{Bauzet:2012kx}.

\begin{lemma}\label{lemma:DoubelingWithoutLimits}
Suppose \eqref{assumption:LipOnf}, \eqref{assumption:LipOnSigma} hold. 
Fix $\phi \in \mathfrak{N}$, and let $\seq{\ue}_{\varepsilon > 0}$ be a sequence of viscous 
approximations with initial condition $u^0 \in L^2(\Omega,\F_0,P;L^2(\R^d,\phi))$. 
Let $v$ be a Young measure-valued entropy solution in the sense 
of Definition~\ref{Def:YoungEntropySolution} with initial condition 
$v^0 \in L^2(\Omega,\F_0,P;L^2(\R^d,\phi))$. 

For any $0< \gamma <\frac{1}{2}T$ take $t_0 \in (0,T-2\gamma]$ and define 
$$
\xi_{\gamma,t_0}(t) := 1 - \int_0^t J_\gamma^+(s-t_0)\ds.
$$
Let $\psi \in C^\infty_c(\R^d)$ be non-negative and define
$$
\test(t,x,s,y) = \frac{1}{2^d}\psi\left(\frac{x+y}{2}\right)J_r\left(\frac{x-y}{2}\right)
\xi_{\gamma,t_0}(t)J_{r_0}^+(t-s).
$$
Let $\Sd$ be a function satisfying 
$$
\Sd'(\sigma) = 2\int_0^\sigma J_\delta(z)\dz, 
\qquad \Sd(0) = 0.
$$
Furthermore, define
$$
Q_\delta(u,c) = \int_c^u \Sd'(z-c)f'(z)\dz,
$$
and note that the pair $(\Sd,Q_\delta)$ belongs 
to $\mathscr{E}$. 

Then
\begin{equation}\label{eq:DoublingLemmaIneq}
 L \geq  R + F + \mathcal{T}_1 + \mathcal{T}_2 + \mathcal{T}_3,
\end{equation}
where 
\begin{align*}
\begin{split}
&L = \E{\iint_{\Pi_T}\int_{\R^d} \Sd(v^0(y)-\ue(t,x))\test(t,x,0,y) \,dydxdt}, 
\end{split} \\
&R =  -\E{\iiiint_{\Pi_L^2}\int_{[0,1]}\Sd(v-\ue)(\partial_s + \partial_t)\test\,d\beta dX}, \\ 
&F =  -\E{\iiiint_{\Pi_L^2}\int_{[0,1]} Q_\delta(\ue,v)\cdot \nabla_x \test 
+  Q_\delta(v,\ue)\cdot \nabla_y\test\,d\beta dX}, \\
\begin{split}  
&\mathcal{T}_1 = -\frac{1}{2}\E{ \iiiint_{\Pi_L^2}\int_{[0,1]}\int_Z \Sd''(v-\ue)
\left(\sigma(y,v,z)-\sigma(x,\ue,z)\right)^2\test\,d\mu(z)d\beta dX},
\end{split} \\
&\mathcal{T}_2 = \E{ \iiiint_{\Pi_L^2}\int_{[0,1]}\int_Z \Sd''(v-\ue)\left(D_{s,z}\ue
-\sigma(x,\ue,z)\right)\sigma(y,v,z)\test\,d\mu(z)d\beta dX}, \\
&\mathcal{T}_3 = -\varepsilon \E{\iiiint_{\Pi_L^2}\int_{[0,1]} \Sd(\ue-v)\Delta_x\test\,d\beta dX},
\end{align*}
where $dX = dxdtdyds$.
\end{lemma}

\begin{remark}
In \cite[Section~4.6]{FengNualart2008} the authors prove existence of a 
\textit{strong} entropy solution. The additional condition attached to the notion of 
strong solution stems from the difficulties in sending 
$\varepsilon \downarrow 0$ before $r_0 \downarrow 0$. 
In our setting, the existence of a strong entropy solution amounts to 
showing that we can send $r_0 \downarrow 0$ 
and $\varepsilon \downarrow 0$ simultaneously in such a way that 
$\lim_{(\varepsilon,t) \downarrow (0,s)}\mathcal{T}_2 = 0$. 
This requires a careful study of how the continuity 
properties of \eqref{eq:MallEqSat} depends on 
$\varepsilon$, cf.~Lemma~\ref{lemma:MalliavinDerivativeWeakTimeCont}. 
We do not proceed along this path in this paper, instead we let $r_0 \downarrow 0$ 
before $\varepsilon \downarrow 0$ as in \cite{Bauzet:2012kx}.
\end{remark}

\begin{proof}
Recall that $\mathrm{supp}(J_{r_0}^+) \subset (0,2r_0)$, so $J_{r_0}^+(t-s)$ is zero 
whenever $s \geq t$. Applying Lemma \ref{lemma:EntIneqViscForAdapted} 
with $V = v(s,y,\beta)$ and integrating in $y,s,\beta$, we obtain
\begin{equation}\label{eq:EntIneqForViscConstv}
\begin{split}    
	&\E{\iiiint_{\Pi_T^2}\int_{[0,1]} \Sd(\ue-v)\partial_t\test 
	+ Q_\delta(\ue,v)\cdot \nabla \test \,d\beta dX}\\ 
	&\geq-\frac{1}{2}\E{\iiiint_{\Pi_T^2}\int_{[0,1]}
	\int_Z \Sd''(\ue-v)\sigma(x,\ue,z)^2\test\,d\mu(z)\, d\beta dX}\\
	&\qquad -\varepsilon \E{\iiiint_{\Pi_T^2}\int_{[0,1]} \Sd(\ue-v)\Delta_x\test \,d\beta dX}.
\end{split}
\end{equation}
Similarly, in the entropy inequality for $v = v(s,y,\beta)$ we take $V = \ue(t,x)$ 
and integrate in $t,x$, resulting in    
\begin{equation}\label{eq:EntIneqForMeasureValuedvConstue}
\begin{split}    
	& \E{\iiiint_{\Pi_T^2}\int_{[0,1]} \Sd(v-\ue)\partial_s\test 
	+ Q_\delta(v,\ue)\cdot \nabla_y \test \,d\beta dX} \\ 
	& +\E{\iint_{\Pi_T}\int_{\R^d} \Sd(v^0(y)-\ue(t,x))\test(t,x,0,y) \,dydxdt}\\
	& \geq \E{\iiiint_{\Pi_T^2}\int_{[0,1]}\int_Z \Sd''(v-\ue) 
	D_{s,z}\ue \sigma(y,v,z)\test \,d\mu(z) \,d\beta dX} \\
	& \qquad -\frac{1}{2}\E{ \iiiint_{\Pi_T^2}\int_{[0,1]}
	\int_Z \Sd''(v-\ue)\sigma(y,v,z)^2\test\,d\mu(z)d\beta dX}.
\end{split}
\end{equation}
The result follows by adding \eqref{eq:EntIneqForViscConstv} 
and \eqref{eq:EntIneqForMeasureValuedvConstue}.
\end{proof}

\begin{proposition}[Kato inequality]\label{proposition:KatoInequalityEntSol}
Fix $\phi \in \mathfrak{N}$. Suppose \eqref{assumption:LipOnf}, 
\eqref{assumption:LipOnSigma}, and \eqref{assumption:SigmaRegularity} hold. 
Let $u$ be the Young measure-valued limit of the viscous 
approximations $\seq{\ue}_{\varepsilon > 0}$ with initial 
condition $u^0 \in L^2(\Omega,\F_0,P;L^2(\R^d,\phi))$, 
constructed in Theorem \ref{theorem:ExistenceOfSolution}. 
Let $v$ be a Young measure-valued entropy solution in the 
sense of Definition~\ref{Def:YoungEntropySolution} with 
initial condition $v^0 \in L^2(\Omega,\F_0,P;L^2(\R^d,\phi))$. 
Then, for almost all $t_0 \in (0,T)$ and 
any non-negative $\psi \in C^\infty_c(\R^d)$,
\begin{equation}\label{eq:KatoInequalityEntSol}
\begin{split}
	& E\bigg[\int_{\R^d}\iint_{[0,1]^2}\abs{u(t_0,x,\alpha)-v(t_0,x,\beta)}
	\psi(x)\,d\alpha d\beta dx \bigg] \\
	& \quad \leq \E{\int_{\R^d} \abs{u^0(x)-v^0(x)}\psi(x)\dx} \\
	& \quad \qquad +E\bigg[\int_0^{t_0}\int_{\R^d}\iint_{[0,1]^2} 
	\sign{u(t,x,\alpha)-v(t,x,\beta)} \\
	&\quad\qquad\qquad\qquad
	\times (f(u(t,x,\alpha))-f(v(t,x,\beta)))\cdot 
	\nabla \psi(x)\,d\beta d\alpha dxdt \bigg].
\end{split}
\end{equation}
\end{proposition}

\begin{proof}
Starting off from \eqref{eq:DoublingLemmaIneq}, we send $r_0$ 
and $\varepsilon$ to zero (in that order). 
Next, we send $(\delta,r)$ to $(0,0)$ simultaneously. 
In view of Limits \ref{limit:F} and \ref{limit:T1}, we let $\delta(r) = r^{1 + \eta}$ 
with $0 < \eta <2\kappa -1$. Finally, we send $\gamma \downarrow 0$. 
We arrive at the Kato inequality \eqref{eq:KatoInequalityEntSol} 
thanks to the upcoming Limits~\ref{limit:L}--\ref{limit:T3}. 
\end{proof}
\begin{remark}\label{remark:testFuncProp}
Later we will make repeated use of two elementary 
identities.  Set
$$
\xi_r(x):= \frac{1}{2^d}\int_{\R^d} 
\psi \left(\frac{x+y}{2}\right)J_r\left(\frac{x-y}{2}\right)\,dy.
$$
Then note that $\xi_r = \psi \star J_r$. 
Indeed, making the change of variable $z = (x+y)/2$, 
it follows that $(x-y)/2 = x-z$ and $dy=2^d\,dz$. 
Next, consider the change of variables
\begin{displaymath}
 \Phi(x,y) = \left(\frac{x + y}{2},\frac{x-y}{2}\right) = (\tilde{x},z).
\end{displaymath}
By the change of variables formula 
\begin{displaymath}
 \iint_{\R^d \times \R^d} g(\tilde{x},z)\,d\tilde{x}dz = 
 \iint_{\R^d \times \R^d} g(\Phi(x,y))\abs{\mathrm{det}(\Jac\Phi(x,y))}\,dxdy,
\end{displaymath}
for any measurable function $g(\cdot,\cdot)$. A computation 
yields $\abs{\mathrm{det}(\Jac\Phi(x,y))} = 1/2^d$. It follows that 
\begin{align*}
	\frac{1}{2^d}\iint_{\R^d \times \R^d} &\underbrace{h(x,y)
	\psi \left(\frac{x+y}{2}\right)J_r\left(\frac{x-y}{2}\right)}_{g(\Phi(x,y))}\,dxdy \\
	&= \iint_{\R^d \times \R^d} 
	\underbrace{h(\tilde{x} + z,\tilde{x} - z)\psi(\tilde{x})J_r(z)}_{g(\tilde{x},z)}\,d\tilde{x}dz,
\end{align*}
for any measurable function $h(\cdot,\cdot)$. Most of the 
time we drop the tilde and write $x$ instead of $\tilde{x}$.
\end{remark}

\begin{limit}\label{limit:L}
With $L$ defined in Lemma~\ref{lemma:DoubelingWithoutLimits},
$$
\lim_{r_0 \downarrow 0}L 
= \E{\iint_{\R^d \times \R^d} \Sd(v^0(x - z)-u^0(x + z))\psi(x)J_r(z)\,dxdz}.
$$
If $\delta = \delta(r)$ is a nondecreasing function 
satisfying $\delta(r) \downarrow 0$ 
as $r \downarrow 0$, then
$$
\lim_{\gamma,(\delta,r),\varepsilon,r_0 \downarrow 0} L 
= \E{\norm{v^0-u^0}_{1,\psi}}.
$$
\end{limit}

\begin{proof}
Note that 
\begin{equation}\label{eq:SdAbsEst}
 \abs{\Sd(b)-\Sd(a)} = \abs{\int_a^b\Sd'(z)\dz} \leq \abs{b-a}.
\end{equation}
Furthermore, observe that $\xi_{\gamma,t_0}(t) = 1$ whenever $t \leq t_0$. 
Hence, due to Remark~\ref{remark:testFuncProp},
\begin{align*}
&\abs{L -\E{\int_{\R^d}\int_{\R^d} \Sd(v^0(y)-u^0(x))
\frac{1}{2^d}\psi\left(\frac{x+y}{2}\right)J_r\left(\frac{x-y}{2}\right)\,dxdy}} \\
& \qquad \leq \E{\iint_{\Pi_T} \abs{u^0(x)-\ue(t,x)}J_{r_0}^+(t)(\psi \star J_r)(x)\,dxdt},
\end{align*}
whenever $2r_0 < t_0$. Arguing as in Lemma~\ref{lemma:InitialCondition} for the 
viscous approximation, it follows that
\begin{align*}
	\lim_{r_0 \downarrow 0}L &= \E{\frac{1}{2^d}\int_{\R^d}\int_{\R^d} 
	\Sd(v^0(y)-u^0(x))\psi\left(\frac{x+y}{2}\right)J_r\left(\frac{x-y}{2}\right)\,dxdy} \\
	& = \E{\iint_{\R^d \times \R^d} \Sd(v^0(x - z)-u^0(x + z))\psi(x)J_r(z) \,dxdz}.
\end{align*}
This proves the first limit. The second limit follows by the 
dominated convergence theorem 
and Lemma~\ref{lemma:DeltaRconvergence}. 
\end{proof}
\begin{remark}\label{remark:YoungLimitInDoubling}
To establish Limits~\ref{limit:R} and \ref{limit:F} we 
need to send $\varepsilon \downarrow 0$ in terms of the form 
$$
\E{\,\,\, \iiiint\limits_{\,\,\Pi_T \times \R^d \times [0,1]} 
\Psi(\ue(t,x,\omega),t,x,y,\beta,\omega)\underbrace{\frac{1}{2^d}\phi\left(\frac{x+y}{2}\right)
J_r\left(\frac{x-y}{2}\right)\,d\beta dydxdt}_{d\eta_{\phi,r}}},
$$
where $\Psi$ is continuous in the first variable. Essentially we 
proceed as in the proof of Theorem~\ref{theorem:ExistenceOfSolution}, but 
now the underlying measure space is $\Pi_T \times \R^d \times [0,1] \times \Omega$ 
instead of $\Pi_T \times \Omega$. 
By Lemma~\ref{lemma:UniformBoundsOnVisc} 
and Remark~\ref{remark:testFuncProp},
\begin{multline*}
	\sup_{\varepsilon > 0}\left\{\E{\iiiint_{\Pi_T \times \R^d \times [0,1]}  
	\abs{\ue(t,x)}^2\,d\eta_{\phi,r}}\right\} \\
	= \sup_{\varepsilon > 0}
	\left\{\E{\int_0^T \norm{\ue(t)}_{2,\phi \star J_r}^2
	 \,dt}\right\} < \infty.
\end{multline*}
By Theorem~\ref{theorem:YoungMeasureLimitOfComposedFunc}, there 
exists $\nu \in \Young{\Pi_T \times \R^d \times [0,1] \times \Omega}$ 
such that whenever $\Psi(\ue,\cdot) \rightharpoonup \overline{\Psi}$ (weakly)
along some subsequence in 
$L^1(\Pi_T \times \R^d \times [0,1] \times \Omega,d\eta_{\phi,r} \otimes dP)$, 
\begin{equation*}
	\overline{\Psi} 
	= \int_\R \Psi(\xi,t,x,y,\beta,\omega)\,d\nu_{t,x,\omega}(\xi) 
	= \int_0^1 \Psi(u(t,x,\alpha,\omega),t,x,y,\beta,\omega)\,d\alpha,
\end{equation*}
where $u$ is defined through \eqref{eq:ReprOfProcessByYoung}. 
The fact that $\nu_{t,x,y,\beta,\omega} = \nu_{t,x,\omega}$ comes out since the 
limit is independent of $y,\beta$ when $\Psi$ is independent of $y,\beta$. 
For measurability considerations, see Step 1 in proof 
of Theorem~\ref{theorem:ExistenceOfSolution}.
\end{remark}
\begin{limit}\label{limit:R} With $R$ defined in 
Lemma \ref{lemma:DoubelingWithoutLimits},
$$
\lim_{\gamma,\varepsilon, r_0 \downarrow 0} R 
= E\Bigg[\, \, \iiiint\limits_{\,\,\, \R^d \times \R^d \times [0,1]^2}
\Sd(v(t_0,x-z,\beta)-u(t_0,x+z,\alpha)) \psi(x)J_r(z)\,d\alpha d\beta dxdz\Bigg],
$$
for $dt$-a.a.~$t_0 \in [0,T]$. 
If $\delta = \delta(r)$ is a nondecreasing 
function satisfying $\delta(r) \downarrow 0$ 
as $r \downarrow 0$, then
$$
\lim_{\gamma,(\delta,r),\varepsilon,r_0 \downarrow 0} 
R = \E{\int_{\R^d}\iint_{[0,1]^2}
\abs{v(t_0,x,\beta)-u(t_0,x,\alpha)}\psi(x)\,d\alpha d\beta dx},
$$
for $dt$-a.a.~$t_0 \in [0,T]$.
\end{limit}

\begin{proof}
Since $\partial_tJ_{r_0}^+(t-s) = -\partial_sJ_{r_0}^+(t-s)$ and 
$\partial_t\xi_{\gamma,t_0}(t) = -J_\gamma^+(t-t_0)$,
$$
(\partial_s + \partial_t)\test(t,x,s,y) = -\frac{1}{2^d}\psi\left(\frac{x+y}{2}\right)
J_r\left(\frac{x-y}{2}\right)J_\gamma^+(t-t_0)J_{r_0}^+(t-s).
$$
It follows that
\begin{align*}
	R  = E\Bigg[\frac{1}{2^d}\iiiint_{\Pi_T^2}\int_{[0,1]}&
	\Sd(v-\ue)\psi\left(\frac{x+y}{2}\right)J_r\left(\frac{x-y}{2}\right) \\
	&\qquad \qquad \qquad 
	\times J_\gamma^+(t-t_0)J_{r_0}^+(t-s)\,d\beta\dX\Bigg].
\end{align*}
Thanks to
$$
\abs{\Sd(v-\ue)} \leq \abs{v} + \abs{\ue},
$$
we can apply the dominated convergence theorem 
and Lemma~\ref{lemma:TimeTraceMollLimit}, resulting in 
\begin{multline*}
 \lim_{r_0 \downarrow 0} R = E\Bigg[\iiiint\limits_{\quad \Pi_T \times \R^d \times [0,1]}
 \underbrace{\Sd(v(t,y,\beta)-\ue(t,x))
 J_\gamma^+(t-t_0)(\psi\phi^{-1})\left(\frac{x+y}{2}\right)}_{\Psi(\ue,\cdot)} \\
 \hphantom{XXXXXXXXXXXXXXXX}\times 
 \frac{1}{2^d}\phi\left(\frac{x+y}{2}\right)J_r\left(\frac{x-y}{2}\right)\,d\beta dy dx dt\Bigg].
\end{multline*}
By Lemma~\ref{lemma:UniformIntCriteria}(ii), $\seq{\Psi_\varepsilon(\ue,\cdot)}$ 
is uniformly integrable, and so, cf.~Theorem~\ref{theorem:DunfordPettis}, 
we can extract a weakly convergent subsequence. 
By Remarks~\ref{remark:YoungLimitInDoubling} and \ref{remark:testFuncProp},
\begin{align*}
 \lim_{\varepsilon, r_0 \downarrow 0} R 
  &= E\Bigg[\iint_{\Pi_T}\int_{\R^d} \iint_{[0,1]^2}
  \Sd(v(t,y,\beta)-u(t,x,\alpha))J_\gamma^+(t-t_0) \\
  &\hphantom{XXXXXXXXXXx} \times \frac{1}{2^d}
  \psi\left(\frac{x+y}{2}\right)J_r\left(\frac{x-y}{2}\right)\,d\alpha d\beta  dydxdt \Bigg] \\
  &= E\Bigg[\iint_{\Pi_T}\int_{\R^d} \iint_{[0,1]^2}\Sd(v(t,x-z,\beta)-u(t,x+z,\alpha)) \\
  &\hphantom{XXXXXXXXXXXXXx} 
  \times J_\gamma^+(t-t_0)\psi(x)J_r(z)\,d\alpha d\beta  dzdxdt \Bigg].
\end{align*}
Note that 
$$
\abs{\Sd(a-b)-\Sd(c-d)} \leq \abs{b-d} + \abs{a-c}, 
\qquad a,b,c,d \in \R. 
$$
Applying this inequality and Lemma~\ref{lemma:TimeTraceMollLimit}, we 
can send $\gamma \downarrow 0$ to obtain the first inequality. 
To send $(\delta,r) \downarrow (0,0)$ we apply 
the dominated convergence theorem 
and Lemma~\ref{lemma:DeltaRconvergence}, yielding 
$$
\lim_{r,\varepsilon,r_0 \downarrow 0}R 
= \E{\iint_{\Pi_T}\iint_{[0,1]^2}
\abs{v(t,x,\beta)-u(t,x,\alpha)}
\psi(x)J_\gamma^+(t-t_0)\,d\alpha d\beta\dxdt}.
$$
To send $\gamma \downarrow 0$ we apply 
Lemma~\ref{lemma:TimeTraceMollLimit}. 
This provides the second limit.
\end{proof}

\begin{limit}\label{limit:F}
With $F$ defined in Lemma~\ref{lemma:DoubelingWithoutLimits},
\begin{multline}\label{eq:FluxFracBVLimit}
	\lim_{\gamma,\varepsilon,r_0 \downarrow 0}F
	= E\Bigg[\int_0^{t_0}\iiiint\limits_{\R^d \times \R^d \times [0,1]^2}
	\Sd'(u(t,x+z,\alpha)-v(t,x-z,\beta)) \\
	\times (f(u(t,x + z,\alpha))-f(v(t,x-z,\beta)))\cdot \nabla \psi(x)J_r(z)
	\,d\alpha d\beta dxdzdt\Bigg] \\
	+ \mathcal{O}\left(\delta + \frac{\delta}{r}\right).
\end{multline}
If $\delta:[0,\infty) \rightarrow [0,\infty)$ satisfy 
$\lim_{r \rightarrow 0}\frac{\delta(r)}{r} = 0$, then
\begin{multline}\label{eq:FluxLimit}
	\lim_{\gamma,r,\varepsilon,r_0 \downarrow 0}F 
	= -E\bigg[\int_0^{t_0}\int_{\R^d}\iint_{[0,1]^2} \sign{u(t,x,\alpha)-v(t,x,\beta)} \\
	\times (f(u(t,x,\alpha))-f(v(t,x,\beta)))\cdot \nabla \psi(x)\,d\alpha d\beta dxdt\bigg].
\end{multline}
\end{limit}

\begin{proof}
Using integration by parts,
\begin{displaymath}
 Q_\delta(\ue,v) = \Sd'(\ue-v)(f(\ue)-f(v)) - \int_v^{\ue} \Sd''(z-v)(f(z)-f(v))\dz
\end{displaymath}
and 
\begin{displaymath}
 Q_\delta(v,\ue) = \Sd'(v-\ue)(f(v)-f(\ue)) - \int_{\ue}^v \Sd''(z-\ue)(f(z)-f(\ue))\dz.
\end{displaymath}
Due to the symmetry of $\Sd$,
\begin{align*}
	F &= -\E{\iiiint_{\Pi_T^2}\int_{[0,1]} Q_\delta(\ue,v)\cdot \nabla_x \test 
	+  Q_\delta(v,\ue)\cdot \nabla_y \test\,d\beta dX} \\
	&=-\E{\iiiint_{\Pi_T^2}\int_{[0,1]} \Sd'(\ue-v)(f(\ue)-f(v))\cdot(\nabla_x 
	+ \nabla_y)\test \,d\beta dX} \\
	&\quad + \E{\iiiint_{\Pi_T^2}\int_{[0,1]}\left(\int_v^{\ue} \Sd''(z-v)(f(z)-f(v))\dz\right)
	\cdot \nabla_x \test\,d\beta dX} \\ 
	& \quad + \E{\iiiint_{\Pi_T^2}\int_{[0,1]}\left(\int_{\ue}^v 
	\Sd''(z-\ue)(f(z)-f(\ue))\dz\right)\cdot \nabla_y \test\,d\beta dX} \\
	&= -F_1 + F_2 + F_3,
\end{align*}
where $dX = dxdtdyds$ as in Lemma~\ref{lemma:DoubelingWithoutLimits}.
Note that 
\begin{equation}\label{est:SingTimesLipInt}
\abs{\int_v^u \Sd''(z-v)(f(z)-f(v))\dz} 
\leq \norm{f}_{\mathrm{Lip}}\delta, 
\qquad u,v\in \R.
\end{equation}
To see this, recall that 
$\Sd''(\sigma) = 2J_\delta(\sigma)$. 
By \eqref{assumption:LipOnf}, 
$$
\abs{\int_v^u \Sd''(z-v)(f(z)-f(v))\dz} 
\leq  2\norm{f}_{\mathrm{Lip}}\sign{u-v}
\int_v^u J_\delta(z-v)\abs{z-v}\dz,
$$
and letting $\xi = \abs{z-v}/\delta$,
$$
\sign{u-v}\int_v^u J_\delta(z-v)\abs{z-v}\dz 
= \delta \int_0^{\delta^{-1}\abs{u-v}}J(\xi)\xi\,d\xi 
\leq \frac{\delta}{2}.
$$

In view of \eqref{est:SingTimesLipInt}, it is clear that
$$
F_2 \leq \norm{f}_{\mathrm{Lip}}\delta 
\iiiint_{\Pi_T^2}\abs{\nabla_x \test} \dX.
$$
A computation shows $\norm{\nabla \test}_{L^1(\Pi_T^2)} \leq C(1 + r^{-1})$, 
for some constant $C$ depending only on $J,T,\psi$. 
Consequently,
$$
F_2 \leq C\norm{f}_{\mathrm{Lip}} 
\delta\left(1 + \frac{1}{r}\right).
$$
The same type of estimate applies to $F_3$. 

Let us consider $F_1$. Observe that   
$$
 (\nabla_x + \nabla_y)\test(t,x,s,y) 
 = \frac{1}{2^d}\nabla \psi\left(\frac{x+y}{2}\right)
 J_r\left(\frac{x-y}{2}\right)\xi_{\gamma,t_0}(t)J_{r_0}^+(t-s).
$$
For $\delta>0$, define
$$
\mathcal{F}_\delta(a,b) := \Sd'(a-b)(f(a)-f(b)), \qquad a,b \in \R,
$$
and note that $(t,b) \mapsto \mathcal{F}_\delta(\ue(t,x),b)$ 
obeys the hypotheses of Lemma \ref{lemma:TimeTraceMollLimit}. 
By the dominated convergence theorem 
and Lemma \ref{lemma:TimeTraceMollLimit},
\begin{multline*}
	\lim_{r_0 \downarrow 0}F_1 
	= E\Bigg[\, \, \iint_{\Pi_T}\int_{\R^d}\int_{[0,1]} 
	\underbrace{\mathcal{F}_\delta(\ue(t,x),v(t,y,\beta))\cdot \zeta
	\left(\frac{x+y}{2}\right)\xi_{\gamma,t_0}(t)}_{\Psi(\ue,\cdot)} \\
	\times \frac{1}{2^d}\phi\left(\frac{x+y}{2}\right)J_r\left(\frac{x-y}{2}\right)
	\,d\beta dydxdt\Bigg],
\end{multline*}
where $\zeta(x) = \phi^{-1}(x)\nabla \psi(x)$. 
The uniform integrability of $\seq{\Psi(\ue,\cdot)}_{\varepsilon > 0}$ 
follows thanks to Lemma~\ref{lemma:UniformIntCriteria}(ii). 
Indeed, $\abs{\zeta} \leq C_\phi$ 
and $\abs{\mathcal{F}_\delta(\ue,v)} \leq 
\norm{f}_{\mathrm{Lip}}\abs{\ue-v}$, so
$$
\abs{\Psi(\ue,\cdot)}^2 \leq 2C_\phi^2\norm{f}_{\mathrm{Lip}}^2
(\abs{\ue}^2+ \abs{v}^2).
$$
By Theorem~\ref{theorem:DunfordPettis} and Remark~\ref{remark:YoungLimitInDoubling}
\begin{multline*}
	\lim_{\varepsilon,r_0 \downarrow 0}F_1 
	= E\Bigg[\iint_{\Pi_T}\int_{\R^d}\iint_{[0,1]^2} 
	\mathcal{F}_\delta(u(t,x,\alpha),v(t,y,\beta)) \\
	\cdot \frac{1}{2^d}\nabla \psi\left(\frac{x+y}{2}\right)
	J_r\left(\frac{x-y}{2}\right)\xi_{\gamma,t_0}(t)
	\, d\alpha d\beta dydxdt \Bigg],
\end{multline*}
along a subsequence. Sending $\gamma \downarrow 0$, 
applying Remark \ref{remark:testFuncProp}, yields \eqref{eq:FluxFracBVLimit}. 

Next we want to prove \eqref{eq:FluxLimit}. To send $r \downarrow 0$ we apply 
Lemma \ref{lemma:DeltaRconvergence}. It is easily 
verified that condition (i) and (iii) are satisfied with $F_\delta = \mathcal{F}_\delta$. 
Consider condition (ii). Since $\Sd'(-\sigma) = -\Sd'(\sigma)$ for all $\sigma \in \R$, it follows that 
\begin{align*}
 \mathcal{F}_\delta(a,b) - \mathcal{F}_\delta(a,c) 
  &= \Sd'(b-a)(f(b)-f(a)) - \Sd'(c-a)(f(c)-f(a))\\
  &= \int_b^c \partial_z(\Sd'(z-a)(f(z)-f(a)))\dz \\
  &= \int_b^c \Sd''(z-a)(f(z)-f(a)))\dz + \int_b^c \Sd'(z-a)f'(z)\dz,
\end{align*}
for $a,b,c \in \R$. By \eqref{est:SingTimesLipInt},
\begin{align*}
 \abs{\mathcal{F}_\delta(a,b) - \mathcal{F}_\delta(a,c)} &\leq  
    \underbrace{\abs{\int_b^c \Sd''(z-a)(f(z)-f(a)))\dz}}_{\leq \norm{f}_{\mathrm{Lip}}2\delta} 
    + \underbrace{\abs{ \int_b^c \Sd'(z-a)f'(z)\dz}}_{\leq \norm{f}_{\mathrm{Lip}}\abs{b-c}}.
\end{align*}
This and the symmetry of $\mathcal{F}_\delta$, i.e., $\mathcal{F}_\delta(a,b) 
= \mathcal{F}_\delta(b,a)$ for $a,b \in \R$, yields condition (ii). 
Hence, by Lemma~\ref{lemma:DeltaRconvergence},
\begin{align*}
	\lim_{(\delta,r),\varepsilon,r_0 \downarrow 0} F_1 
	& = E\bigg[\,\, 
	\iint_{\Pi_T}\iint_{[0,1]^2} 
	\sign{u(t,x,\alpha)-v(t,x,\beta)} 
	\\ & \qquad \qquad
	\times (f(u(t,x,\alpha))-f(v(t,x,\beta)))
	\cdot \nabla\psi(x)\xi_{\gamma,t_0}(t)
	\,d\beta d\alpha dxdt \bigg].
\end{align*}
At long last, Limit \eqref{eq:FluxLimit} follows 
by sending $\gamma \downarrow 0$.
\end{proof}

\begin{limit}\label{limit:T1}
Suppose assumptions \eqref{assumption:SigmaRegularity} 
and \eqref{assumption:LipOnSigma} hold. 
With $\mathcal{T}_1$ defined in Lemma~\ref{lemma:DoubelingWithoutLimits}, 
$$
\mathcal{T}_1 = \mathcal{O}
\left(\frac{r^{2\kappa + 1}}{\delta} + \delta\right).
$$
If $\sigma$ is independent of $x$, i.e., $\sigma(x,u,z) = \sigma(u,z)$, then 
$\mathcal{T}_1 = \mathcal{O}(\delta)$.
\end{limit}

\begin{proof} 
By assumption \eqref{assumption:SigmaRegularity} 
and \eqref{assumption:LipOnSigma},
$$
\abs{\sigma(y,v,z)-\sigma(x,\ue,z)} 
\leq  M(z)\abs{y-x}^\kappa(1 + \abs{\ue}) + M(z)\abs{v-\ue}.
$$
and thus
\begin{align*}
	\abs{\mathcal{T}_1} 
	&= \frac{1}{2}\E{ \iiiint_{\Pi_T^2}\int_{[0,1]}\int_Z 
	\Sd''(v-\ue)\left(\sigma(y,v,z)-\sigma(x,\ue,z)\right)^2\test\,d\mu(z)d\beta\dX} \\
	& \leq \norm{M}_{L^2(Z)}^2\E{\iiiint_{\Pi_T^2}\int_{[0,1]}
	\Sd''(v-\ue)\abs{x-y}^{2\kappa + 1}(1 +\abs{\ue})^2\test\,d\beta\dX} \\ 
	&\qquad +\norm{M}_{L^2(Z)}^2\E{\iiiint_{\Pi_T^2}\int_{[0,1]}
	\Sd''(v-\ue)\abs{v-\ue}^2\test\,d\beta\dX} \\
	&=:\mathcal{T}_1^1 + \mathcal{T}_1^2.
\end{align*}
Since $J_r(\frac{x-y}{2}) = 0$ whenever $\abs{x-y} \geq 2r$, 
$$
\mathcal{T}_1^1 \leq 4\norm{M_K}_{L^2(Z)}^2
\norm{J}_\infty\frac{r^{2\kappa + 1}}{\delta}
\E{\iiiint_{\Pi_T^2}(1 + \abs{\ue})^2\test \dX}.
$$
Moreover, as
$$
\E{\iiiint_{\Pi_T^2}(1 + \abs{\ue})^2\test \dX} 
\leq \int_0^T  \E{\norm{1 + \ue(t)}_{2,\psi \star J_r}^2}\,dt,
$$
there is a constant $C > 0$, independent of $r_0, \varepsilon, \delta, \gamma, r$,
such that $\mathcal{T}_1^1 \leq Cr^{2\kappa + 1}\delta^{-1}$. 
Regarding the second term $\mathcal{T}_1^2$, observe that 
$$
\Sd''(v-\ue)\abs{v-\ue}^2 = J_\delta(v-\ue)\abs{v-\ue}^2 
\leq 2\norm{J}_\infty \delta.
$$
Hence, $\mathcal{T}_1^2 \leq \delta 2
\norm{J}_\infty\norm{M}_{L^2(Z)}^2\norm{\test}_{L^1(\Pi_T^2)}$. 
Regarding the case $\sigma(x,u,z) = \sigma(u,z)$, observe that $\mathcal{T}_1^1 = 0$. 
\end{proof}

Let us consider the term involving the Malliavin derivative.
\begin{limit}\label{limit:T2}
With $\mathcal{T}_2$ defined in 
Lemma~\ref{lemma:DoubelingWithoutLimits},
 \begin{displaymath}
  \lim_{r_0 \downarrow 0}\mathcal{T}_2 = 0.
 \end{displaymath}
\end{limit}

\begin{proof}
Let us split $\mathcal{T}_2$ as follows:
\begin{align*}
	\mathcal{T}_2 & = E\Bigg[\iiiint_{\Pi_T^2}\int_{[0,1]}
	\int_Z \Sd''(v-\ue(s,x))\bigg(D_{s,z}\ue(t,x)-\sigma(x,\ue(s,x),z)\bigg) \\
	& \qquad\qquad\qquad \qquad \qquad \qquad\qquad\qquad\qquad
	\times \sigma(y,v,z)\test\,d\mu(z)d\beta dX\Bigg] \\
	&\qquad 
	+ E\Bigg[\iiiint_{\Pi_T^2}\int_{[0,1]}\int_Z \bigg(\Sd''(v-\ue(t,x))-\Sd''(v-\ue(s,x))\bigg) \\
	& \qquad\qquad\qquad \qquad \qquad \qquad\qquad\qquad\qquad
	\times D_{s,z}\ue(t,x)\sigma(y,v,z)\test\,d\mu(z)d\beta dX\Bigg] \\
	&\qquad 
	+ E\Bigg[\iiiint_{\Pi_T^2}\int_{[0,1]}\int_Z  \Sd''(v-\ue(s,x))
	\bigg(\sigma(x,\ue(s,x),z)-\sigma(x,\ue(t,x),z)\bigg)\\
	& \qquad\qquad\qquad \qquad \qquad \qquad\qquad\qquad\qquad
	\times \sigma(y,v,z)\test\,d\mu(z)d\beta dX\Bigg] \\
	&\qquad + E\Bigg[\iiiint_{\Pi_T^2}\int_{[0,1]}\int_Z  \bigg(\Sd''(v-\ue(s,x))-\Sd''(v-\ue(t,x))\\
	& \qquad\qquad\qquad \qquad \qquad \qquad\qquad\qquad\qquad
	\times \sigma(x,\ue,z)\sigma(y,v,z)\test\,d\mu(z)d\beta dX\Bigg] \\
	&=: \mathcal{T}_2^1 + \mathcal{T}_2^2 + \mathcal{T}_2^3 + \mathcal{T}_2^4.
\end{align*}

Consider $\mathcal{T}_2^1$. We want to apply 
Lemma~\ref{lemma:MalliavinDerivativeWeakTimeCont} 
for fixed $(s,y,\beta)$ with 
\begin{align*}
	&\Psi_{s,y,\beta}(x,z)  = \Sd''(v-\ue(s,x))\sigma(y,v,z), \\
	&\phi_y(x)  = \frac{1}{2^d}\psi\left(\frac{x+y}{2}\right)
	J_r\left(\frac{x-y}{2}\right).
\end{align*}
Then 
\begin{displaymath}
	\mathcal{T}_2^1 = \iint_{\Pi_T}\int_0^1 
	\mathcal{T}_{r_0}(\Psi_{s,y,\beta}) \,d\beta dsdy.
\end{displaymath}
By means of Lemma \ref{lemma:MalliavinDerivativeWeakTimeCont}, 
$\lim_{r_0 \downarrow 0} \mathcal{T}_{r_0}(\Psi_{s,y,\beta}) = 0$ 
$dsdyd\beta$-a.e., and so 
$\lim_{r_0 \downarrow 0}\mathcal{T}_2^1=0$ 
by the dominated convergence theorem. To this end, in view of 
\eqref{eq:UniformBoundOnTr0}, there 
exists a constant $C$, not depending on $r_0$, such that
\begin{align*}
	\abs{\mathcal{T}_{r_0}(\Psi_{s,y,\beta})}^2 
	&\leq C^2\E{\iint_{Z \times \R^d}
	\abs{\Psi_{s,y,\beta}(x,z)}^2\phi_y(x)\,dx\,d\mu(z)} \\
	&\leq C^2\norm{\Sd''}_\infty^2
	\E{\int_{Z}\abs{\sigma(y,v,z)}^2(\psi \star J_r)(y)\,d\mu(z)} \\
	& \leq C^2\norm{\Sd''}_\infty^2
	\norm{M}_{L^2(Z)}^2\E{(1 + \abs{v})^2(\psi \star J_r)(y)}.
\end{align*}
Due to the compact support of $\psi \star J_r$, we see that 
$\abs{\mathcal{T}_{r_0}(\Psi_{s,y,\beta})}$ is dominated by an integrable function.

Let us consider $\mathcal{T}_2^2$. Note that
\begin{align*}
	& \abs{\Sd''(v-\ue(t,x))-\Sd''(v-\ue(s,x))} \\ 
	& \qquad 
	\leq \underbrace{\max \seq{2\norm{\Sd''}_\infty,\norm{\Sd''}_{\mathrm{Lip}}
	\abs{\ue(t,x)-\ue(s,x)}}}_{\Psi}.
\end{align*}
By H\"older's inequality,
\begin{align*} 
	\mathcal{T}_2^2 &\leq E\Bigg[\iiiint_{\Pi_T^2}\int_{[0,1]}
	\int_Z \Psi^2(s,t,x)\abs{\sigma(y,v,z)}^2
	\test\,d\mu(z)d\beta dX\Bigg]^{1/2} \\
	& \qquad \quad
	\times E\Bigg[\iiiint_{\Pi_T^2}\int_{[0,1]}\int_Z 
	\abs{D_{s,z}\ue(t,x)}^2\test\,d\mu(z)d\beta dX\Bigg]^{1/2} \\
	& =: F_1 \times F_2.
\end{align*}
By the uniform boundedness of $\Psi$ we can
apply the dominated convergence theorem and 
Lemma~\ref{lemma:TimeTraceMollLimit}, to conclude 
that $\lim_{r_0 \downarrow 0}F_1 = 0$. 
It remains to show that $\abs{F_2} \leq C$, with
$C$ independent of $r_0 > 0$. We deduce easily
\begin{align*}
	F_2^2 &= \int_0^T\int_0^T\E{\norm{D_{s}\ue(t)}_{L^2(Z;L^2(\R^d,\psi \star J_r))}^2}
	J_{r_0}^+(t-s)\xi_{\gamma,t_0}(t) dsdt \\
       & \leq \int_0^T \sup_{t \in [0,T]}
       \seq{\E{\norm{D_{s}\ue(t)}_{L^2(Z;L^2(\R^d,\psi \star J_r))}^2}}\,ds,
\end{align*}
and so $\abs{F_2}$ is uniformly bounded by \eqref{eq:supInrMallEst}. 

Consider $\mathcal{T}_2^3$. By H\"older's 
inequality and \eqref{assumption:LipOnSigma},
\begin{align*}
	\abs{\mathcal{T}_2^3} &\leq \norm{\Sd''}_\infty \norm{M}_{L^2(Z)}
	E\Bigg[\iiiint_{\Pi_T^2}\int_{[0,1]}\int_Z 
	\abs{\sigma(y,v,z)}^2\test\,d\mu(z)d\beta dX\Bigg]^{1/2} \\
	&\quad\qquad \times 
	E\Bigg[\iiint\limits_{\quad \R^d \times [0,T]^2}
	\abs{\ue(s,x)-\ue(t,x)}^2(\psi \star J_r)(x)J_{r_0}^+(t-s)\,dxdtds\Bigg]^{1/2}.
\end{align*}
By the dominated convergence theorem and Lemma~\ref{lemma:TimeTraceMollLimit}, 
$\lim_{r_0 \downarrow 0}\mathcal{T}_2^3 = 0$.

The term $\mathcal{T}_2^4$ is treated in the same manner as $\mathcal{T}_2^2$, resulting 
in $\lim_{r_0 \downarrow 0}\mathcal{T}_2^4 = 0$. 
\end{proof}

\begin{limit}\label{limit:T3}
With $\mathcal{T}_3$ defined in Lemma~\ref{lemma:DoubelingWithoutLimits},
$$
\mathcal{T}_3 = \mathcal{O}(\varepsilon).
$$
\end{limit}
\begin{proof}
Note that 
$$
\abs{\Sd(\ue-v)\Delta_x\test} \leq (\abs{\ue} + \abs{v})\abs{\Delta_x\test}.
$$
Using this inequality, it follows from 
Lemma~\ref{lemma:UniformBoundsOnVisc} that 
$$
\E{\iiiint_{\Pi_T^2}\int_{[0,1]}
\Sd(\ue-v)\Delta_x\test \,d\beta dX} \leq C,
$$
for some constant $C > 0$ independent of $\varepsilon$ and $r_0$.
\end{proof}

Having established Proposition \ref{proposition:KatoInequalityEntSol}, the proof 
of Theorem \ref{theorem:UniquenessOfEntSol} follows easily. 

\begin{proof}[Proof of Theorem~\ref{theorem:UniquenessOfEntSol}]
In the setting of Proposition~\ref{proposition:KatoInequalityEntSol},
suppose $u^0 = v^0$. Let $\seq{\phi_R}_{R > 1}$ be as in 
Lemma~\ref{lemma:NSmCompApprox} and take 
$\psi = \phi_R$ in \eqref{eq:KatoInequalityEntSol}. 
Exploiting  that $\phi$ belongs to $\mathfrak{N}$, 
sending $R \rightarrow \infty$ yields
$$
\eta(t_0) \leq 
C_\phi\norm{f}_{\mathrm{Lip}}\int_0^{t_0}\eta(t)\,dt,
$$
where
$$
\eta(t) =E\Bigg[\,\,\iiint\limits_{\;\; \R^d \times [0,1]^2}
\abs{u(t,x,\alpha)-v(t,x,\beta)}\phi(x)\,d\beta d\alpha dx\Bigg].
$$
An application of Gr\"onwall's inequality gives $\eta(t) = 0$ for a.a.~$t \in [0,T]$. 
Hence $u(t,x,\alpha) = v(t,x,\beta)$ $(t,x,\alpha,\beta,\omega)$-almost everywhere.
\end{proof}

\begin{proof}[Proof of Proposition~\ref{proposition:FracBounds}]
Let $\seq{\phi_R}_{R > 1}$ be as in Lemma~\ref{lemma:NSmCompApprox}, and start off from 
Lemma \ref{lemma:DoubelingWithoutLimits} with $\psi = \phi_R$ and $v^0 = u^0$. 
We then compute the limits $r_0 \downarrow 0$, $\varepsilon \downarrow 0$, 
and $\gamma \downarrow 0$ (in that order). Recall that by 
Theorem \ref{theorem:UniquenessOfEntSol}, $v = u$ with 
$u =\lim_{\varepsilon \downarrow 0}\ue$. 
Furthermore, $u$ is a solution according to Definition~\ref{Def:EntropySolution}. 
Due to Limits~\ref{limit:L}--\ref{limit:T3} we arrive at the inequality
\begin{equation}\label{eq:FracIneq}
\begin{split}
	&E\Bigg[\,\,\iint\limits_{\,\,\,\R^d \times \R^d} 
	\Sd(u^0(x-z)-u^0(x + z))\phi_R(x)J_r(z)\,dxdz\Bigg] \\
	&\quad
	 \geq E\Bigg[\,\, \iint\limits_{\,\,\, \R^d \times \R^d}
	\Sd(u(t_0,x-z)-u(t_0,x+z)) \phi_R(x)J_r(z)\,dxdz\Bigg] \\
	& \quad\qquad 
	+ E\Bigg[\,\, \int_0^{t_0}
	\iint\limits_{\R^d \times \R^d}\Sd'(u(t,x+z)-u(t,x-z)) \\
	&\quad\quad \qquad\qquad\qquad 
	\times (f(u(t,x + z))-f(u(t,x-z)))\cdot \nabla \phi_R(x)J_r(z)\,dxdzdt\Bigg] \\
	&\quad\qquad\qquad 
	+ \mathcal{O}\left(\delta + \frac{\delta}{r} 
	+ \frac{r^{2\kappa + 1}}{\delta}\right),
\end{split}
\end{equation}
where $\mathcal{O}(\cdot)$ is independent
 of $R$, cf.~Limits \ref{limit:F} and \ref{limit:T1} 
and Lemmas \ref{lemma:ContMollWeightedNorm} 
and \ref{lemma:NSmCompApprox}. 

Note that
$$
\abs{\Sd(\sigma)-\abs{\sigma}} 
\leq \delta, \qquad 
\forall \sigma \in \R,
$$
and $\abs{\nabla \phi} \leq C_\phi\phi$.  With the help of 
Lemma~\ref{lemma:NSmCompApprox}, we can 
now send $R \rightarrow \infty$ in \eqref{eq:FracIneq}, obtaining
$$
\eta(t_0) \leq \eta(0) + C_\phi \norm{f}_{\mathrm{Lip}}\int_0^{t_0}\eta(t)\,dt 
+ \mathcal{O}\left(\delta + \frac{\delta}{r} + \frac{r^{2\kappa + 1}}{\delta}\right),
$$
where
$$
\eta(t) = E\Bigg[\,\, \iint\limits_{\,\,\,\R^d \times \R^d} 
\abs{u(t,x-z)-u(t,x + z)}\phi(x)J_r(z)\,dxdz\Bigg].
$$
By Gr\"onwall's inequality,
$$
\eta(t) \leq  \left(1 + C_\phi 
\norm{f}_{\mathrm{Lip}}te^{C_\phi \norm{f}_{\mathrm{Lip}}t}\right)\left(\eta(0) + 
\mathcal{O}\left(\delta + \frac{\delta}{r} + \frac{r^{2\kappa+1}}{\delta}\right)\right).
$$
Prescribing $\delta = r^{\kappa + 1}$ concludes the proof. 
Regarding the case $\sigma(x,u,z) = \sigma(u,z)$, observe that by Limit~\ref{limit:T1} we 
may replace $\mathcal{O}\left(\delta + \frac{\delta}{r} + \frac{r^{2\kappa+1}}{\delta}\right)$ 
by $\mathcal{O}\left(\delta + \frac{\delta}{r}\right)$ in the 
above argument. The result follows by letting $\delta \downarrow 0$.
\end{proof}

\section{Appendix}\label{sec:Appendix}

\subsection{Some ``doubling of variables" tools} 

\begin{lemma}\label{lemma:DeltaRconvergence}
Suppose $u,v \in L^1_{\mathrm{loc}}(\R^d)$ and $\seq{F_\delta}_{\delta > 0}$ satisfy:
\begin{itemize}
	\item[(i)] There is $F: \R^2 \rightarrow \R$ such 
	that $F_\delta \rightarrow F$ pointwise as $\delta \downarrow 0$.
	\item[(ii)] There exists a constant $C > 0$ such that
	$$
	\abs{F_\delta(a,b)-F_\delta(c,d)} \leq C(\abs{a-c} + \abs{b-d} + \delta),
	$$
	for all $a,b,c,d \in \R$ and all $\delta > 0$.
	\item[(iii)] There is a constant $C > 0$ such that
	$$
	\abs{F_\delta(a,a)} \leq C(1 + \abs{a}) \mbox{ for all $\delta > 0$.}
	$$
 \end{itemize}
Fix $\psi \in C_c(\R^d)$. Suppose 
$\delta:[0,\infty) \rightarrow [0,\infty)$ 
satisfies $\delta(r) \downarrow 0$ as 
$r \downarrow 0$. Set
\begin{align*}
	\mathcal{T}_r &:= \int_{\R^d}\int_{\R^d} F_{\delta(r)}(u(x),v(y))
	\frac{1}{2^d}\psi\left(\frac{x+y}{2}\right)J_r\left(\frac{x-y}{2}\right)\,dydx \\
	& \qquad\qquad \qquad 
	-\int_{\R^d} F(u(x),v(x))\psi(x)\dx.
\end{align*}
Then $\mathcal{T}_r \rightarrow 0$ as $r \downarrow 0$.
\end{lemma}

\begin{proof}
Due to Remark~\ref{remark:testFuncProp},
\begin{align*}
	\mathcal{T}_r & = \int_{\R^d}\underbrace{\int_{\R^d} 
	F_{\delta(r)}(u(x+z),v(x-z))\psi(x)\,dx}_{g_\delta(z)}J_r(z)\,dz
	\\ & \qquad\qquad
	-\underbrace{\int_{\R^d} F(u(x),v(x))\psi(x)\,dx}_{g(0)}.
\end{align*}

Suppose for the moment that given a number $\varepsilon > 0$, 
there exists two numbers $\eta=\eta(\varepsilon) > 0$ and 
$\delta=\delta_0(\varepsilon) > 0$ such that
\begin{equation}\label{eq:gEquicontLimAt0}
	\abs{g_\delta(z)-g(0)} \leq \varepsilon, 
	\mbox{ whenever } \abs{z} \leq \eta \mbox{ and } \delta < \delta_0.
\end{equation}

The change of variables $z = r\zeta$ yields
$$
\abs{\mathcal{T}_r}  \leq \int_{\R^d}\abs{g_\delta(z)-g(0)}J_r(z)\,dz 
= \int_{\R^d}\abs{g_\delta(r\zeta)-g(0)}J(\zeta)\,d\zeta.
$$
Fix $\varepsilon > 0$, and pick $\eta, \delta_0$ as 
dictated by \eqref{eq:gEquicontLimAt0}. Let $r_0 > 0$ satisfy 
$r_0 \leq \eta$ and $\delta(r_0) \leq \delta_0$. 
It follows by \eqref{eq:gEquicontLimAt0} that 
$\abs{\mathcal{T}_{r_0}} \leq \varepsilon$. 
Hence, $\mathcal{T}_r \downarrow 0$ as $r \downarrow 0$. 

Let us now prove \eqref{eq:gEquicontLimAt0}. 
By assumption (ii),
\begin{align*}
	\abs{g_\delta(z)-g(0)} 
	& \leq C\int_{\R^d} \abs{u(x+z)-u(x)}\psi(x)\,dx 
	+C\int_{\R^d} \abs{v(x-z)-v(x)}\psi(x)\,dx \\
	& \quad\quad 
	+\int_{\R^d} \abs{F_\delta(u(x),v(x))-F(u(x),v(x))}\psi(x)\,dx 
	+ C\delta \norm{\psi}_{L^1(\R^d)}. 
\end{align*}
Because of assumptions (i) and (iii), we can apply the 
dominated convergence theorem to conclude that
$$
\lim_{\delta \downarrow 0}\int_{\R^d} 
\abs{F_\delta(u(x),v(x))-F(u(x),v(x))}\psi(x)\,dx = 0.
$$
It remains to show that 
\begin{equation}\label{eq:ContAt0IntTrans}
	\lim_{z \rightarrow 0}\int_{\R^d} 
	\abs{u(x+z)-u(x)}\psi(x)\,dx = 0.
\end{equation}
The term involving $v$ follows by the same argument. 
Pick a compact $K \subset \R^d$ such that $\bigcup_{\abs{z} \leq 1}
\mathrm{supp}\,(\psi(\cdot + z)) \subset K$. 
Fix $\varepsilon > 0$. By the density of continuous functions in $L^1(K)$, 
we can find $w \in C(K)$ such that 
$\norm{w-u}_{L^1(K)} \leq \varepsilon$. 
Then
$$
\int_{\R^d} \abs{u(x+z)-u(x)}\psi(x)\,dx \leq 
2\norm{\psi}_\infty\varepsilon 
+ \int_{\R^d}\abs{w(x+z)-w(x)}\psi(x)\,dx,
$$
for any $\abs{z} \leq 1$. 
Next  we send $z \rightarrow 0$. 
The claim \eqref{eq:ContAt0IntTrans} follows 
by the dominated convergence theorem and the 
arbitrariness of $\varepsilon > 0$.
\end{proof}

\begin{lemma}\label{lemma:TimeTraceMollLimit}
Let  $v \in L^p([0,T])$, $1 \leq p < \infty$. Moreover,
Let $F:[0,T] \times \R \rightarrow \R$ be measurable 
in the first variable and Lipschitz in the second variable,
$$
\abs{F(s,a)-F(s,b)} 
\leq C\abs{a-b}, \qquad \forall a,b \in \R, \forall s \in [0,T],
$$
for some constant $C > 0$. Set
$$
\mathcal{T}_{r_0}(s) = \left(\int_0^T 
\abs{F(s,v(t))-F(s,v(s))}^pJ_{r_0}^+(t-s)\,dt\right)^{1/p}.
$$
Then $\mathcal{T}_{r_0}(s) \rightarrow 0$ 
$ds$-a.e.~as $r_0 \downarrow 0$.
\end{lemma}

\begin{proof}
We can write $v = v_1^n + v_2^n$ with $v_1^n$ 
continuous and $\norm{v_2^n}_{L^p([0,T])} \leq 1/n$. 
This is possible since the continuous functions are dense in $L^p([0,T])$. 
Assuming $s\in [0, T-2r_0]$, an application of 
the triangle inequality gives
\begin{align*}
	\abs{\mathcal{T}_{r_0}(s)} & \leq 
	C\left(\int_0^T \abs{v(t)- v(s)}^pJ_{r_0}^+(t-s)\,dt\right)^{1/p} \\
	& \leq C\left(\int_0^T \abs{v_1^n(t)-v_1^n(s)}^pJ_{r_0}^+(t-s)\,dt\right)^{1/p} \\
	& \qquad\qquad 
	+ (\abs{v_2^n}^p \star J_{r_0}(s))^{1/p} + \abs{v_2^n(s)}.
\end{align*}
Sending $r_0 \downarrow 0$, it follows that 
$\lim_{r_0 \downarrow 0}\abs{\mathcal{T}_{r_0}(s)} 
\leq 2\abs{v_2^n(s)}$ for $ds$-a.a.~$s \in [0,T)$. Since $v_2^n \rightarrow 0$ 
in $L^p([0,T])$, it has a subsequence that 
converges $ds$-a.e., and this concludes the proof.
\end{proof}

\subsection{Weighted $L^p$ spaces.}
First we make some elementary observations regarding 
functions in $\mathfrak{N}$ (see Section~\ref{sec:Entropy_Formulation} 
for the definition of $\mathfrak{N}$).

\begin{lemma}\label{lemma:PhiProp}
Suppose $\phi \in \mathfrak{N}$ and $0 < p < \infty$. 
Then, for $x,z\in \R^d$,
$$
\abs{\phi^{1/p}(x+z)-\phi^{1/p}(x)} 
\leq w_{p,\phi}(\abs{z})\phi^{1/p}(x), 
$$
where 
$$
w_{p,\phi}(r) = \frac{C_\phi}{p}r\left(1 + \frac{C_\phi}{p}r e^{C_\phi r/p}\right),
$$
which is defined for all $r \geq 0$. As a consequence it follows that if 
$\phi(x_0) = 0$ for some $x_0 \in \R^d$, then $\phi \equiv 0$ (and by 
definition $\phi \notin \mathfrak{N}$).
\end{lemma}

\begin{proof}
Set $g(\lambda) = \phi^{1/p}(x + \lambda z)$. Then 
$$
g'(\lambda) = \frac{1}{p}\phi^{1/p-1}(x + \lambda z)
(\nabla \phi(x + \lambda z) \cdot z).
$$
Since $\phi \in \mathfrak{N}$, it follows that 
$\abs{g'(\lambda)} \leq \frac{C_\phi}{p}g(\lambda)\abs{z}$. Hence
$$
g(\lambda) \leq g(0) + \frac{C_\phi}{p}
\abs{z}\int_0^\lambda g(\xi)\,d\xi.
$$
By Gr\"onwall's inequality,
$$
g(\lambda) \leq g(0)\left(1 + \frac{C_\phi}{p}\abs{z}
\lambda e^{C_\phi\abs{z}\lambda/p}\right).
$$
Hence,
$$
\abs{g(1)-g(0)} \leq \frac{C_\phi}{p}\abs{z}g(0)
\left(1 + \frac{C_\phi}{p}\abs{z} e^{C_\phi\abs{z}/p}\right).
$$
This concludes the proof.
\end{proof}

Next, we consider an adaption of Young's inequality for convolutions. 
\begin{proposition}\label{prop:YoungsForLocalized}
Fix $\phi \in \mathfrak{N}$. 
Suppose $f \in C_c(\R^d)$, and 
$g \in L^p(\R^d,\phi)$ for some finite $p\ge 1$. Then
$$
\norm{f \star g}_{L^p(\R^d,\phi)} \leq 
\left(\int_{\R^d}\abs{f(x)}(1 + w_{p,\phi}(\abs{x}))\,dx \right)
\norm{g}_{L^p(\R^d,\phi)}.
$$
where $w_{p,\phi}$ is defined in Lemma \ref{lemma:PhiProp}.
\end{proposition}

\begin{proof}
First observe that 
\begin{align*}
	\norm{f \star g}_{L^p(\R^d,\phi)}^p 
	&= \int_{\R^d}\abs{\int_{\R^d} f(x-y)g(y)\,dy}^p \phi(x)\,dx \\
	&\leq \int_{\R^d}\left(\int_{\R^d} \abs{f(x-y)}
	\left(\frac{\phi(x)}{\phi(y)}\right)^{1/p}\abs{g(y)}\phi^{1/p}(y)\,dy\right)^p \,dx.
\end{align*}
By Lemma~\ref{lemma:PhiProp}(iii),
$$
\left(\frac{\phi(x)}{\phi(y)}\right)^{1/p} \leq 
\frac{1}{\phi^{1/p}(y)}\left(\phi^{1/p}(y) + \abs{\phi^{1/p}(x)-\phi^{1/p}(y)}\right) 
\leq \left(1 + w_{p,\phi}(\abs{x-y})\right).
$$
Set
$$
\zeta(x) := \abs{f(x)}(1 + w_{p,\phi}(\abs{x})),
\qquad 
\xi(x) := \abs{g(x)}\phi^{1/p}(x).
$$
Then, by Young's inequality for convolutions,
$$
\norm{f \star g}_{L^p(\R^d,\phi)} \leq \norm{\zeta \star \xi}_{L^p(\R^d)} 
\leq \norm{\zeta}_{L^1(\R^d)}\norm{\xi}_{L^p(\R^d)}.
$$
\end{proof}

\begin{lemma}\label{lemma:ContMollWeightedNorm}
Fix $\phi \in \mathfrak{N}$, and let $w_{p,\phi}$ be defined in 
Lemma \ref{lemma:PhiProp}. Let $J$ be a mollifier as defined 
in Section~\ref{sec:Entropy_Formulation} and 
take $\phi_\delta = \phi \star J_\delta$ for $\delta > 0$. Then
\begin{itemize}
	\item[(i)] $\phi_\delta \in \mathfrak{N}$ with $C_{\phi_\delta} = C_\phi$.
	\item[(ii)] For any $u \in L^p(\R^d,\phi)$,
	$$
	\abs{\norm{u}_{p,\phi}^p-\norm{u}_{p,\phi_\delta}^p} 
	\leq w_{1,\phi}(\delta)\min\seq{\norm{u}_{p,\phi}^p,\norm{u}_{p,\phi_\delta}^p}.
	$$
	\item[(iii)] 
	\begin{displaymath}
	 \abs{\Delta \phi_\delta(x)} \leq \frac{1}{\delta} 
	 C_\phi\norm{\nabla J}_{L^1(\R^d)}(1 + w_{1,\phi}(\delta))^2\phi_\delta(x).
	\end{displaymath}
 \end{itemize}
\end{lemma}

\begin{proof}
Consider (i). Young's inequality for convolutions 
yields $\phi_\delta \in L^1(\R^d)$. Furthermore,
$$
\abs{\nabla(\phi \star J_\delta)(x)} 
= \abs{\int_{\R^d} J_\delta(y)\nabla\phi(x-y)\,dy} 
\leq C_\phi (\phi \star J_\delta)(x).
$$
Consider (ii). By Lemma \ref{lemma:PhiProp},
\begin{align*}
	\abs{\norm{u}_{p,\phi}^p -\norm{u}_{p,\phi_\delta}^p} 
	&= \abs{\int_{\R^d}\int_{\R^d} \abs{u(x)}^p(\phi(x-z)-\phi(x))
	J_\delta(z)\,dzdx} \\
	&\leq \min\seq{\norm{u}_{p,\phi}^p,\norm{u}_{p,\phi_\delta}^p}
	\int_{\R^d}w_{1,\phi}(\abs{z})J_\delta(z)\,dz.
\end{align*}
This proves (ii). Consider(iii). Integration by parts yields 
\begin{displaymath}
 \abs{\Delta (\phi_\delta)(x)}  = 
	\abs{\int_{\R^d}\nabla J_\delta(x-y) \cdot \nabla \phi (y)\,dy}
	\leq C_\phi\int_{\R^d}\abs{\nabla J_\delta(x-y)}\phi (y)\,dy.
\end{displaymath}
By Lemma~\ref{lemma:PhiProp},
\begin{align*}
\int_{\R^d}\abs{\nabla J_\delta(x-y)}\phi (y)\,dy 
  &\leq \left(\int_{\R^d}\abs{\nabla J_\delta(x-y)}(1 + w_{1,\phi}(\abs{x-y}))\,dy\right)\phi(x) \\
  &\leq \frac{1}{\delta}\norm{\nabla J}_{L^1(\R^d)}(1 + w_{1,\phi}(\delta))\phi(x).
\end{align*}
Again, by Lemma~\ref{lemma:PhiProp}
\begin{displaymath}
 \phi(x) \leq \abs{\phi(x)-\phi_\delta(x)} + \phi_\delta(x) 
 \leq (1 + w_{1,\phi}(\delta))\phi_\delta(x).
\end{displaymath}
The result follows.
\end{proof}

\begin{lemma}\label{lemma:NSmCompApprox}
Let $\phi \in \mathfrak{N}$. Then there exists 
$\seq{\phi_R}_{R > 1} \subset C^\infty_c(\R^d)$ such that 
\begin{itemize}
	\item[(i)] $\phi_R \rightarrow \phi$ and $\nabla \phi_R \rightarrow \nabla \phi$ 
	pointwise in $\R^d$ as $R \rightarrow \infty$,
	\item[(ii)] $\exists$ a constant $C$ independent of $R > 1$ such that 
	$$
	\max \seq{\norm{\phi_R}_{\infty,\phi^{-1}},
	\norm{\nabla \phi_R}_{\infty,\phi^{-1}}} \leq C.
	$$
\end{itemize}
\end{lemma}

\begin{proof}
Modulo a mollification step, we may assume $\phi \in C^\infty$. 
Let $\zeta \in C_c^\infty(\R^d)$ satisfy $0 \leq \zeta \leq 1$, $\zeta(0) =1$. 
Let $\phi_R(x) := \phi(x)\zeta(R^{-1}x)$. Then
$$
\nabla \phi_R(x) = \nabla \phi(x)
\zeta(R^{-1}x) + R^{-1}\phi(x)\nabla \zeta(R^{-1}x).
$$
Hence (i) follows. Clearly, $\norm{\phi_R}_{\infty,\phi^{-1}} 
= \sup_x \seq{\abs{\phi_R(x)}\phi^{-1}(x)} = \norm{\zeta}_\infty$. 
Furthermore,
$$
\abs{\nabla \phi_R(x)} \leq 
\left(C_\phi\zeta(R^{-1}x) + R^{-1}\abs{\nabla \zeta(R^{-1}x)}\right)\phi(x).
$$
Hence, $\norm{\nabla \phi_R}_{\infty,\phi^{-1}} 
\leq C_\phi + R^{-1}\norm{\nabla \zeta}_\infty$.
\end{proof}

\subsection{A version of It\^{o}'s formula} 
Here we establish the particular 
anticipating It\^{o} formula applied in 
the proof of Theorem \ref{theorem:ExistenceOfSolution}.
\begin{theorem}\label{theorem:AntIto}
Let 
$$
X(t) = X_0 + \int_0^t\int_Z u(s,z)\,W(dz,ds) 
+ \int_0^t v(s)\,ds,
$$
where $u:[0,T] \times Z \times \Omega \rightarrow \R$ and 
$v:[0,T] \times \Omega \rightarrow \R$ are jointly measurable 
and $\seq{\F_t}$-adapted processes, satisfying 
\begin{equation}\label{eq:AssumptionItoProcess}
	\E{\bigg(\int_0^T\int_Z u^2(s,z)\,d\mu(z)ds\bigg)^2} < \infty, 
	\qquad \E{\int_0^T v^2(s)\,ds} < \infty.
\end{equation}
Let $F:\R^2 \times [0,T] \rightarrow \R$ be twice continuously 
differentiable. Suppose there exists a 
constant $C > 0$ such that for all 
$(\zeta,\lambda,t) \in \R^2 \times [0,T]$,
\begin{align*}
	& \abs{F(\zeta,\lambda,t)},\abs{\partial_3F(\zeta,\lambda,t)} 
	\leq C(1 + \abs{\zeta} + \abs{\lambda}), \\
	& \abs{\partial_1F(\zeta,\lambda,t)}, 
	\abs{\partial_{1,2}^2F(\zeta,\lambda,t)}, 
	\abs{\partial_1^2F(\zeta,\lambda,t)} \leq C.
\end{align*}
Let $V \in \Sm$. Then $s \mapsto \partial_1F(X(s),V,s)u(s)$ is 
Skorohod integrable, and
\begin{align*}
	F(X(t),V,t) &= F(X_0,V,0) \\
	&\qquad +\int_0^t\partial_3F(X(s),V,s)\,ds \\
	&\qquad +\int_0^t\int_Z\partial_1F(X(s),V,s)u(s,z)\,W(dz,ds) \\	
	&\qquad +\int_0^t\partial_1F(X(s),V,s)v(s)\,ds \\
	&\qquad +\int_0^t\int_Z \partial_{1,2}^2F(X(s),V,s)D_{s,z}Vu(s,z)\,d\mu(z)ds \\
	&\qquad +\frac{1}{2}\int_0^t\int_Z \partial_1^2F(X(s),V,s)u^2(s,z)\,d\mu(z)ds,
	\quad \text{$dP$-almost surely}.
\end{align*}
\end{theorem}

\begin{proof}
The proof follows \cite[Theorem~3.2.2 and Proposition~1.2.5]{NualartMalliavinCalc2006}. 
We give an outline and some details where there are 
considerable differences. Furthermore, we assume that $F$ is 
independent of $t$ as this is a standard modification. 

Set $t_i^n = \frac{it}{2^n}$, $0 \leq i \leq 2^n$. 
By  Taylor's formula,
\begin{multline*}
	F(X(t),V) = F(X_0,V) 
	+ \underbrace{\sum_{i=0}^{2^n-1} 
	\partial_1F(X(t_i^n),V)(X(t_{i+1}^n)-X(t_i^n))}_{\mathcal{T}^1_n} \\
	+ \underbrace{\frac{1}{2}\sum_{i=0}^{2^n-1}
	\partial_1^2F(\overline{X}_i,V)(X(t_{i+1}^n)-X(t_i^n))^2}_{\mathcal{T}^2_n},
\end{multline*}
where $\overline{X}_i$ denotes a random intermediate point 
between $X(t_i^n)$ and $X(t_{i+1}^n)$.  
As in the proof of \cite[Proposition~1.2.5]{NualartMalliavinCalc2006},
$$
\mathcal{T}^2_n \rightarrow \frac{1}{2}\int_0^t\int_Z 
\partial_1^2F(X(s),V)u^2(s,z)\,d\mu(z)ds, 
\quad
\text{in $L^1(\Omega)$ as $n \rightarrow \infty$.}
$$
Note that 
\begin{multline*}
	\mathcal{T}^1_n = \underbrace{\sum_{i=0}^{2^n-1} 
	\partial_1F(X(t_i^n),V)\int_{t_i^n}^{t_{i+1}^n}
	\int_Z u(s,z)\,W(dz,ds)}_{\mathcal{T}^{1,1}_n} \\
	+  \underbrace{\sum_{i=0}^{2^n-1} \partial_1F(X(t_i^n),V)
	\int_{t_i^n}^{t_{i+1}^n}v(s)\,ds}_{\mathcal{T}^{1,2}_n}. 
\end{multline*}
Clearly,
$$
\mathcal{T}^{1,2}_n \rightarrow \int_0^t\partial_1F(X(s),V)v(s)\,ds, 
\quad
\text{in $L^1(\Omega)$ as $n \rightarrow \infty$.} 
$$

Consider $\mathcal{T}^{1,1}_n$. 
By \cite[Proposition~1.3.5]{NualartMalliavinCalc2006}, 
$s \mapsto \partial_1F(X(t_i^n),V)u(s)$ 
is Skorohod integrable on $[t_i^n,t_{i+1}^n]$ and
\begin{multline*}
	\mathcal{T}^{1,1}_n = 
	\underbrace{\sum_{i=0}^{2^n-1}\int_{t_i^n}^{t_{i+1}^n}
	\int_Z \partial_1F(X(t_i^n),V)u(s,z)\,W(dz,ds)}_{\mathcal{T}^{1,1,1}_n} \\
	+ \underbrace{\sum_{i=0}^{2^n-1} \int_{t_i^n}^{t_{i+1}^n}
	\int_Z\partial_{1,2}^2F(X(t_i^n),V)
	D_{s,z}Vu(s,z)\,d\mu(z)ds}_{\mathcal{T}^{1,1,2}_n}.
\end{multline*}
As before
$$
\mathcal{T}^{1,1,2}_n 
\rightarrow \int_0^t\int_Z \partial_{1,2}^2F(X(s),V)D_{s,z}Vu(s,z)\,d\mu(z)ds,
\quad
\text{in $L^1(\Omega)$ as $n \rightarrow \infty$.} 
$$
Consider $\mathcal{T}^{1,1,1}_n$. Let
$$
\zeta_n(s,z) = \sum_{i=0}^{2^n-1}\partial_{1,2}^2F(X(t_i^n),V)
\car{[t_i^n,t_{i+1}^n)}(s)D_{s,z}Vu(s,z),
$$
and note that $\zeta_n$ is Skorohod integrable on $[0,t]$. 
We need to show the following:
\begin{itemize}
	\item[(i)] There exists $\zeta \in L^2(\Omega;H)$ 
	such that $\zeta_n \rightarrow \zeta$ in $L^2(\Omega;H)$.
	\item[(ii)] There exists a $G \in L^2(\Omega)$ 
	such that for each $U \in \Sm$
	$$
	\E{\int_0^t\int_Z \zeta_n(s,z)W(dz,ds)U} \rightarrow \E{GU}.
	$$
 \end{itemize}
Then we may conclude by \cite[Proposition~1.3.6]{NualartMalliavinCalc2006} 
that $\zeta$ is Skorohod integrable and $\int_0^t \zeta(s)\,dW(s) = G$. 
The result then follows. Consider (i). Let 
$$
\zeta(s,z) = \partial_{1,2}^2F(X(s),V)D_{s,z}Vu(s,z).
$$
Then
$$
\E{\int_0^t\int_Z \abs{\zeta_n(s)-\zeta(s)}^2\,d\mu(z)ds}
\leq \E{H_n\int_0^t\int_Z \abs{D_{s,z}Vu(s,z)}^2\,d\mu(z)ds},
$$
where
$$
H_n = \sup_{\abs{t_i^n-s} \leq t2^{-n}}
\seq{\abs{\partial_{1,2}^2F(X(t_i^n),V)-\partial_{1,2}^2F(X(s),V)}^2}.
$$
Hence, (i) follows by the dominated convergence theorem. Consider (ii). 
The existence of a random variable $G$ follows by the convergence of the other terms. 
This also yields the weak convergence. It remains to check that 
$G \in L^2(\Omega)$. This is a consequence 
of assumptions \eqref{eq:AssumptionItoProcess}. 
\end{proof}

\subsection{The Lebesgue-Bochner space}\label{sec:LebBoch}
Let $(X,\mathscr{A},\mu)$ be a $\sigma$-finite measure space 
and $E$ a Banach space. In the previous sections 
$X = [0,T] \times \Omega$, $\mu = dt \otimes dP$, $E$ is 
typically $L^p(\R^d,\phi)$ for some $1 \leq p < \infty$, and $\mathscr{A}$ is the 
predictable $\sigma$-algebra $\Pred$. A function $u:X \rightarrow E$ 
is \emph{strongly $\mu$-measurable} if there exists a sequence 
of $\mu$-simple functions $\seq{u_n}_{n \geq 1}$ such that 
$u_n \rightarrow u$ $\mu$-almost everywhere. 
By a $\mu$-\emph{simple function} $s:X \rightarrow E$ 
we mean a function of the form
\begin{displaymath}
 s(\zeta) =  \sum_{k = 1}^N \car{A_k}(\zeta)x_k, \qquad \zeta \in X,
\end{displaymath}
where $x_k \in E$ and $A_k \in \mathscr{A}$ satisfy 
$\mu(A_k)< \infty$ for all $1 \leq k \leq N$. 
The Lebesgue-Bochner space $L^p(X,\mathscr{A},\mu;E)$ is the 
linear space of $\mu$-equivalence classes of strongly 
measurable functions $u:X \rightarrow E$ satisfying
\begin{displaymath}
 \int_X \norm{u(\xi)}_E^p \,d\mu(\xi) < \infty.
\end{displaymath}
A map $u:X \rightarrow E$ is \emph{weakly $\mu$-measurable} if 
the map $\xi \mapsto \inner{u(\xi)}{\test^*}$ has a $\mu$-version which 
is $\mathscr{A}$-measurable for each $\test^*$ in the dual space $E^*$.
By the Pettis measurability theorem \cite[Theorem~1.11]{Neerven2007}, strong 
$\mu$-measurability is equivalent to weak 
$\mu$-measurability, whenever $E$ is separable.

For $u \in L^1(X,\mathscr{A},\mu;L^1(\R^d,\phi))$, it is convenient 
to know that $\zeta \mapsto u(\zeta)(x)$ has a $\mu$-version which 
is $\mathscr{A}$-measurable for almost all $x$. In fact this is crucial to 
the manipulations performed in the previous sections. 
The following result verifies that this is indeed the case.

\begin{lemma}\label{lemma:LebBochRepr}
 Let $(X,\mathscr{A},\mu)$ be a $\sigma$-finite measure space 
 and $\phi \in \mathfrak{N}$. Let
 \begin{displaymath}
  \Psi:L^1(X \times \R^d,\mathscr{A} \otimes \Borel{\R^d},d\mu \otimes d\phi) 
  \rightarrow L^1(X,\mathscr{A},\mu;L^1(\R^d,\phi))
 \end{displaymath}
 be defined by $\Psi(u)(\xi)=u(\xi,\cdot)$. Then $\Psi$ is an isometric isomorphism. 
\end{lemma}

\begin{remark}
 The measure space $(X \times \R^d,\mathscr{A} \otimes \Borel{\R^d},d\mu \otimes d\phi)$ 
 is not necessarily complete. Strictly speaking we should rather consider its completion. 
 What this ensures is that every representative is measurable 
 with respect to the complete $\sigma$-algebra. A remedy is to define 
 $L^1(X \times \R^d,\mathscr{A} \otimes \Borel{\R^d},d\mu \otimes d\phi)$ by 
 asking that any element $u$ has a $d\mu \otimes d\phi$-version $\tilde{u}$ 
 which is $\mathscr{A} \otimes \Borel{\R^d}$-measurable. Now, $\tilde{u}(\cdot,x)$ is 
 $\mathscr{A}$ measurable, and so for $d\phi$-almost all $x$, $u(\cdot,x)$ 
 has a $\mu$-version which is $\mathscr{A}$-measurable. 
\end{remark}

\begin{proof}
 Let us first check that $\Psi(u) \in L^1(X;L^1(\R^d,\phi))$. 
 By the Pettis measurability theorem \cite[Theorem~1.11]{Neerven2007}, strong 
 $\mu$-measurability follows due to the separability of $L^1(\R^d,\phi)$ if $\Psi(u)$ is 
 weakly $\mu$-measurable. That is, for any $\test \in L^\infty(\R^d,\phi)$, the map
 \begin{displaymath}
  \xi \mapsto  \int_{\R^d} \test(x)u(\xi,x)\phi(x)\,dx,
 \end{displaymath}
 has a $\mu$-version which is $\mathscr{A}$ measurable. This is a consequence 
 of Fubini's theorem \cite[Proposition~5.2.2]{Cohn2013}. The fact that $\Psi$ is an 
 isometry is obvious. It remains to prove that $\Psi$ is surjective. 
 Let $v \in L^1(X;L^1(\R^d,\phi))$. By definition there exists 
 a sequence $\seq{v_n}_{n \geq 1}$ of simple 
 functions such that $v_n \rightarrow v$ $\mu$-almost everywhere. Set
 \begin{displaymath}
  v_n(\xi) = \sum_{k=1}^{N_n} \car{A_{k,n}}(\xi)f_{k,n}, \quad u_n(\xi,x) = v_n(\xi)(x),
 \end{displaymath}
 where $A_{k,n} \in \mathscr{A}$, $f_{k,n} \in L^1(\R^d,\phi)$. Note that 
 $u_n$ is $\mathscr{A} \otimes \Borel{\R^d}$ measurable, and $\Psi(u_n) = v_n$. 
 By the Lebesgue dominated convergence theorem, $v_n \rightarrow v$ 
 in $L^1(X,L^1(\R^d,\phi))$ \cite[Proposition~1.16]{Neerven2007}. 
 By the isometry property, $\seq{u_n}_{n \geq 1}$ is Cauchy, and so by 
 completeness there exists $u$ such that
 \begin{displaymath}
  u_n \rightarrow u \mbox{ in } L^1(X \times \R^d,\mathscr{A} 
  \otimes \Borel{\R^d},d\mu \otimes d\phi).
 \end{displaymath}
 Since
 \begin{align*}
  \int_X \norm{v-\Psi(u)}_{1,\phi}\,d\mu 
      &= \lim_{n \rightarrow \infty}\int_X \norm{v_n-\Psi(u)}_{1,\phi}\,d\mu \\
      &= \lim_{n \rightarrow \infty}\iint_{X \times \R^d}
      \abs{u_n(\xi,x)-u(\xi,x)}\, d\mu \otimes d\phi(\xi,x) = 0,
 \end{align*}
 it follows that $\Psi(u) = v$.
\end{proof}

\subsection{Young measures}\label{sec:YoungMeasures}
The purpose of this subsection is to provide a 
reference for some results concerning Young measures 
and their application as generalized limits. 
Let $(X,\mathscr{A},\mu)$ be a $\sigma$-finite measure space, 
and $\mathscr{P}(\R)$ denote the set of probability measures on $\R$. In the 
previous sections $X$ is typically $\Pi_T \times \Omega$. 
A \emph{Young measure} from $X$ into $\R$ is a function 
$\nu:X \rightarrow \mathscr{P}(\R)$ such that $x \mapsto \nu_x(B)$ is 
$\mathscr{A}$-measurable for every Borel measurable set $B \subset \R$. 
We denote by $\Young{X,\mathscr{A},\mu;\R}$, or simply $\Young{X;\R}$ if the 
measure space is understood, the set of all Young measures from $X$ into $\R$. 
The following theorem is proved in \cite[Theorem~6.2]{Pedregal1997} in the 
case that $X \subset \R^n$ and $\mu$ is the Lebesgue measure: 

\begin{theorem}\label{theorem:YoungMeasureLimitOfComposedFunc}
Let $(X,\mathscr{A},\mu)$ be a $\sigma$-finite measure space. 
Let $\zeta:[0,\infty) \rightarrow [0,\infty]$ be a continuous, nondecreasing 
function satisfying $\lim_{\xi \rightarrow \infty}\zeta(\xi) = \infty$ and 
$\seq{u^n}_{n \geq 1}$ a sequence of measurable functions such that 
$$
\sup_{n} \int_X \zeta(\abs{u^n})d\mu(x)  < \infty.
$$
Then there exist a subsequence $\seq{u^{n_j}}_{j \geq 1}$ 
and $\nu \in \Young{X,\mathscr{A},\mu;\R}$ such that for 
any Carath\'eodory function $\psi:\R \times X \rightarrow \R$ 
with $\psi(u^{n_j}(\cdot),\cdot) \rightharpoonup \overline{\psi}$ (weakly) 
in $L^1(X)$, we have 
$$
\overline{\psi}(x) = \int_{\R} \psi(\xi,x)\,d\nu_x(\xi).
$$
\end{theorem}

Recall that $\psi:\R \times X \rightarrow \R$ is a Carath\'eodory function 
if $\psi(\cdot,x):\R \rightarrow \R$ is continuous for all $x \in X$ 
and $\psi(u,\cdot):X \rightarrow \R$ 
is measurable for all $u \in \R$. The proof is based on the embedding 
of $\Young{X;\R}$ into $L^\infty_{w*}(X,\Rad{\R})$. 
Here $\Rad{\R}$ denotes the space of Radon 
measures on $\R$ and $L^\infty_{w*}(X,\Rad{\R})$ denotes the 
space of weak$*$-measurable bounded maps $\nu:X \rightarrow \Rad{\R}$. 
The crucial observation is that $(L^1(X,C_0(\R)))^*$ is isometrically isomorphic 
to $L^\infty_{w*}(X,\Rad{\R})$ also in the case that 
$(X,\mathscr{A},\mu)$ is an abstract $\sigma$-finite measure space. 
It is relatively straightforward to go through the proof and 
extend to this more general case \cite[Theorem~2.11]{Malek1996}. 
Note however that the use of weighted $L^p$ spaces 
allows us to stick with the version 
for finite measure spaces.

\subsection{Weak compactness in $L^1$.}
To apply Theorem~\ref{theorem:YoungMeasureLimitOfComposedFunc} one
must first be able to extract from 
$\seq{\psi(u^n(\cdot),\cdot)}_{n \geq 1}$ a weakly 
convergent subsequence in $L^1(X)$. 
The key result is the Dunford-Pettis Theorem. 

\begin{definition}\label{def:equiintegrability}
Let $\mathcal{K} \subset L^1(X,\mathscr{A},\mu)$.
\begin{itemize}
	\item[(i)] $\mathcal{K}$ is \emph{uniformly integrable} 
	if for any $\varepsilon > 0$ there 
	exists $c_0(\varepsilon)$ such that 
	$$
	\sup_{f \in \mathcal{K}} \int_{\abs{f} \geq c} \abs{f} \,d\mu 
	< \varepsilon, 
	\quad \mbox{whenever $c > c_0(\varepsilon)$.}
	$$
	\item[(ii)] $\mathcal{K}$ has \emph{uniform tail} if for 
	any $\varepsilon > 0$ there exists 
	$E \in \mathscr{A}$ with $\mu(E) < \infty$ such that 
	$$
	\sup_{f \in \mathcal{K}}\int_{X \setminus E} \abs{f} \,d\mu 
	<\varepsilon.
	$$
\end{itemize}
If $\mathcal{K}$ satisfies both (i) and (ii) it is 
said to be \emph{equiintegrable}.
\end{definition}

\begin{remark}
Note that (ii) is void when $\mu$ is finite. As a consequence 
uniform integrability and equiintegrability are equivalent for finite measure spaces.
\end{remark}

\begin{theorem}[Dunford-Pettis]\label{theorem:DunfordPettis}
Let $(X,\mathscr{A},\mu)$ be a $\sigma$-finite measure space. 
A subset $\mathcal{K}$ of $L^1(X)$ is 
relatively weakly sequentially compact 
if and only if it is equiintegrable.
\end{theorem}

By the Eberlain-\v{S}mulian theorem \cite{Whitley1967}, in the 
weak topology of a Banach space, relative weak compactness 
is equivalent with relative sequentially weak compactness. 
There are a couple of well known reformulations of uniform integrability.

\begin{lemma}\label{lemma:UniformIntCriteria}
Suppose $\mathcal{K} \subset L^1(X)$ is bounded. 
Then $\mathcal{K}$ is uniformly integrable if and only if:
\begin{itemize}
	\item[(i)] For any $\varepsilon > 0$ there exists 
	$\delta(\varepsilon) > 0$ such that
 	$$
	\sup_{f \in \mathcal{K}}\int_E \abs{f} \,d\mu < 
	\varepsilon,  \quad 
	\mbox{whenever $\mu(E) <\delta(\varepsilon)$.}
	$$
	\item[(ii)] There is an increasing function 
	$\Psi:[0,\infty) \rightarrow [0,\infty)$ 
	such that $\Psi(\zeta)/\zeta \rightarrow \infty$ 
	as $\zeta \rightarrow \infty$ and
	$$
	\sup_{f \in \mathcal{K}}\int_X \Psi(\abs{f(x)}) \,d\mu(x) < \infty.
	$$
 \end{itemize}
\end{lemma}

\begin{remark}\label{remark:DominatedFamilyUniformInt}
Suppose there exists $g \in L^1(X)$ such 
that $\abs{f} \leq g$ for all $f \in \mathcal{K}$. Then
$$
\sup_{f \in \mathcal{K}}\int_E \abs{f} \,d\mu \leq \int_E g \,d\mu, 
\qquad \forall E \in \mathscr{A}.
$$ 
Since $\seq{g} \subset L^1(X)$ is uniformly 
integrable, it follows by Lemma~\ref{lemma:UniformIntCriteria}(i) 
that $\mathcal{K}$ is uniformly integrable.
\end{remark}

\end{document}